\documentclass[numbers=enddot,12pt,final,onecolumn,notitlepage]{scrartcl}%
\usepackage[headsepline,footsepline,manualmark]{scrlayer-scrpage}
\usepackage[all,cmtip]{xy}
\usepackage{amssymb}
\usepackage{amsmath}
\usepackage{amsthm,amscd}
\usepackage{framed}
\usepackage{comment}
\usepackage{color}
\usepackage[hyperindex,breaklinks]{hyperref}
\usepackage[sc]{mathpazo}
\usepackage[T1]{fontenc}
\usepackage{needspace}
\usepackage{tabls}
\providecommand{\U}[1]{\protect\rule{.1in}{.1in}}
\newcounter{exer}
\newcounter{exera}
\theoremstyle{definition}
\newtheorem{theo}{Theorem}[section]
\newenvironment{theorem}[1][]
{\begin{theo}[#1]\begin{leftbar}}
{\end{leftbar}\end{theo}}
\newtheorem{lem}[theo]{Lemma}
\newenvironment{lemma}[1][]
{\begin{lem}[#1]\begin{leftbar}}
{\end{leftbar}\end{lem}}
\newtheorem{prop}[theo]{Proposition}
\newenvironment{proposition}[1][]
{\begin{prop}[#1]\begin{leftbar}}
{\end{leftbar}\end{prop}}
\newtheorem{defi}[theo]{Definition}
\newenvironment{definition}[1][]
{\begin{defi}[#1]\begin{leftbar}}
{\end{leftbar}\end{defi}}
\newtheorem{remk}[theo]{Remark}
\newenvironment{remark}[1][]
{\begin{remk}[#1]\begin{leftbar}}
{\end{leftbar}\end{remk}}
\newtheorem{coro}[theo]{Corollary}
\newenvironment{corollary}[1][]
{\begin{coro}[#1]\begin{leftbar}}
{\end{leftbar}\end{coro}}
\newtheorem{conv}[theo]{Convention}

\newtheorem{quest}[theo]{Question}

\newtheorem{warn}[theo]{Warning}

\newtheorem{conj}[theo]{Conjecture}

\newtheorem{exam}[theo]{Example}
\newenvironment{example}[1][]
{\begin{exam}[#1]\begin{leftbar}}
{\end{leftbar}\end{exam}}
\newtheorem{exmp}[exer]{Exercise}

\newtheorem{exetwo}[exera]{Additional exercise}

\newenvironment{statement}{\begin{quote}}{\end{quote}}
\let\sumnonlimits\sum
\let\prodnonlimits\prod
\let\cupnonlimits\bigcup
\let\capnonlimits\bigcap
\renewcommand{\sum}{\sumnonlimits\limits}
\renewcommand{\prod}{\prodnonlimits\limits}
\renewcommand{\bigcup}{\cupnonlimits\limits}
\renewcommand{\bigcap}{\capnonlimits\limits}
\newenvironment{noncompile}{}{}
\excludecomment{verlong}
\includecomment{vershort}
\excludecomment{noncompile}

\newcommand{\Fq}{\ensuremath{\mathbb{F}_q}\xspace}

\newcommand{\N}{\ensuremath{\mathbb{N}}\xspace}

\newcommand{\myvec}[1]{\ensuremath{\boldsymbol{#1}}\xspace}

\newcommand{\ccc}{\myvec{c}}

\newcommand{\bbb}{\myvec{b}}
\newcommand{\eee}{\myvec{e}}

\usepackage{aurical,bbm} 
\newcommand{\pp}{\mathbbm{p}}
\newcommand{\mm}{\mathbbm{a}}

\renewcommand{\P}{\myvec{P}}

\newcommand{\K}{\mathbb{K}}

\newcommand{\Sp}{{ ({\K}^n)^{{\Z}}}}

\newcommand{\Z}{\mathbb{Z}}

\newcommand{\zetaki}{\Z/k_i\Z}
\newcommand{\zetapk}{\Z/p^k\Z}
\newcommand{\zetapku}{\Z/p^{k_1}\Z}
\newcommand{\zetapki}{\Z/p^{k_i}\Z}
\newcommand{\zetapkn}{\Z/p^{k_n}\Z}

\newcommand{\LP}{\ensuremath{\K[X,X^{-1}]}\xspace}
\newcommand{\LPG}{\ensuremath{G[X,X^{-1}]}\xspace}
\newcommand{\LPGt}{\ensuremath{\hat{G}[X,X^{-1}]}\xspace}

\newcommand{\LPKuno}{\ensuremath{\Z/p^{k_1}\Z\left[X,X^{-1}\right]}\xspace}

\newcommand{\emb}{\psi}
\newcommand{\Emb}{\Psi}
\usepackage{xspace, comment}

\newcommand{\ie}{i.e.\@\xspace}

\newcommand{\modulo}[2]{\left[ #1\right]_{#2}}

\newcommand{\z}{\mathbb{Z}}
\newcommand{\n}{\mathbb{N}}
\newcommand{\az}{S^{\mathbb{Z}}}

\newcommand\arxiv[1]{\href{http://www.arxiv.org/abs/#1}{\texttt{arXiv:#1}}}

\ihead{Integrality of matrices and cellular automata dynamics}
\ohead{page \thepage}
\begin{document}

\title{Integrality of matrices, finiteness of matrix semigroups, and dynamics of linear and additive cellular automata}
\author{Alberto Dennunzio%
\thanks{Dipartimento di Informatica, Sistemistica e Comunicazione,
  Universit\`a degli Studi di Milano-Bicocca,
  Viale Sarca 336/14, 20126 Milano, Italy;
  \texttt{\href{mailto:alberto.dennunzio@unimib.it}{alberto.dennunzio@unimib.it}}}
\and Enrico Formenti%
\thanks{Universit\'e C\^ote d'Azur, CNRS, I3S, France;
  \texttt{\href{mailto:enrico.formenti@unice.fr}{enrico.formenti@univ-cotedazur.fr}}}
\and Darij Grinberg%
\thanks{Drexel University, 15 S 33rd Street, Philadelphia PA, 19104, USA;
Mathematisches Forschungsinstitut Oberwolfach, Schwarzwaldstr. 9-11, 77709 Oberwolfach-Walke, Germany;
Institut Mittag-Leffler, Aurav\"agen 17, SE-182 60 Djursholm, Sweden;
  \texttt{\href{mailto:darijgrinberg@gmail.com}{darijgrinberg@gmail.com}}}
\and Luciano Margara%
\thanks{Department of Computer Science and Engineering, University of Bologna, Campus of Cesena, Via dell'universit\`a 50, Cesena, Italy;
  \texttt{\href{mailto:luciano.margara@unibo.it}{luciano.margara@unibo.it}}}
}
\date{
\today
}

\maketitle

\begin{abstract}

\textbf{Abstract.}
Let $\mathbb{K}$ be a finite commutative ring,
and let $\mathbb{L}$ be a commutative $\mathbb{K}$-algebra.
Let $A$ and $B$ be two $n \times n$-matrices over $\mathbb{L}$
that have the same characteristic polynomial.
The main result of this paper (Thm.~\ref{thm.finpowmat.main}) states that
the set $\left\{ A^0,A^1,A^2,\ldots\right\}$ is
finite if and only if the set $\left\{ B^0,B^1,B^2,\ldots\right\}$ is finite.
We apply this result to the theory of discrete time dynamical systems.
Indeed, it gives a complete and easy-to-check characterization of sensitivity to initial conditions and equicontinuity for linear cellular automata over the alphabet $\mathbb{K}^n$ for $\mathbb{K} = \mathbb{Z}/m\Z$ (Theorem~\ref{froblca}), \ie, cellular automata in which the local rule is defined by $n\times n$-matrices with elements in $\mathbb{Z}/m\Z$.

To prove our main result, we rederive a classical integrality criterion for matrices (Thm.~\ref{thm.finpowmat.char-int} and Prop.~\ref{prop.finpowmat.char-int-conv}).
Namely, let $\mathbb{K}$ be any commutative ring (not necessarily finite), and let $\mathbb{L}$ be a commutative $\mathbb{K}$-algebra.
Consider any $n \times n$-matrix $A$ over $\mathbb{L}$.
Then, $A \in \mathbb{L}^{n \times n}$ is integral over $\mathbb{K}$ (that is, there exists a monic polynomial $f \in \mathbb{K}\left[t\right]$ satisfying $f\left(A\right) = 0$) if and only if all coefficients of the characteristic polynomial of $A$ are integral over $\mathbb{K}$.

Furthermore, we extend the decidability result concerning sensitivity and equicontinuity to the wider class of additive cellular automata over a finite abelian group. For such cellular automata, we also prove the decidability of injectivity, surjectivity, topological transitivity and all the properties (as, for instance, ergodicity) that are equivalent to the latter. The proof of each of  these decidabilty results exploits the corresponding decidability result for linear cellular automata.

\vspace{1.2pc}

\end{abstract}

\tableofcontents

\subsection*{Acknowledgments}

DG thanks the Mathematisches Forschungsinstitut Oberwolfach for
its hospitality during part of the writing process.
He also thanks the math.stackexchange user named
``punctured dusk'' for
\href{https://math.stackexchange.com/a/3357622/}{informing him}
of the appearance of Theorem~\ref{thm.finpowmat.char-int} and
Proposition~\ref{prop.finpowmat.char-int-conv} in
\cite{Bourba72}.

\section{\label{chp.finpowmat}On matrices with finitely many distinct powers}

\subsection{The main theorem}

We recall a standard definition from linear algebra:\footnote{%
Here and in the following, $\mathbb{N}$ denotes the set
$\left\{0,1,2,\ldots\right\}$.}

\begin{definition}
\label{def.finpowmat.char}Let $\mathbb{K}$ be a commutative ring. Let
$n\in\mathbb{N}$. Let $A$ be an $n\times n$-matrix over $\mathbb{K}$. Then,
the \textit{characteristic polynomial} $\chi_{A}$ of $A$ is defined to be the
polynomial $\det\left(  tI_{n}-A\right)  \in\mathbb{K}\left[  t\right]  $.
Here, $I_{n}$ stands for the $n\times n$ identity matrix, and $tI_{n}-A$ is
considered as an $n\times n$-matrix over the polynomial ring $\mathbb{K}%
\left[  t\right]  $.
\end{definition}

Our goal in this section is to prove the following theorem:

\begin{theorem}
\label{thm.finpowmat.main}Let $\mathbb{K}$ be a finite commutative ring. Let
$\mathbb{L}$ be a commutative $\mathbb{K}$-algebra. Let $n\in\mathbb{N}$. Let
$A$ and $B$ be two $n\times n$-matrices over $\mathbb{L}$ such that $\chi
_{A}=\chi_{B}$. Then, the set $\left\{  A^{0},A^{1},A^{2},\ldots\right\}  $ is
finite if and only if the set $\left\{  B^{0},B^{1},B^{2},\ldots\right\}  $ is finite.
\end{theorem}

\begin{example}
\label{exa.finpowmat.1}
\textbf{(a)} Let $\mathbb{K}=\mathbb{Z}/4$ and $\mathbb{L}=\left(
\mathbb{Z}/4\right)  \left[  x\right]  $ (a polynomial ring) and $n=2$ and
$A=\left(
\begin{array}
[c]{cc}%
1 & x\\
0 & 1
\end{array}
\right)  $ and $B=\left(
\begin{array}
[c]{cc}%
1 & 0\\
0 & 1
\end{array}
\right)  $. Then, $\chi_{A}=\left(  t-1\right)  ^{2}=\chi_{B}$. Hence, Theorem
\ref{thm.finpowmat.main} yields that the set $\left\{  A^{0},A^{1}%
,A^{2},\ldots\right\}  $ is finite if and only if the set $\left\{
B^{0},B^{1},B^{2},\ldots\right\}  $ is finite. And indeed, both of these sets
are finite: The former has $4$ elements, while the latter has $1$.

\textbf{(b)} Now, let $\mathbb{K}=\mathbb{Q}$ and $\mathbb{L}=\mathbb{Q}$ and
$n=2$ and $A=\left(
\begin{array}
[c]{cc}%
1 & 1\\
0 & 1
\end{array}
\right)  $ and $B=\left(
\begin{array}
[c]{cc}%
1 & 0\\
0 & 1
\end{array}
\right)  $. Then, $\chi_{A}=\left(  t-1\right)  ^{2}=\chi_{B}$. The ring
$\mathbb{K}$ is not finite, so Theorem \ref{thm.finpowmat.main} does not apply
here. And we see why: The set $\left\{  B^{0},B^{1},B^{2},\ldots\right\}  $ is
finite, but the set $\left\{  A^{0},A^{1},A^{2},\ldots\right\}  $ is not.
\end{example}

We shall eventually prove Theorem \ref{thm.finpowmat.main}, but first let us
briefly discuss what rings $\mathbb{L}$ it applies to:

\begin{proposition}
\label{prop.when-L-K}Let $\mathbb{L}$ be a commutative ring. Then, the
following two statements are equivalent:

\begin{itemize}
\item \textit{Statement }$\mathcal{K}$\textit{:} There exist a finite
commutative ring $\mathbb{K}$ and a $\mathbb{K}$-algebra structure on
$\mathbb{L}$.

\item \textit{Statement }$\mathcal{M}$\textit{:} There exists a positive
integer $m$ such that $m\cdot1_{\mathbb{L}}=0$. (Here, we denote the unity of
any ring $\mathbb{A}$ by $1_{\mathbb{A}}$.)
\end{itemize}
\end{proposition}

\begin{proof}
[Proof of Proposition \ref{prop.when-L-K}.]We shall prove the two
implications $\mathcal{K}\Longrightarrow\mathcal{M}$ and $\mathcal{M}%
\Longrightarrow\mathcal{K}$:

\textit{Proof of the implication }$\mathcal{K}\Longrightarrow\mathcal{M}%
$\textit{:} Assume that Statement $\mathcal{K}$ holds. In other words, there
exist a finite commutative ring $\mathbb{K}$ and a $\mathbb{K}$-algebra
structure on $\mathbb{L}$. Consider this ring $\mathbb{K}$ and this structure.

The ring $\mathbb{K}$ is finite. Hence, Lagrange's theorem (applied to the
finite group $\left(  \mathbb{K},+\right)  $) yields $\left\vert
\mathbb{K}\right\vert \cdot a=0$ for each $a\in\mathbb{K}$. Applying this to
$a=1_{\mathbb{K}}$, we obtain $\left\vert \mathbb{K}\right\vert \cdot
1_{\mathbb{K}}=0$. Now, $\left\vert \mathbb{K}\right\vert \cdot
\underbrace{1_{\mathbb{L}}}_{=1_{\mathbb{K}}\cdot1_{\mathbb{L}}}%
=\underbrace{\left\vert \mathbb{K}\right\vert \cdot1_{\mathbb{K}}}_{=0}%
\cdot1_{\mathbb{L}}=0$. Thus, there exists a positive integer $m$ such that
$m\cdot1_{\mathbb{L}}=0$ (namely, $m=\left\vert \mathbb{K}\right\vert $). In
other words, Statement $\mathcal{M}$ holds. This proves the implication
$\mathcal{K}\Longrightarrow\mathcal{M}$.

\textit{Proof of the implication }$\mathcal{M}\Longrightarrow\mathcal{K}%
$\textit{:} Assume that Statement $\mathcal{M}$ holds. In other words, there
exists a positive integer $m$ such that $m\cdot1_{\mathbb{L}}=0$. Consider
this $m$. Then, $\mathbb{Z}/m\mathbb{Z}$ is a finite commutative ring. Now,
the canonical ring homomorphism $\mathbb{Z}\rightarrow\mathbb{L},\ a\mapsto
a\cdot1_{\mathbb{L}}$ factors through the quotient ring $\mathbb{Z}%
/m\mathbb{Z}$ (since it sends $m$ to $m\cdot1_{\mathbb{L}}=0$, and thus its
kernel contains $m$ and therefore the whole ideal $m\mathbb{Z}$). Hence, we
have found a ring homomorphism $\mathbb{Z}/m\mathbb{Z}\rightarrow\mathbb{L}$.
This homomorphism makes $\mathbb{L}$ into a $\mathbb{Z}/m\mathbb{Z}$-algebra
(since $\mathbb{L}$ and $\mathbb{Z}/m\mathbb{Z}$ are commutative). Thus, there
exist a finite commutative ring $\mathbb{K}$ (namely, $\mathbb{Z}/m\mathbb{Z}%
$) and a $\mathbb{K}$-algebra structure on $\mathbb{L}$ (namely, the
$\mathbb{Z}/m\mathbb{Z}$-algebra we have just found). In other words,
Statement $\mathcal{K}$ holds. This proves the implication $\mathcal{M}%
\Longrightarrow\mathcal{K}$.

We have now proven both implications $\mathcal{K}\Longrightarrow\mathcal{M}$
and $\mathcal{M}\Longrightarrow\mathcal{K}$. Thus, Proposition
\ref{prop.when-L-K} is proven.
\end{proof}

Using Proposition \ref{prop.when-L-K}, we can restate Theorem
\ref{thm.finpowmat.main} as follows:

\begin{corollary}
\label{cor.finpowmat.main-rest}Let $\mathbb{L}$ be a commutative ring. Assume
that there exists a positive integer $m$ such that $m\cdot1_{\mathbb{L}}=0$.
Let $n\in\mathbb{N}$. Let $A$ and $B$ be two $n\times n$-matrices over
$\mathbb{L}$ such that $\chi_{A}=\chi_{B}$. Then, the set $\left\{
A^{0},A^{1},A^{2},\ldots\right\}  $ is finite if and only if the set $\left\{
B^{0},B^{1},B^{2},\ldots\right\}  $ is finite.
\end{corollary}

\begin{remark}
A converse of this corollary holds as well: Let $\mathbb{L}$ be a commutative
ring for which there is \textbf{no} positive integer $m$ such that
$m\cdot1_{\mathbb{L}}=0$. Let $n\geq2$ be an integer. Then, there exist two
$n\times n$-matrices $A$ and $B$ over $\mathbb{L}$ such that $\chi_{A}%
=\chi_{B}$ and the set $\left\{  A^{0},A^{1},A^{2},\ldots\right\}  $ is
infinite but the set $\left\{  B^{0},B^{1},B^{2},\ldots\right\}  $ is finite.
Such matrices can easily be constructed by imitation of Example
\ref{exa.finpowmat.1} \textbf{(b)}.
\end{remark}

\subsection{Ingredient 1: Finite semigroups}

We now start preparing the ground for the proof of Theorem
\ref{thm.finpowmat.main}. The first ingredient of our proof are two basic
facts about semigroups.

In the following, semigroups will always be written multiplicatively: That is,
if $M$ is a semigroup, then the operation of $M$ will be written as
multiplication (i.e., we will write $ab$ for the image of $\left(  a,b\right)
\in M\times M$ under this operation).

\begin{theorem}
\label{thm.finpowmat.finmon}Let $M$ be a finite semigroup. Let $a\in M$. Then,
there exists a positive integer $m$ such that $a^{m}=a^{2m}$.
\end{theorem}

\begin{proof}
[Proof of Theorem \ref{thm.finpowmat.finmon}.]This is simply saying that the
sub-semigroup $\left\{  a^{1},a^{2},a^{3},\ldots\right\}  $ of $M$ contains an
idempotent. But this is well-known. See, e.g., \cite[Corollary 1.2]{Steinbe16}
(where $M$ and $a$ are called $S$ and $u$, respectively) or \cite[Proposition
6.31]{Pin19}.
\end{proof}

\begin{proposition}
\label{prop.finpowmat.sg-fin}Let $M$ be a semigroup. Let $a\in M$. Let $p$ and
$q$ be two positive integers such that $p>q$ and $a^{p}=a^{q}$. Then,
$\left\{  a^{1},a^{2},a^{3},\ldots\right\}  =\left\{  a^{1},a^{2}%
,\ldots,a^{p-1}\right\}  $.
\end{proposition}

This proposition is also well-known, but proving it is easier than finding a reference:

\begin{proof}
[Proof of Proposition \ref{prop.finpowmat.sg-fin}.]We claim that%
\begin{equation}
a^{i}\in\left\{  a^{1},a^{2},\ldots,a^{p-1}\right\}
\ \ \ \ \ \ \ \ \ \ \text{for each positive integer }i.
\label{pf.prop.finpowmat.sg-fin.1}%
\end{equation}

[\textit{Proof of (\ref{pf.prop.finpowmat.sg-fin.1}):} We proceed by strong
induction on $i$. Thus, we fix a positive integer $j$, and we assume that
(\ref{pf.prop.finpowmat.sg-fin.1}) holds for all $i<j$. We must then prove
that (\ref{pf.prop.finpowmat.sg-fin.1}) holds for $i=j$. In other words, we
must prove that $a^{j}\in\left\{  a^{1},a^{2},\ldots,a^{p-1}\right\}  $. If
$j\leq p-1$, then this is obvious; thus, for the rest of this proof, we WLOG
assume that $j>p-1$. Hence, $j\geq p$, so that $j-p\geq0$ and thus
$q+\underbrace{\left(  j-p\right)  }_{\geq0}\geq q$. Hence, $q+\left(
j-p\right)  $ is a positive integer (since $q$ is a positive integer).
Furthermore, $q+\left(  j-p\right)  =j+q-\underbrace{p}_{>q}<j+q-q=j$. Thus,
(\ref{pf.prop.finpowmat.sg-fin.1}) holds for $i=q+\left(  j-p\right)  $ (since
we have assumed that (\ref{pf.prop.finpowmat.sg-fin.1}) holds for all $i<j$).
In other words, we have $a^{q+\left(  j-p\right)  }\in\left\{  a^{1}%
,a^{2},\ldots,a^{p-1}\right\}  $.

But we have $a^{p}=a^{q}$. Thus,%
\begin{equation}
a^{p+k}=a^{q+k}\ \ \ \ \ \ \ \ \ \ \text{for each }k\in\mathbb{N}.
\label{pf.prop.finpowmat.sg-fin.1.pf.1}%
\end{equation}
(Indeed, if $k=0$, then this follows from $a^{p}=a^{q}$; but in the other case
it follows from $a^{p+k}=\underbrace{a^{p}}_{=a^{q}}a^{k}=a^{q}a^{k}=a^{q+k}$.)

We have $j-p\geq0$, thus $j-p\in\mathbb{N}$. Hence, applying
(\ref{pf.prop.finpowmat.sg-fin.1.pf.1}) to $k=j-p$, we find $a^{p+\left(
j-p\right)  }=a^{q+\left(  j-p\right)  }$. But $j=p+\left(  j-p\right)  $, so
that $a^{j}=a^{p+\left(  j-p\right)  }=a^{q+\left(  j-p\right)  }\in\left\{
a^{1},a^{2},\ldots,a^{p-1}\right\}  $. In other words,
(\ref{pf.prop.finpowmat.sg-fin.1}) holds for $i=j$. This completes the
induction step. Thus, (\ref{pf.prop.finpowmat.sg-fin.1}) is proven.]

Now, (\ref{pf.prop.finpowmat.sg-fin.1}) immediately yields that $\left\{
a^{1},a^{2},a^{3},\ldots\right\}  \subseteq\left\{  a^{1},a^{2},\ldots
,a^{p-1}\right\}  $. Combining this with the obvious fact that $\left\{
a^{1},a^{2},\ldots,a^{p-1}\right\}  \subseteq\left\{  a^{1},a^{2},a^{3}%
,\ldots\right\}  $, we obtain $\left\{  a^{1},a^{2},a^{3},\ldots\right\}
=\left\{  a^{1},a^{2},\ldots,a^{p-1}\right\}  $. This proves Proposition
\ref{prop.finpowmat.sg-fin}.
\end{proof}

\subsection{Ingredient 2: Integrality basics}

Our proof will rely on some basic properties of integrality over a commutative
ring. This concept is defined as follows:

\begin{definition}
\label{def.finpowmat.integrality}Let $\mathbb{K}$ be a commutative ring. Let
$\mathbb{L}$ be a $\mathbb{K}$-algebra (not necessarily commutative). An
element $u\in\mathbb{L}$ is said to be \textit{integral over }$\mathbb{K}$ if
and only if there exists a monic polynomial $f\in\mathbb{K}\left[  t\right]  $
such that $f\left(  u\right)  =0$.
\end{definition}

Recall that a polynomial is said to be \textit{monic} if its leading
coefficient is $1$. Definition \ref{def.finpowmat.integrality} generalizes
\cite[Definition 2.1.1]{SwaHun06} from commutative ring extensions to
arbitrary algebras, and generalizes \cite[Definition (10.21)]{AllKle14} from
commutative $\mathbb{K}$-algebras $\mathbb{L}$ to arbitrary $\mathbb{K}%
$-algebras $\mathbb{L}$.

Philosophically, there is a similarity between integral elements of a
$\mathbb{K}$-algebra, and \textquotedblleft finite-order\textquotedblright%
\ elements of a semigroup (i.e., elements $a$ such that the set $\left\{
a^{1},a^{2},a^{3},\ldots\right\}  $ is finite). In Proposition
\ref{prop.finpowmat.char-crit}, we shall see a direct connection between these
two concepts, but even before that, the similarity is helpful as a guide.

\begin{definition}
Let $\mathbb{K}$ be a commutative ring. Let $M$ be a $\mathbb{K}$-module, and
let $n\in\mathbb{N}$.

\textbf{(a)} If $m_{1},m_{2},\ldots,m_{n}$ are $n$ elements of $M$, then we
let $\left\langle m_{1},m_{2},\ldots,m_{n}\right\rangle _{\mathbb{K}}$ denote
the $\mathbb{K}$-submodule of $M$ spanned by $m_{1},m_{2},\ldots,m_{n}$. This
$\mathbb{K}$-submodule is called the $\mathbb{K}$\textit{-linear span} of
$m_{1},m_{2},\ldots,m_{n}$. A similar notation will be used for spans of
infinitely many elements.

\textbf{(b)} We say that the $\mathbb{K}$-module $M$ is $n$\textit{-generated}
if and only if there exist $n$ elements $m_{1},m_{2},\ldots,m_{n}\in M$ such
that $M=\left\langle m_{1},m_{2},\ldots,m_{n}\right\rangle _{\mathbb{K}}$.
\end{definition}

We notice that a $\mathbb{K}$-module $M$ is finitely generated if and only if
there exists some $n\in\mathbb{N}$ such that $M$ is $n$-generated.

We recall one basic fact about finitely generated $\mathbb{K}$-modules:

\begin{lemma}
\label{lem.finpowmat.finger-sur}Let $\mathbb{K}$ be a commutative ring. Let
$M$ and $N$ be two $\mathbb{K}$-modules such that $M$ is finitely generated.
Let $f:M\rightarrow N$ be a surjective $\mathbb{K}$-module homomorphism. Then,
the $\mathbb{K}$-module $N$ is finitely generated.
\end{lemma}

\begin{proof}
[Proof of Lemma \ref{lem.finpowmat.finger-sur}.]The $\mathbb{K}%
$-module $M$ is finitely generated. In other words, there exists a finite
list $\left(  m_{1},m_{2},\ldots,m_{n}\right)  $ of elements of $M$ such that
$M=\left\langle m_{1},m_{2},\ldots,m_{n}\right\rangle _{\mathbb{K}}$. Consider
this list.
From $M=\left\langle m_{1},m_{2},\ldots,m_{n}\right\rangle _{\mathbb{K}}$,
we obtain
\[
f\left(  M\right)  =f\left(  \left\langle m_{1},m_{2}%
,\ldots,m_{n}\right\rangle _{\mathbb{K}}\right)  =\left\langle f\left(
m_{1}\right)  ,f\left(  m_{2}\right)  ,\ldots,f\left(  m_{n}\right)
\right\rangle _{\mathbb{K}}
\]
(since $f$ is a $\mathbb{K}$-module
homomorphism). But $f\left(  M\right)  =N$ (since $f$ is surjective). Hence,
$N=f\left(  M\right)  =\left\langle f\left(  m_{1}\right)  ,f\left(
m_{2}\right)  ,\ldots,f\left(  m_{n}\right)  \right\rangle _{\mathbb{K}}$.
Thus, the $\mathbb{K}$-module $N$ is finitely generated. This proves Lemma
\ref{lem.finpowmat.finger-sur}.
\end{proof}

The following fact provides several criteria for when an element of a
commutative $\mathbb{K}$-algebra is integral over $\mathbb{K}$:

\begin{theorem}
\label{thm.finpowmat.int-G10}Let $\mathbb{K}$ be a commutative ring. Let
$\mathbb{L}$ be a commutative $\mathbb{K}$-algebra. Let $n\in\mathbb{N}$. Let
$u\in\mathbb{L}$. Then, the following assertions $\mathcal{A}$, $\mathcal{B}$,
$\mathcal{C}$ and $\mathcal{D}$ are equivalent:

\begin{itemize}
\item \textit{Assertion }$\mathcal{A}$\textit{:} There exists a monic
polynomial $f\in\mathbb{K}\left[  t\right]  $ of degree $n$ such that
$f\left(  u\right)  =0$.

\item \textit{Assertion }$\mathcal{B}$\textit{:} There exist an $\mathbb{L}%
$-module $C$ and an $n$-generated $\mathbb{K}$-submodule $U$ of $C$ such that
$uU\subseteq U$ and such that every $v\in\mathbb{L}$ satisfying $vU=0$
satisfies $v=0$. (Here, we are making use of the fact that each $\mathbb{L}%
$-module canonically becomes a $\mathbb{K}$-module, since $\mathbb{L}$ is a
$\mathbb{K}$-algebra.)

\item \textit{Assertion }$\mathcal{C}$\textit{:} There exists an $n$-generated
$\mathbb{K}$-submodule $U$ of $\mathbb{L}$ such that $1\in U$ and $uU\subseteq
U$.

\item \textit{Assertion }$\mathcal{D}$\textit{:} We have $\mathbb{K}\left[
u\right]  =\left\langle u^{0},u^{1},\ldots,u^{n-1}\right\rangle _{\mathbb{K}}$.
\end{itemize}
\end{theorem}

\begin{proof}
[Proof of Theorem \ref{thm.finpowmat.int-G10}.]
Theorem \ref{thm.finpowmat.int-G10} is precisely
\cite[Theorem 1.1]{Grinbe19} (with $A$, $B$, $X$ and $P$ renamed as $\mathbb{K}%
$, $\mathbb{L}$, $t$ and $f$).
\end{proof}

Note that Theorem \ref{thm.finpowmat.int-G10} is just one of several
\textquotedblleft determinantal tricks\textquotedblright\ used in studying
integrality over rings. See \cite[Chapter V, Section 1.1, Theorem 1]{Bourba72}
or \cite[Theorem 8.1.6]{ChaLoi14} for another. We shall only use the
implications $\mathcal{B}\Longrightarrow\mathcal{A}$ and $\mathcal{A}%
\Longrightarrow\mathcal{D}$ of Theorem \ref{thm.finpowmat.int-G10}.

\begin{noncompile}
TODO: It is probably better to use the more standard versions of the
determinantal trick. I'm using mine for familiarity reasons only. Use, e.g.,
\cite[Proposition (10.23), implication (4) $\Longrightarrow$ (1)]{AllKle14}.
\end{noncompile}

We shall draw the following conclusion from Theorem
\ref{thm.finpowmat.int-G10}:

\begin{corollary}
\label{cor.finpowmat.det-crit}Let $\mathbb{K}$ be a commutative ring. Let
$\mathbb{L}$ be a commutative $\mathbb{K}$-algebra. Let $u\in\mathbb{L}$. Let
$C$ be an $\mathbb{L}$-module. Let $U$ be a finitely generated $\mathbb{K}%
$-submodule of $C$ such that $uU\subseteq U$. Assume that every $v\in
\mathbb{L}$ satisfying $vU=0$ satisfies $v=0$. (Here, we are making use of the
fact that each $\mathbb{L}$-module canonically becomes a $\mathbb{K}$-module,
since $\mathbb{L}$ is a $\mathbb{K}$-algebra.)

Then, $u\in\mathbb{L}$ is integral over $\mathbb{K}$.
\end{corollary}

\begin{proof}
[First proof of Corollary \ref{cor.finpowmat.det-crit}.]The $\mathbb{K}%
$-module $U$ is finitely generated. In other words, it is $n$-generated for
some $n\in\mathbb{N}$. Consider this $n$. Thus, Assertion $\mathcal{B}$ of
Theorem \ref{thm.finpowmat.int-G10} is satisfied. Hence, Assertion
$\mathcal{A}$ of Theorem \ref{thm.finpowmat.int-G10} is satisfied as well
(since Theorem \ref{thm.finpowmat.int-G10} shows that these two assertions are
equivalent). In other words, there exists a monic polynomial $f\in
\mathbb{K}\left[  t\right]  $ of degree $n$ such that $f\left(  u\right)  =0$.
Hence, $u$ is integral over $\mathbb{K}$. This proves Corollary
\ref{cor.finpowmat.det-crit}.
\end{proof}

\begin{proof}
[Second proof of Corollary \ref{cor.finpowmat.det-crit} (sketched).]The
$\mathbb{K}$-module $U$ is finitely generated. In other words, it is
$n$-generated for some $n\in\mathbb{N}$. Consider this $n$.

Consider the commutative $\mathbb{K}$-subalgebra $\mathbb{K}\left[  u\right]
$ of $\mathbb{L}$. Then, the $\mathbb{L}$-module $C$ is a $\mathbb{K}\left[
u\right]  $-module (by restriction). We have $uU\subseteq U$. Using this fact,
it is easy to prove (by induction on $k$) that $u^{k}U\subseteq U$ for each
$k\in\mathbb{N}$. Hence, $fU\subseteq U$ for each $f\in\mathbb{K}\left[
u\right]  $ (since each $f\in\mathbb{K}\left[  u\right]  $ is a $\mathbb{K}%
$-linear combination of the elements $u^{0},u^{1},u^{2},\ldots$, and since $U$
is a $\mathbb{K}$-module). Thus, $U$ is a $\mathbb{K}\left[  u\right]
$-submodule of $C$.

Moreover, we assumed that every $v\in\mathbb{L}$ satisfying $vU=0$ satisfies
$v=0$. Hence, every $v\in\mathbb{K}\left[  u\right]  $ satisfying $vU=0$
satisfies $v=0$ (since $v\in\mathbb{K}\left[  u\right]  \subseteq\mathbb{L}$).
In the parlance of commutative algebra, this is saying that the $\mathbb{K}%
\left[  u\right]  $-module $U$ is faithful. Hence, there is a faithful
$\mathbb{K}\left[  u\right]  $-module which is $n$-generated when considered
as a $\mathbb{K}$-module (namely, $U$). Thus, \cite[Proposition (10.23),
implication (4) $\Longrightarrow$ (1)]{AllKle14} (applied to $R=\mathbb{K}$,
$R^{\prime}=\mathbb{K}\left[  u\right]  $ and $x=u$) shows that there exists a
monic polynomial $f\in\mathbb{K}\left[  t\right]  $ of degree $n$ such that
$f\left(  u\right)  =0$. Hence, $u$ is integral over $\mathbb{K}$. This proves
Corollary \ref{cor.finpowmat.det-crit} again.
\end{proof}

The following proposition is a linear analogue of Proposition
\ref{prop.finpowmat.sg-fin}:

\begin{proposition}
\label{prop.finpowmat.int-fin}Let $\mathbb{K}$ be a commutative ring. Let
$\mathbb{L}$ be a $\mathbb{K}$-algebra. Let $u\in\mathbb{L}$ be integral over
$\mathbb{K}$. Then, there exists a $g\in\mathbb{N}$ such that $\mathbb{K}%
\left[  u\right]  =\left\langle u^{0},u^{1},\ldots,u^{g-1}\right\rangle
_{\mathbb{K}}$.
\end{proposition}

Proposition \ref{prop.finpowmat.int-fin} appears, e.g., in \cite[Proposition
(10.23), implication (1) $\Longrightarrow$ (2)]{AllKle14}. For the sake of
self-containedness, let us prove it as well:

\begin{proof}
[Proof of Proposition \ref{prop.finpowmat.int-fin}.]We first observe that
$\mathbb{K}\left[  u\right]  $ is a $\mathbb{K}$-subalgebra of $\mathbb{L}$,
and that $u\in\mathbb{K}\left[  u\right]  $. The meaning of our assumption
\textquotedblleft$u$ is integral over $\mathbb{K}$\textquotedblright, and also
of our claim \textquotedblleft there exists a $g\in\mathbb{N}$ such that
$\mathbb{K}\left[  u\right]  =\left\langle u^{0},u^{1},\ldots,u^{g-1}%
\right\rangle _{\mathbb{K}}$\textquotedblright, does not depend on whether we
consider $u$ as an element of $\mathbb{L}$ or as an element of $\mathbb{K}%
\left[  u\right]  $. Thus, for the rest of this proof, we can WLOG assume that
$\mathbb{L}=\mathbb{K}\left[  u\right]  $ (since otherwise, we can simply
replace $\mathbb{L}$ by $\mathbb{K}\left[  u\right]  $). Assume this. Then,
$\mathbb{L}$ is commutative (since $\mathbb{K}\left[  u\right]  $ is clearly commutative).

We have assumed that $u$ is integral over $\mathbb{K}$. In other words, there
exists a monic polynomial $f\in\mathbb{K}\left[  t\right]  $ such that
$f\left(  u\right)  =0$. Consider this $f$. Set $n=\deg f$. Then, Assertion
$\mathcal{A}$ of Theorem \ref{thm.finpowmat.int-G10} is satisfied (since $f$
has degree $n$). Hence, Assertion $\mathcal{D}$ of Theorem
\ref{thm.finpowmat.int-G10} is satisfied as well (since Theorem
\ref{thm.finpowmat.int-G10} shows that these two assertions are equivalent).
In other words, we have $\mathbb{K}\left[  u\right]  =\left\langle u^{0}%
,u^{1},\ldots,u^{n-1}\right\rangle _{\mathbb{K}}$. Thus, there exists a
$g\in\mathbb{N}$ such that $\mathbb{K}\left[  u\right]  =\left\langle
u^{0},u^{1},\ldots,u^{g-1}\right\rangle _{\mathbb{K}}$ (namely, $g=n$). This
proves Proposition \ref{prop.finpowmat.int-fin}.
\end{proof}

\begin{theorem}
\label{thm.finpowmat.fin-alg}Let $\mathbb{K}$ be a commutative ring. Let
$\mathbb{L}$ be a commutative $\mathbb{K}$-algebra. Let $u_{1},u_{2}%
,\ldots,u_{m}$ be a finite list of elements of $\mathbb{L}$. Assume that these
$m$ elements $u_{1},u_{2},\ldots,u_{m}$ are all integral over $\mathbb{K}$,
and generate $\mathbb{L}$ as a $\mathbb{K}$-algebra. Then, the $\mathbb{K}%
$-module $\mathbb{L}$ is finitely generated.
\end{theorem}

Theorem \ref{thm.finpowmat.fin-alg} appears, e.g., in \cite[Theorem (10.28),
implication (2) $\Longrightarrow$ (3)]{AllKle14}. For the sake of
self-containedness, let us prove it as well:

\begin{proof}
[Proof of Theorem \ref{thm.finpowmat.fin-alg}.]Fix $i\in\left\{
1,2,\ldots,m\right\}  $. Then, $u_{i}\in\mathbb{L}$ is integral over
$\mathbb{K}$ (by assumption). Hence, Proposition \ref{prop.finpowmat.int-fin}
(applied to $u=u_{i}$) shows that there exists a $g\in\mathbb{N}$ such that
$\mathbb{K}\left[  u_{i}\right]  =\left\langle u_{i}^{0},u_{i}^{1}%
,\ldots,u_{i}^{g-1}\right\rangle _{\mathbb{K}}$. Thus, the $\mathbb{K}$-module
$\mathbb{K}\left[  u_{i}\right]  $ is finitely generated.

Now, forget that we fixed $i$. We thus have shown that for each $i\in\left\{
1,2,\ldots,m\right\}  $, the $\mathbb{K}$-module $\mathbb{K}\left[
u_{i}\right]  $ is finitely generated. Hence, the $\mathbb{K}$%
-module\footnote{In this proof, the \textquotedblleft$\otimes$%
\textquotedblright\ symbol always means a tensor product over $\mathbb{K}$.}
$\mathbb{K}\left[  u_{1}\right]  \otimes\mathbb{K}\left[  u_{2}\right]
\otimes\cdots\otimes\mathbb{K}\left[  u_{m}\right]  $ is also finitely
generated\footnote{This is because of the following general fact: If
$A_{1},A_{2},\ldots,A_{m}$ are $m$ finitely generated $\mathbb{K}$-modules,
then the $\mathbb{K}$-module $A_{1}\otimes A_{2}\otimes\cdots\otimes A_{m}$ is
also finitely generated. (This fact can be proven as follows: For each
$i\in\left\{  1,2,\ldots,m\right\}  $, we fix a finite family $\left(
a_{i,s}\right)  _{s\in S_{i}}$ of vectors in $A_{i}$ that generates $A_{i}$.
Then, the family
\[
\left(  a_{1,s_{1}}\otimes a_{2,s_{2}}\otimes\cdots\otimes a_{m,s_{m}}\right)
_{\left(  s_{1},s_{2},\ldots,s_{m}\right)  \in S_{1}\times S_{2}\times
\cdots\times S_{m}}%
\]
of vectors in $A_{1}\otimes A_{2}\otimes\cdots\otimes A_{m}$ is finite and
generates the $\mathbb{K}$-module $A_{1}\otimes A_{2}\otimes\cdots\otimes
A_{m}$. Hence, $A_{1}\otimes A_{2}\otimes\cdots\otimes A_{m}$ is finitely
generated.)}. The $\mathbb{K}$-module homomorphism%
\begin{align*}
\pi:\mathbb{K}\left[  u_{1}\right]  \otimes\mathbb{K}\left[  u_{2}\right]
\otimes\cdots\otimes\mathbb{K}\left[  u_{m}\right]   &  \rightarrow
\mathbb{L},\\
a_{1}\otimes a_{2}\otimes\cdots\otimes a_{m}  &  \mapsto a_{1}a_{2}\cdots
a_{m}%
\end{align*}
is surjective (since $u_{1},u_{2},\ldots,u_{m}$ generate $\mathbb{L}$ as a
$\mathbb{K}$-algebra\footnote{To be more precise: The image of $\pi$ is a
$\mathbb{K}$-submodule of $\mathbb{L}$ (since $\pi$ is a $\mathbb{K}$-module
homomorphism). But we have assumed that $u_{1},u_{2},\ldots,u_{m}$ generate
$\mathbb{L}$ as a $\mathbb{K}$-algebra. Thus, each element of $\mathbb{L}$ can
be written as a polynomial in the $u_{1},u_{2},\ldots,u_{m}$ with coefficients
in $\mathbb{K}$ (since $\mathbb{L}$ is commutative). In other words, each
element of $\mathbb{L}$ can be written as a $\mathbb{K}$-linear combination of
products of the form $u_{1}^{n_{1}}u_{2}^{n_{2}}\cdots u_{m}^{n_{m}}$ with
$n_{1},n_{2},\ldots,n_{m}\in\mathbb{N}$. But each of the latter products
belongs to the image of $\pi$ (because $u_{1}^{n_{1}}u_{2}^{n_{2}}\cdots
u_{m}^{n_{m}}=\pi\left(  u_{1}^{n_{1}}\otimes u_{2}^{n_{2}}\otimes
\cdots\otimes u_{m}^{n_{m}}\right)  $ for all $n_{1},n_{2},\ldots,n_{m}%
\in\mathbb{N}$). Hence, each element of $\mathbb{L}$ can be written as a
$\mathbb{K}$-linear combination of elements of the image of $\pi$, and thus
itself belongs to the image of $\pi$ (since the image of $\pi$ is a
$\mathbb{K}$-submodule of $\mathbb{L}$). In other words, $\pi$ is
surjective.}). Hence, Lemma \ref{lem.finpowmat.finger-sur} (applied to
$M=\mathbb{K}\left[  u_{1}\right]  \otimes\mathbb{K}\left[  u_{2}\right]
\otimes\cdots\otimes\mathbb{K}\left[  u_{m}\right]  $, $N=\mathbb{L}$ and
$f=\pi$) shows that the $\mathbb{K}$-module $\mathbb{L}$ is finitely
generated. This proves Theorem \ref{thm.finpowmat.fin-alg}.
\end{proof}

\subsection{Characterizing integral matrices}

The following is a simple consequence of the Cayley--Hamilton theorem:

\begin{proposition}
\label{prop.finpowmat.ch-cor}Let $\mathbb{K}$ be a commutative ring. Let
$n\in\mathbb{N}$. Let $A$ be an $n\times n$-matrix over $\mathbb{K}$. Then,
$A$ is integral over $\mathbb{K}$ (as an element of the $\mathbb{K}$-algebra
$\mathbb{K}^{n\times n}$).
\end{proposition}

\begin{proof}
[Proof of Proposition \ref{prop.finpowmat.ch-cor}.]The characteristic
polynomial $\chi_{A}$ of $A$ is a monic polynomial in $\mathbb{K}\left[
t\right]  $. The Cayley--Hamilton theorem yields $\chi_{A}\left(  A\right)
=0$. Thus, there exists a monic polynomial $f\in\mathbb{K}\left[  t\right]  $
such that $f\left(  A\right)  =0$ (namely, $f=\chi_{A}$). In other words, $A$
is integral over $\mathbb{K}$. This proves Proposition
\ref{prop.finpowmat.ch-cor}.
\end{proof}

It is not hard to prove a generalization of Proposition
\ref{prop.finpowmat.ch-cor}:

\begin{proposition}
\label{prop.finpowmat.char-int-conv}Let $\mathbb{K}$ be a commutative ring.
Let $n\in\mathbb{N}$. Let $\mathbb{L}$ be a commutative $\mathbb{K}$-algebra.
Let $A$ be an $n\times n$-matrix over $\mathbb{L}$. Assume that each
coefficient of the characteristic polynomial $\chi_{A}\in\mathbb{L}\left[
t\right]  $ is integral over $\mathbb{K}$. Then, $A$ is integral over
$\mathbb{K}$ (as an element of the $\mathbb{K}$-algebra $\mathbb{L}^{n\times
n}$).
\end{proposition}

We will not need this proposition, so we banish its proof into Section
\ref{sec.finpowmat.converse}. However, we will use its converse:

\begin{theorem}
\label{thm.finpowmat.char-int}Let $\mathbb{K}$ be a commutative ring. Let
$n\in\mathbb{N}$. Let $\mathbb{L}$ be a commutative $\mathbb{K}$-algebra. Let
$A$ be an $n\times n$-matrix over $\mathbb{L}$. Assume that $A$ is integral
over $\mathbb{K}$ (as an element of the $\mathbb{K}$-algebra $\mathbb{L}%
^{n\times n}$). Then, each coefficient of the characteristic polynomial
$\chi_{A}\in\mathbb{L}\left[  t\right]  $ is integral over $\mathbb{K}$.
\end{theorem}

Proposition \ref{prop.finpowmat.char-int-conv} and
Theorem \ref{thm.finpowmat.char-int} are parts of
\cite[Chapter V, Section 1.6, Proposition 17, (a) $\Longleftrightarrow$ (c)]{Bourba72}.
For convenience and expository value, we shall nevertheless
reprove them here.

\subsection{Gert Almkvist's exterior-power trick}

Our following proof of Theorem \ref{thm.finpowmat.char-int} will rely on the
notion of exterior powers of an $\mathbb{L}$-module (where $\mathbb{L}$ is a
commutative ring). See \cite[Chapter III, \S 7]{Bourba74} or
\cite{Conrad-extmod} for the relevant background. Our method is inspired by
Gert Almkvist's exterior-power trick (\cite[proof of Theorem 1.7]{Almkvi73},
\cite{Zeilbe93}). We shall need the following proposition (which is
essentially the equality ($\ast^{\prime}$) in \cite{Zeilbe93}, or the equality
($\ast\ast$) in \cite[proof of Theorem 1.7]{Almkvi73}):

\begin{proposition}
\label{prop.finpowmat.almk-coeff}Let $\mathbb{K}$ be a commutative ring. Let
$n\in\mathbb{N}$. Let $A\in\mathbb{K}^{n\times n}$ be an $n\times n$-matrix.
Let $V$ be the free $\mathbb{K}$-module $\mathbb{K}^{n}$ (consisting of column
vectors of size $n$). Consider $A$ as an endomorphism of the free $\mathbb{K}%
$-module $V=\mathbb{K}^{n}$ (in the usual way: i.e., we let $A\left(
v\right)  =Av$ for each column vector $v\in\mathbb{K}^{n}$). Consider the
$n$-th exterior power $\Lambda^{n}V$ of the $\mathbb{K}$-module $V$.

Fix $k\in\mathbb{N}$. Let $a_{k}\in\mathbb{K}$ be the coefficient of $t^{k}$
in the characteristic polynomial $\chi_{A}\in\mathbb{K}\left[  t\right]  $.
Then, for each $w_{1},w_{2},\ldots,w_{n}\in V$, we have%
\[
a_{k}\cdot w_{1}\wedge w_{2}\wedge\cdots\wedge w_{n}=\left(  -1\right)
^{n-k}\sum_{\substack{i_{1},i_{2},\ldots,i_{n}\in\left\{  0,1\right\}
;\\i_{1}+i_{2}+\cdots+i_{n}=n-k}}A^{i_{1}}w_{1}\wedge A^{i_{2}}w_{2}%
\wedge\cdots\wedge A^{i_{n}}w_{n}.
\]

\end{proposition}

Before we prove this proposition, we need a well-known lemma that connects
exterior powers with determinants:

\begin{lemma}
\label{lem.finpowmat.ext-det}Let $\mathbb{L}$ be a commutative ring. Let
$n\in\mathbb{N}$. If $M$ is an $\mathbb{L}$-module, then $\Lambda_{\mathbb{L}%
}^{n}M$ shall denote the $n$-th exterior power of the $\mathbb{L}$-module $M$.

\textbf{(a)} If $M$ is an $\mathbb{L}$-module, then each endomorphism $u$ of
the $\mathbb{L}$-module $M$ induces an endomorphism $\Lambda_{\mathbb{L}}%
^{n}u$ of the $\mathbb{L}$-module $\Lambda_{\mathbb{L}}^{n}M$, defined by%
\begin{align*}
\left(  \Lambda_{\mathbb{L}}^{n}u\right)  \left(  w_{1}\wedge w_{2}%
\wedge\cdots\wedge w_{n}\right)   &  =uw_{1}\wedge uw_{2}\wedge\cdots\wedge
uw_{n}\\
&  \ \ \ \ \ \ \ \ \ \ \text{for all }w_{1},w_{2},\ldots,w_{n}\in M.
\end{align*}

\textbf{(b)} For each endomorphism $u$ of the free $\mathbb{L}$-module
$\mathbb{L}^{n}$ and each $p\in\Lambda_{\mathbb{L}}^{n}\left(  \mathbb{L}%
^{n}\right)  $, we have%
\begin{equation}
\left(  \Lambda_{\mathbb{L}}^{n}u\right)  p=\det u\cdot p.
\label{pf.lem.finpowmat.almk-coeff.ht1}%
\end{equation}
(Here, $\det u$ stands for the determinant of $u$, which is defined as the
determinant of the $n\times n$-matrix representing $u$.)
\end{lemma}

\begin{proof}
[Proof of Lemma \ref{lem.finpowmat.ext-det}.]Lemma \ref{lem.finpowmat.ext-det}
\textbf{(a)} appears (e.g.) in \cite[Chapter III, \S 7.2, equation
(4)]{Bourba74} and in \cite[Theorem 5.1]{Conrad-extmod}. Lemma
\ref{lem.finpowmat.ext-det} \textbf{(b)} appears (e.g.) in \cite[Theorem
6.1]{Conrad-extmod} and (implicitly) in \cite{Bourba74} as well (indeed, the
formula (\ref{pf.lem.finpowmat.almk-coeff.ht1}) is how $\det u$ is defined in
\cite[Chapter III, \S 8, Section 1, Definition 1]{Bourba74} (applied to
$A=\mathbb{L}$ and $M=\mathbb{L}^{n}$)).
\end{proof}

We shall also need a lemma about $n$-th exterior powers of free modules of
rank $n$:

\begin{lemma}
\label{lem.finpowmat.ext-rankn}Let $\mathbb{L}$ be a commutative ring. Let
$n\in\mathbb{N}$. Let $M$ be an $\mathbb{L}$-module with a basis $\left(
b_{1},b_{2},\ldots,b_{n}\right)  $. Then, the $1$-tuple $\left(  b_{1}\wedge
b_{2}\wedge\cdots\wedge b_{n}\right)  $ is a basis of the $\mathbb{L}$-module
$\Lambda_{\mathbb{L}}^{n}M$. (Here, $\Lambda_{\mathbb{L}}^{n}M$ denotes the
$n$-th exterior power of the $\mathbb{L}$-module $M$.)
\end{lemma}

\begin{proof}
[Proof of Lemma \ref{lem.finpowmat.ext-rankn}.]See \cite[Theorem
4.2]{Conrad-extmod} or \cite[Chapter III, \S 7.9, Corollary 1]{Bourba74}.
\end{proof}

We shall now give a proof of Proposition \ref{prop.finpowmat.almk-coeff}; a
second proof (somewhat more elementary, but more laborious) will be provided
in Section \ref{sec.finpowmat.almk-coeff-pf2}.

\begin{proof}
[First proof of Proposition \ref{prop.finpowmat.almk-coeff}.]Note that $V$ is
a free $\mathbb{K}$-module of rank $n$; thus, $\Lambda^{n}V$ is a free
$\mathbb{K}$-module of rank $\dbinom{n}{n}=1$. In other words, $\Lambda
^{n}V\cong\mathbb{K}$ as a $\mathbb{K}$-module.

Let $\mathbb{L}$ be the polynomial ring $\mathbb{K}\left[  t\right]  $. Then,
$\mathbb{L}$ is a commutative $\mathbb{K}$-algebra. We consider the
commutative ring $\mathbb{K}$ as a subring of the polynomial ring
$\mathbb{L}=\mathbb{K}\left[  t\right]  $ (which is also commutative). Thus,
the free $\mathbb{K}$-module $V=\mathbb{K}^{n}$ canonically embeds into the
free $\mathbb{L}$-module $\mathbb{L}^{n}$, and its $n$-th exterior
power\footnote{In the following, the symbol \textquotedblleft$\Lambda^{n}%
$\textquotedblright\ without a subscript will always mean an $n$-th exterior
power over the base ring $\mathbb{K}$.} $\Lambda^{n}V$ canonically embeds into
the corresponding $n$-th exterior power $\Lambda_{\mathbb{L}}^{n}\left(
\mathbb{L}^{n}\right)  $, where the subscript \textquotedblleft$_{\mathbb{L}}%
$\textquotedblright\ signals that this is an exterior power over the base ring
$\mathbb{L}$. (This is indeed an embedding, since both modules $\Lambda^{n}V$
and $\Lambda_{\mathbb{L}}^{n}\left(  \mathbb{L}^{n}\right)  $ have bases
consisting of $1$ element only (by Lemma \ref{lem.finpowmat.ext-rankn}), and
the canonical map $\Lambda^{n}V\rightarrow\Lambda_{\mathbb{L}}^{n}\left(
\mathbb{L}^{n}\right)  $ sends the basis of one to the basis of the other.)

But $\mathbb{L}\otimes_{\mathbb{K}}\underbrace{V}_{=\mathbb{K}^{n}}%
=\mathbb{L}\otimes_{\mathbb{K}}\mathbb{K}^{n}\cong\mathbb{L}^{n}$ as
$\mathbb{L}$-modules. Hence, the canonical $\mathbb{L}$-module homomorphism%
\begin{align*}
\mathbb{L}\otimes_{\mathbb{K}}\left(  \Lambda^{n}V\right)   &  \rightarrow
\Lambda_{\mathbb{L}}^{n}\left(  \mathbb{L}^{n}\right)  ,\\
\ell\otimes\left(  w_{1}\wedge w_{2}\wedge\cdots\wedge w_{n}\right)   &
\mapsto\ell\cdot\left(  w_{1}\wedge w_{2}\wedge\cdots\wedge w_{n}\right)
\end{align*}
is an $\mathbb{L}$-module isomorphism (see, e.g., \cite[Theorem 1]{Conrad13};
see also \cite[Chapter III, \S 7.5, Proposition 8]{Bourba74} for the inverse
of this isomorphism). We use this isomorphism to identify the $\mathbb{L}%
$-module $\Lambda_{\mathbb{L}}^{n}\left(  \mathbb{L}^{n}\right)  $ with
$\mathbb{L}\otimes_{\mathbb{K}}\left(  \Lambda^{n}V\right)  $. Thus,%
\begin{equation}
\Lambda_{\mathbb{L}}^{n}\left(  \mathbb{L}^{n}\right)  =\underbrace{\mathbb{L}%
}_{=\mathbb{K}\left[  t\right]  }\otimes_{\mathbb{K}}\left(  \Lambda
^{n}V\right)  =\mathbb{K}\left[  t\right]  \otimes_{\mathbb{K}}\left(
\Lambda^{n}V\right)  \cong\left(  \Lambda^{n}V\right)  \left[  t\right]  .
\label{pf.lem.finpowmat.almk-coeff.identifies}%
\end{equation}
Concretely, this means that every element of $\Lambda_{\mathbb{L}}^{n}\left(
\mathbb{L}^{n}\right)  $ can be written as a polynomial in $t$ with
coefficients in $\Lambda^{n}V$.

Note that our canonical embedding $\Lambda^{n}V\hookrightarrow\Lambda
_{\mathbb{L}}^{n}\left(  \mathbb{L}^{n}\right)  $ sends each $p\in\Lambda
^{n}V$ to $1\otimes p\in\mathbb{L}\otimes_{\mathbb{K}}\left(  \Lambda
^{n}V\right)  =\Lambda_{\mathbb{L}}^{n}\left(  \mathbb{L}^{n}\right)  $.

Consider the matrix $tI_{n}-A\in\left(  \mathbb{K}\left[  t\right]  \right)
^{n\times n}=\mathbb{L}^{n\times n}$ (since $\mathbb{K}\left[  t\right]
=\mathbb{L}$). It satisfies%
\[
tI_{n}-A=\sum_{i\in\left\{  0,1\right\}  }t^{1-i}\left(  -1\right)  ^{i}A^{i}%
\]
(since $\sum_{i\in\left\{  0,1\right\}  }t^{1-i}\left(  -1\right)  ^{i}%
A^{i}=\underbrace{t^{1-0}}_{=t}\underbrace{\left(  -1\right)  ^{0}}%
_{=1}\underbrace{A^{0}}_{=I_{n}}+\underbrace{t^{1-1}}_{=t^{0}=1}%
\underbrace{\left(  -1\right)  ^{1}}_{=-1}\underbrace{A^{1}}_{=A}%
=tI_{n}+\left(  -1\right)  A=tI_{n}-A$).

The $n\times n$-matrix $tI_{n}-A\in\mathbb{L}^{n\times n}$ can be viewed as an
endomorphism of the free $\mathbb{L}$-module $\mathbb{L}^{n}$ (since any
$n\times n$-matrix over $\mathbb{L}$ can be viewed as such an endomorphism).
Applying (\ref{pf.lem.finpowmat.almk-coeff.ht1}) to $u=tI_{n}-A$ (or, more
precisely, to the $\mathbb{L}$-module endomorphism we just mentioned), we
obtain%
\begin{equation}
\left(  \Lambda_{\mathbb{L}}^{n}\left(  tI_{n}-A\right)  \right)
p=\det\left(  tI_{n}-A\right)  \cdot p \label{pf.lem.finpowmat.almk-coeff.ht2}%
\end{equation}
for each $p\in\Lambda_{\mathbb{L}}^{n}\left(  \mathbb{L}^{n}\right)  $. (Here,
of course, the meaning of $\Lambda_{\mathbb{L}}^{n}\left(  tI_{n}-A\right)  $
is as in Lemma \ref{lem.finpowmat.ext-det}.)

Now, fix $w_{1},w_{2},\ldots,w_{n}\in V$, and set $p=w_{1}\wedge w_{2}%
\wedge\cdots\wedge w_{n}\in\Lambda^{n}V$. Note that $p\in\Lambda^{n}%
V\subseteq\Lambda_{\mathbb{L}}^{n}\left(  \mathbb{L}^{n}\right)  $.

Definition \ref{def.finpowmat.char} yields $\chi_{A}=\det\left(
tI_{n}-A\right)  $. Hence,%
\begin{align*}
&  \chi_{A}\cdot p\\
&  =\det\left(  tI_{n}-A\right)  \cdot p\\
&  =\left(  \Lambda_{\mathbb{L}}^{n}\left(  tI_{n}-A\right)  \right)
p\ \ \ \ \ \ \ \ \ \ \left(  \text{by (\ref{pf.lem.finpowmat.almk-coeff.ht2}%
)}\right) \\
&  =\left(  \Lambda_{\mathbb{L}}^{n}\left(  tI_{n}-A\right)  \right)  \left(
w_{1}\wedge w_{2}\wedge\cdots\wedge w_{n}\right)  \ \ \ \ \ \ \ \ \ \ \left(
\text{since }p=w_{1}\wedge w_{2}\wedge\cdots\wedge w_{n}\right) \\
&  =\left(  tI_{n}-A\right)  w_{1}\wedge\left(  tI_{n}-A\right)  w_{2}%
\wedge\cdots\wedge\left(  tI_{n}-A\right)  w_{n}\\
&  \ \ \ \ \ \ \ \ \ \ \left(  \text{by the definition of }\Lambda
_{\mathbb{L}}^{n}\left(  tI_{n}-A\right)  \right) \\
&  =\left(  \sum_{i\in\left\{  0,1\right\}  }t^{1-i}\left(  -1\right)
^{i}A^{i}\right)  w_{1}\wedge\left(  \sum_{i\in\left\{  0,1\right\}  }%
t^{1-i}\left(  -1\right)  ^{i}A^{i}\right)  w_{2}\wedge\cdots\\
&  \ \ \ \ \ \ \ \ \ \ \wedge\left(  \sum_{i\in\left\{  0,1\right\}  }%
t^{1-i}\left(  -1\right)  ^{i}A^{i}\right)  w_{n}\ \ \ \ \ \ \ \ \ \ \left(
\text{since }tI_{n}-A=\sum_{i\in\left\{  0,1\right\}  }t^{1-i}\left(
-1\right)  ^{i}A^{i}\right) \\
&  =\sum_{i_{1},i_{2},\ldots,i_{n}\in\left\{  0,1\right\}  }%
\underbrace{t^{1-i_{1}}\left(  -1\right)  ^{i_{1}}A^{i_{1}}w_{1}\wedge
t^{1-i_{2}}\left(  -1\right)  ^{i_{2}}A^{i_{2}}w_{2}\wedge\cdots\wedge
t^{1-i_{n}}\left(  -1\right)  ^{i_{n}}A^{i_{n}}w_{n}}_{\substack{=t^{\left(
1-i_{1}\right)  +\left(  1-i_{2}\right)  +\cdots+\left(  1-i_{n}\right)
}\left(  -1\right)  ^{i_{1}+i_{2}+\cdots+i_{n}}\cdot A^{i_{1}}w_{1}\wedge
A^{i_{2}}w_{2}\wedge\cdots\wedge A^{i_{n}}w_{n}\\\text{(by the multilinearity
of the exterior product)}}}\\
&  \ \ \ \ \ \ \ \ \ \ \left(  \text{by the multilinearity of the exterior
product}\right) \\
&  =\sum_{i_{1},i_{2},\ldots,i_{n}\in\left\{  0,1\right\}  }%
\underbrace{t^{\left(  1-i_{1}\right)  +\left(  1-i_{2}\right)  +\cdots
+\left(  1-i_{n}\right)  }}_{=t^{n-\left(  i_{1}+i_{2}+\cdots+i_{n}\right)  }%
}\left(  -1\right)  ^{i_{1}+i_{2}+\cdots+i_{n}}\cdot A^{i_{1}}w_{1}\wedge
A^{i_{2}}w_{2}\wedge\cdots\wedge A^{i_{n}}w_{n}\\
&  =\sum_{i_{1},i_{2},\ldots,i_{n}\in\left\{  0,1\right\}  }t^{n-\left(
i_{1}+i_{2}+\cdots+i_{n}\right)  }\left(  -1\right)  ^{i_{1}+i_{2}%
+\cdots+i_{n}}\cdot A^{i_{1}}w_{1}\wedge A^{i_{2}}w_{2}\wedge\cdots\wedge
A^{i_{n}}w_{n}.
\end{align*}
This is an equality in $\Lambda_{\mathbb{L}}^{n}\left(  \mathbb{L}^{n}\right)
$. In view of (\ref{pf.lem.finpowmat.almk-coeff.identifies}), this becomes an
equality in $\left(  \Lambda^{n}V\right)  \left[  t\right]  $. Hence, by
comparing the coefficients of $t^{k}$ on both sides of this equality, we
obtain%
\[
a_{k}\cdot p=\sum_{\substack{i_{1},i_{2},\ldots,i_{n}\in\left\{  0,1\right\}
;\\n-\left(  i_{1}+i_{2}+\cdots+i_{n}\right)  =k}}\left(  -1\right)
^{i_{1}+i_{2}+\cdots+i_{n}}\cdot A^{i_{1}}w_{1}\wedge A^{i_{2}}w_{2}%
\wedge\cdots\wedge A^{i_{n}}w_{n}%
\]
(since the coefficient of $t^{k}$ in $\chi_{A}$ is $a_{k}$, and thus the
coefficient of $t^{k}$ in $\chi_{A}\cdot p$ is $a_{k}\cdot p$). Since the
condition \textquotedblleft$n-\left(  i_{1}+i_{2}+\cdots+i_{n}\right)
=k$\textquotedblright\ under the summation sign is equivalent to
\textquotedblleft$i_{1}+i_{2}+\cdots+i_{n}=n-k$\textquotedblright, we can
rewrite this as follows:%
\begin{align*}
a_{k}\cdot p  &  =\sum_{\substack{i_{1},i_{2},\ldots,i_{n}\in\left\{
0,1\right\}  ;\\i_{1}+i_{2}+\cdots+i_{n}=n-k}}\underbrace{\left(  -1\right)
^{i_{1}+i_{2}+\cdots+i_{n}}}_{\substack{=\left(  -1\right)  ^{n-k}%
\\\text{(since }i_{1}+i_{2}+\cdots+i_{n}=n-k\text{)}}}\cdot A^{i_{1}}%
w_{1}\wedge A^{i_{2}}w_{2}\wedge\cdots\wedge A^{i_{n}}w_{n}\\
&  =\left(  -1\right)  ^{n-k}\sum_{\substack{i_{1},i_{2},\ldots,i_{n}%
\in\left\{  0,1\right\}  ;\\i_{1}+i_{2}+\cdots+i_{n}=n-k}}A^{i_{1}}w_{1}\wedge
A^{i_{2}}w_{2}\wedge\cdots\wedge A^{i_{n}}w_{n}.
\end{align*}
In view of $p=w_{1}\wedge w_{2}\wedge\cdots\wedge w_{n}$, this rewrites as%
\[
a_{k}\cdot w_{1}\wedge w_{2}\wedge\cdots\wedge w_{n}=\left(  -1\right)
^{n-k}\sum_{\substack{i_{1},i_{2},\ldots,i_{n}\in\left\{  0,1\right\}
;\\i_{1}+i_{2}+\cdots+i_{n}=n-k}}A^{i_{1}}w_{1}\wedge A^{i_{2}}w_{2}%
\wedge\cdots\wedge A^{i_{n}}w_{n}.
\]
This proves Proposition \ref{prop.finpowmat.almk-coeff}.
\end{proof}

\subsection{Characterizing integral matrices: the proof}

Now, everything needed for proving Theorem \ref{thm.finpowmat.char-int} is in
place, so we can start the proof:

\begin{proof}
[Proof of Theorem \ref{thm.finpowmat.char-int}.]Let $V$ be the free
$\mathbb{L}$-module $\mathbb{L}^{n}$. In the following, $\Lambda^{n}V$ shall
always mean the $n$-th exterior power of the $\mathbb{L}$-module $V$.

We have assumed that $A$ is integral over $\mathbb{K}$ (as an element of the
$\mathbb{K}$-algebra $\mathbb{L}^{n\times n}$). Hence, Proposition
\ref{prop.finpowmat.int-fin} (applied to $\mathbb{L}^{n\times n}$ and $A$
instead of $\mathbb{L}$ and $u$) shows that there exists a $g\in\mathbb{N}$
such that $\mathbb{K}\left[  A\right]  =\left\langle A^{0},A^{1}%
,\ldots,A^{g-1}\right\rangle _{\mathbb{K}}$. Consider this $g$.

Let $C$ be the $\mathbb{L}$-module $\Lambda^{n}V$. It is easy (using Lemma
\ref{lem.finpowmat.ext-rankn}) to show that $C\cong\mathbb{L}$ as an
$\mathbb{L}$-module, but we shall not use this.

Let $U$ be the $\mathbb{K}$-submodule of $C$ spanned by all elements of the
form%
\begin{equation}
B_{1}v_{1}\wedge B_{2}v_{2}\wedge\cdots\wedge B_{n}v_{n}%
\ \ \ \ \ \ \ \ \ \ \text{with }B_{j}\in\mathbb{K}\left[  A\right]  \text{ and
}v_{j}\in\mathbb{K}^{n}. \label{pf.thm.finpowmat.char-int.form1}%
\end{equation}
Here, we are regarding $\mathbb{K}^{n}$ as a $\mathbb{K}$-submodule of
$\mathbb{L}^{n}$, so that the vectors $v_{j}\in\mathbb{K}^{n}$ automatically
become vectors in $\mathbb{L}^{n}$ (and thus the matrices $B_{j}\in
\mathbb{K}\left[  A\right]  \subseteq\mathbb{L}^{n\times n}$ can be multiplied
onto them, yielding new vectors $B_{j}v_{j}\in\mathbb{L}^{n}$). Now, we claim
the following:

\begin{statement}
\textit{Claim 1:} The $\mathbb{K}$-module $U$ is finitely generated.
\end{statement}

[\textit{First proof of Claim 1:} Let $\left(  e_{1},e_{2},\ldots
,e_{n}\right)  $ be the standard basis of the $\mathbb{K}$-module
$\mathbb{K}^{n}$ (so that $e_{i}$ is the vector with a $1$ in its $i$-th entry
and $0$ everywhere else).

We have $\mathbb{K}\left[  A\right]  =\left\langle A^{0},A^{1},\ldots
,A^{g-1}\right\rangle _{\mathbb{K}}$. Thus, each $B\in\mathbb{K}\left[
A\right]  $ is a $\mathbb{K}$-linear combination of $A^{0},A^{1}%
,\ldots,A^{g-1}$. We can thus easily see that each element of the form
(\ref{pf.thm.finpowmat.char-int.form1}) is a $\mathbb{K}$-linear combination
of elements of the form%
\begin{equation}
A^{i_{1}}v_{1}\wedge A^{i_{2}}v_{2}\wedge\cdots\wedge A^{i_{n}}v_{n}%
\ \ \ \ \ \ \ \ \ \ \text{with }i_{j}\in\left\{  0,1,\ldots,g-1\right\}
\text{ and }v_{j}\in\mathbb{K}^{n}. \label{pf.thm.finpowmat.char-int.form2}%
\end{equation}
Thus, the $\mathbb{K}$-module $U$ is spanned by all elements of the form
(\ref{pf.thm.finpowmat.char-int.form2}) (since it is spanned by all elements
of the form (\ref{pf.thm.finpowmat.char-int.form1})).

But $\left(  e_{1},e_{2},\ldots,e_{n}\right)  $ is the standard basis of the
$\mathbb{K}$-module $\mathbb{K}^{n}$. Thus, the elements $e_{1},e_{2}%
,\ldots,e_{n}$ span the $\mathbb{K}$-module $\mathbb{K}^{n}$. Hence, each
$v\in\mathbb{K}^{n}$ is a $\mathbb{K}$-linear combination of $e_{1}%
,e_{2},\ldots,e_{n}$. Hence, each element of the form
(\ref{pf.thm.finpowmat.char-int.form2}) is a $\mathbb{K}$-linear combination
of elements of the form%
\begin{equation}
A^{i_{1}}v_{1}\wedge A^{i_{2}}v_{2}\wedge\cdots\wedge A^{i_{n}}v_{n}%
\ \ \ \ \ \ \ \ \ \ \text{with }i_{j}\in\left\{  0,1,\ldots,g-1\right\}
\text{ and }v_{j}\in\left\{  e_{1},e_{2},\ldots,e_{n}\right\}  .
\label{pf.thm.finpowmat.char-int.form3}%
\end{equation}
Thus, the $\mathbb{K}$-module $U$ is spanned by all elements of the form
(\ref{pf.thm.finpowmat.char-int.form3}) (since it is spanned by all elements
of the form (\ref{pf.thm.finpowmat.char-int.form2})). Hence, the $\mathbb{K}%
$-module $U$ is finitely generated (since there are only finitely many
elements of the form (\ref{pf.thm.finpowmat.char-int.form3})). This proves
Claim 1.]

[\textit{Second proof of Claim 1:} Here is a more formal way of stating the
same proof. The $\mathbb{K}$-module $\mathbb{K}\left[  A\right]  $ is finitely
generated (since $\mathbb{K}\left[  A\right]  =\left\langle A^{0},A^{1}%
,\ldots,A^{g-1}\right\rangle _{\mathbb{K}}$), and so is the $\mathbb{K}%
$-module $\mathbb{K}^{n}$. Hence, the $\mathbb{K}$-module\footnote{In this
proof, the \textquotedblleft$\otimes$\textquotedblright\ symbol always means a
tensor product over $\mathbb{K}$.}
\[
M:=\underbrace{\mathbb{K}\left[  A\right]  \otimes\mathbb{K}\left[  A\right]
\otimes\cdots\otimes\mathbb{K}\left[  A\right]  }_{n\text{ times}}%
\otimes\underbrace{\mathbb{K}^{n}\otimes\mathbb{K}^{n}\otimes\cdots
\otimes\mathbb{K}^{n}}_{n\text{ times}}%
\]
is finitely generated as well\footnote{This is because of the following
general fact: If $A_{1},A_{2},\ldots,A_{m}$ are $m$ finitely generated
$\mathbb{K}$-modules, then the $\mathbb{K}$-module $A_{1}\otimes A_{2}%
\otimes\cdots\otimes A_{m}$ is also finitely generated. (We already proved
this fact while proving Theorem \ref{thm.finpowmat.fin-alg}.)}.

Now, for any $B_{1},B_{2},\ldots,B_{n}\in\mathbb{K}\left[  A\right]  $ and any
$v_{1},v_{2},\ldots,v_{n}\in\mathbb{K}^{n}$, the element $B_{1}v_{1}\wedge
B_{2}v_{2}\wedge\cdots\wedge B_{n}v_{n}\in C$ belongs to $U$ (since it is an
element of the form (\ref{pf.thm.finpowmat.char-int.form1})). Thus, the
$\mathbb{K}$-linear map%
\begin{align*}
\pi:M  &  \rightarrow U,\\
B_{1}\otimes B_{2}\otimes\cdots\otimes B_{n}\otimes v_{1}\otimes v_{2}%
\otimes\cdots\otimes v_{n}  &  \mapsto B_{1}v_{1}\wedge B_{2}v_{2}\wedge
\cdots\wedge B_{n}v_{n}%
\end{align*}
is well-defined. This $\mathbb{K}$-linear map $\pi$ is surjective (since each
element of the form (\ref{pf.thm.finpowmat.char-int.form1}) belongs to its
image\footnote{Namely:
\[
B_{1}v_{1}\wedge B_{2}v_{2}\wedge\cdots\wedge B_{n}v_{n}=\pi\left(
B_{1}\otimes B_{2}\otimes\cdots\otimes B_{n}\otimes v_{1}\otimes v_{2}%
\otimes\cdots\otimes v_{n}\right)  .
\]
}). Thus, Lemma \ref{lem.finpowmat.finger-sur} (applied to $N=U$ and $f=\pi$)
shows that the $\mathbb{K}$-module $U$ is finitely generated (since the
$\mathbb{K}$-module $M$ is finitely generated). This proves Claim 1 again.]

Now, fix $k\in\mathbb{N}$. Let $a_{k}\in\mathbb{L}$ be the coefficient of
$t^{k}$ in the characteristic polynomial $\chi_{A}\in\mathbb{L}\left[
t\right]  $. We are going to show the following:

\begin{statement}
\textit{Claim 2:} We have $a_{k}U\subseteq U$.
\end{statement}

[\textit{Proof of Claim 2:} It suffices to show that $a_{k}u\in U$ for each
$u\in U$. So let us fix $u\in U$; thus we must prove that $a_{k}u\in U$.

We can WLOG assume that $u$ is an element of the form
(\ref{pf.thm.finpowmat.char-int.form1}) (since $U$ is spanned by elements of
this form). So assume this. Then,%
\[
u=B_{1}v_{1}\wedge B_{2}v_{2}\wedge\cdots\wedge B_{n}v_{n}%
\ \ \ \ \ \ \ \ \ \ \text{for some }B_{j}\in\mathbb{K}\left[  A\right]  \text{
and }v_{j}\in\mathbb{K}^{n}.
\]
Consider these $B_{j}$ and these $v_{j}$. Note that for each $n$-tuple
$\left(  i_{1},i_{2},\ldots,i_{n}\right)  \in\left\{  0,1\right\}  ^{n}$, we
have%
\begin{equation}
A^{i_{1}}B_{1}v_{1}\wedge A^{i_{2}}B_{2}v_{2}\wedge\cdots\wedge A^{i_{n}}%
B_{n}v_{n}\in U, \label{pf.thm.finpowmat.char-int.c2.pf.1}%
\end{equation}
since $A^{i_{1}}B_{1}v_{1}\wedge A^{i_{2}}B_{2}v_{2}\wedge\cdots\wedge
A^{i_{n}}B_{n}v_{n}$ is an element of the form
(\ref{pf.thm.finpowmat.char-int.form1}) (because each $j\in\left\{
1,2,\ldots,n\right\}  $ satisfies $\underbrace{A^{i_{j}}}_{\in\mathbb{K}%
\left[  A\right]  }\underbrace{B_{j}}_{\in\mathbb{K}\left[  A\right]  }%
\in\mathbb{K}\left[  A\right]  \cdot\mathbb{K}\left[  A\right]  \subseteq
\mathbb{K}\left[  A\right]  $).

Multiplying both sides of the equality $u=B_{1}v_{1}\wedge B_{2}v_{2}%
\wedge\cdots\wedge B_{n}v_{n}$ with $a_{k}\in\mathbb{L}$, we obtain%
\begin{align*}
a_{k}u  &  =a_{k}\cdot B_{1}v_{1}\wedge B_{2}v_{2}\wedge\cdots\wedge
B_{n}v_{n}\\
&  =\left(  -1\right)  ^{n-k}\sum_{\substack{i_{1},i_{2},\ldots,i_{n}%
\in\left\{  0,1\right\}  ;\\i_{1}+i_{2}+\cdots+i_{n}=n-k}}\underbrace{A^{i_{1}%
}B_{1}v_{1}\wedge A^{i_{2}}B_{2}v_{2}\wedge\cdots\wedge A^{i_{n}}B_{n}v_{n}%
}_{\substack{\in U\\\text{(by (\ref{pf.thm.finpowmat.char-int.c2.pf.1}))}}}\\
&  \ \ \ \ \ \ \ \ \ \ \left(  \text{by Proposition
\ref{prop.finpowmat.almk-coeff}, applied to }\mathbb{L}\text{ and }B_{j}%
v_{j}\text{ instead of }\mathbb{K}\text{ and }w_{j}\right) \\
&  \in\left(  -1\right)  ^{n-k}\sum_{\substack{i_{1},i_{2},\ldots,i_{n}%
\in\left\{  0,1\right\}  ;\\i_{1}+i_{2}+\cdots+i_{n}=n-k}}U\subseteq U
\end{align*}
(since $U$ is a $\mathbb{K}$-module). This completes the proof of Claim 2.]

\begin{statement}
\textit{Claim 3:} Every $v\in\mathbb{L}$ satisfying $vU=0$ satisfies $v=0$.
\end{statement}

[\textit{Proof of Claim 3:} Let $v\in\mathbb{L}$ satisfy $vU=0$. We must prove
that $v=0$.

Let $\left(  e_{1},e_{2},\ldots,e_{n}\right)  $ be the standard basis of the
$\mathbb{L}$-module $\mathbb{L}^{n}$ (so that $e_{i}$ is the vector with a $1$
in its $i$-th entry and $0$ everywhere else). Thus, $\left(  e_{1}%
,e_{2},\ldots,e_{n}\right)  $ is also the standard basis of the $\mathbb{K}%
$-module $\mathbb{K}^{n}$ (since we are embedding $\mathbb{K}^{n}$ into
$\mathbb{L}^{n}$ in the usual way). Thus, $e_{j}\in\mathbb{K}^{n}$ for each
$j\in\left\{  1,2,\ldots,n\right\}  $. Hence, the element $e_{1}\wedge
e_{2}\wedge\cdots\wedge e_{n}$ of $C$ has the form
(\ref{pf.thm.finpowmat.char-int.form1}) (namely, for $B_{j}=I_{n}$ and
$v_{j}=e_{j}$). Hence, this element belongs to $U$ (by the definition of $U$).
In other words, $e_{1}\wedge e_{2}\wedge\cdots\wedge e_{n}\in U$. Thus,
$v\underbrace{\left(  e_{1}\wedge e_{2}\wedge\cdots\wedge e_{n}\right)  }_{\in
U}\in vU=0$, so that $v\left(  e_{1}\wedge e_{2}\wedge\cdots\wedge
e_{n}\right)  =0$.

But recall that the $n$-tuple $\left(  e_{1},e_{2},\ldots,e_{n}\right)  $ is a
basis of the $\mathbb{L}$-module $\mathbb{L}^{n}$. Hence, Lemma
\ref{lem.finpowmat.ext-rankn} (applied to $M=\mathbb{L}^{n}$ and $b_{j}=e_{j}%
$) shows that the $1$-tuple $\left(  e_{1}\wedge e_{2}\wedge\cdots\wedge
e_{n}\right)  $ is a basis of the $\mathbb{L}$-module $\Lambda^{n}%
\underbrace{\left(  \mathbb{L}^{n}\right)  }_{=V}=\Lambda^{n}V$. Hence, this
$1$-tuple is $\mathbb{L}$-linearly independent. In other words, every
$w\in\mathbb{L}$ satisfying $w\left(  e_{1}\wedge e_{2}\wedge\cdots\wedge
e_{n}\right)  =0$ satisfies $w=0$. Applying this to $w=v$, we obtain $v=0$
(since $v\left(  e_{1}\wedge e_{2}\wedge\cdots\wedge e_{n}\right)  =0$). This
proves Claim 3.]

Now, Corollary \ref{cor.finpowmat.det-crit} can be applied to $u=a_{k}$ (since
Claim 1, Claim 2 and Claim 3 ensure that the assumptions of Corollary
\ref{cor.finpowmat.det-crit} are satisfied). Thus, we conclude that $a_{k}%
\in\mathbb{L}$ is integral over $\mathbb{K}$. In other words, the coefficient
of $t^{k}$ in the characteristic polynomial $\chi_{A}\in\mathbb{L}\left[
t\right]  $ is integral over $\mathbb{K}$ (since $a_{k}$ was defined to be
this coefficient).

Now, forget that we fixed $k$. We thus have shown that for each $k\in
\mathbb{N}$, the coefficient of $t^{k}$ in the characteristic polynomial
$\chi_{A}\in\mathbb{L}\left[  t\right]  $ is integral over $\mathbb{K}$. This
proves Theorem \ref{thm.finpowmat.char-int}.
\end{proof}

\begin{corollary}
\label{cor.finpowmat.char-finger}Let $\mathbb{K}$ be a commutative ring. Let
$n\in\mathbb{N}$. Let $\mathbb{L}$ be a commutative $\mathbb{K}$-algebra. Let
$A$ be an $n\times n$-matrix over $\mathbb{L}$. Assume that $A$ is integral
over $\mathbb{K}$ (as an element of the $\mathbb{K}$-algebra $\mathbb{L}%
^{n\times n}$). Let $\mathbb{M}$ be the $\mathbb{K}$-subalgebra of
$\mathbb{L}$ generated by the coefficients of the characteristic polynomial
$\chi_{A}\in\mathbb{L}\left[  t\right]  $. Then, $\mathbb{M}$ is a finitely
generated $\mathbb{K}$-module.
\end{corollary}

\begin{proof}
[Proof of Corollary \ref{cor.finpowmat.char-finger}.]Let $u_{1},u_{2}%
,\ldots,u_{m}$ be the coefficients of the polynomial $\chi_{A}$. These
coefficients $u_{1},u_{2},\ldots,u_{m}$ are integral over $\mathbb{K}$ (by
Theorem \ref{thm.finpowmat.char-int}), and generate $\mathbb{M}$ as a
$\mathbb{K}$-algebra (by the definition of $\mathbb{M}$); thus, in particular,
they are elements of $\mathbb{M}$. Hence, Theorem \ref{thm.finpowmat.fin-alg}
(applied to $\mathbb{M}$ instead of $\mathbb{L}$) yields that the $\mathbb{K}%
$-module $\mathbb{M}$ is finitely generated. This proves Corollary
\ref{cor.finpowmat.char-finger}.
\end{proof}

\subsection{Two finiteness lemmas}

We need two more lemmas about finite generation of certain modules:

\begin{lemma}
\label{lem.finpowmat.fin-gen-over-fin-ring}Let $\mathbb{K}$ be a finite
commutative ring. Let $M$ be a finitely generated $\mathbb{K}$-module. Then,
$M$ is finite (as a set).
\end{lemma}

\begin{proof}
[Proof of Lemma \ref{lem.finpowmat.fin-gen-over-fin-ring}.]The $\mathbb{K}%
$-module $M$ is finitely generated. In other words, there exist finitely many
vectors $a_{1},a_{2},\ldots,a_{m}\in M$ that generate $M$ as a $\mathbb{K}%
$-module. Consider these $a_{1},a_{2},\ldots,a_{m}$. Thus, each element of $M$
is a $\mathbb{K}$-linear combination of $a_{1},a_{2},\ldots,a_{m}$ (since
$a_{1},a_{2},\ldots,a_{m}$ generate $M$ as a $\mathbb{K}$-module).

There exist only finitely many $\mathbb{K}$-linear combinations of
$a_{1},a_{2},\ldots,a_{m}$ (because a $\mathbb{K}$-linear combination
$\lambda_{1}a_{1}+\lambda_{2}a_{2}+\cdots+\lambda_{m}a_{m}$ of $a_{1}%
,a_{2},\ldots,a_{m}$ is uniquely determined by choosing its $m$ coefficients
$\lambda_{1},\lambda_{2},\ldots,\lambda_{m}\in\mathbb{K}$, but each of these
$m$ coefficients can be chosen in only finitely many ways\footnote{since
$\mathbb{K}$ is finite}). Hence, there are only finitely many elements of $M$
(since each element of $M$ is a $\mathbb{K}$-linear combination of
$a_{1},a_{2},\ldots,a_{m}$). In other words, $M$ is finite. This proves Lemma
\ref{lem.finpowmat.fin-gen-over-fin-ring}.
\end{proof}

\begin{lemma}
\label{lem.finpowmat.monic-fin}Let $\mathbb{K}$ be a commutative ring. Let
$f\in\mathbb{K}\left[  t\right]  $ be a monic polynomial. Then, the
$\mathbb{K}$-module $\mathbb{K}\left[  t\right]  /\left(  f\right)  $ is
finitely generated.
\end{lemma}

\begin{proof}
[First proof of Lemma \ref{lem.finpowmat.monic-fin}.]Much more can be said:
For each $u\in\mathbb{K}\left[  t\right]  $, we let $\overline{u}$ denote the
projection of $u$ onto $\mathbb{K}\left[  t\right]  /\left(  f\right)  $.
Then, the $\mathbb{K}$-module $\mathbb{K}\left[  t\right]  /\left(  f\right)
$ is free with basis $\left(  \overline{t^{0}},\overline{t^{1}},\ldots
,\overline{t^{n-1}}\right)  $, where $n=\deg f$. This is a well-known
fact\footnote{See, e.g., \cite[Chapter III, Proposition 4.6]{Aluffi09} for an
equivalent version of this fact (restated in terms of an isomorphism
$\mathbb{K}\left[  t\right]  /\left(  f\right)  \rightarrow\mathbb{K}^{\oplus
n}$).} and follows easily from \textquotedblleft Euclidean division of
polynomials\textquotedblright. Of course, this entails that the $\mathbb{K}%
$-module $\mathbb{K}\left[  t\right]  /\left(  f\right)  $ is finitely
generated. This proves Lemma \ref{lem.finpowmat.monic-fin}.
\end{proof}

\begin{proof}
[Second proof of Lemma \ref{lem.finpowmat.monic-fin}.]For each $u\in
\mathbb{K}\left[  t\right]  $, we let $\overline{u}$ denote the projection of
$u$ onto $\mathbb{K}\left[  t\right]  /\left(  f\right)  $. The element
$\overline{t}$ of $\mathbb{K}\left[  t\right]  /\left(  f\right)  $ satisfies
$f\left(  \overline{t}\right)  =\overline{f\left(  t\right)  }=0$ (since
$f\left(  t\right)  =f\in\left(  f\right)  $). Hence, $\overline{t}%
\in\mathbb{K}\left[  t\right]  /\left(  f\right)  $ is integral over
$\mathbb{K}$ (by the definition of \textquotedblleft
integral\textquotedblright), since the polynomial $f\in\mathbb{K}\left[
t\right]  $ is monic.

Moreover, the $\mathbb{K}$-algebra $\mathbb{K}\left[  t\right]  $ is generated
by $t$; thus, its quotient $\mathbb{K}$-algebra $\mathbb{K}\left[  t\right]
/\left(  f\right)  $ is generated by $\overline{t}$. Hence, Theorem
\ref{thm.finpowmat.fin-alg} (applied to $\mathbb{L}=\mathbb{K}\left[
t\right]  /\left(  f\right)  $, $m=1$ and $u_{1}=\overline{t}$) yields that
the $\mathbb{K}$-module $\mathbb{K}\left[  t\right]  /\left(  f\right)  $ is
finitely generated. This proves Lemma \ref{lem.finpowmat.monic-fin}.
\end{proof}

\subsection{Proof of Theorem \ref{thm.finpowmat.main}}

The following fact will bring us very close to Theorem
\ref{thm.finpowmat.main}:

\begin{proposition}
\label{prop.finpowmat.char-crit}Let $\mathbb{K}$ be a finite commutative ring.
Let $n\in\mathbb{N}$. Let $\mathbb{L}$ be a commutative $\mathbb{K}$-algebra.
Let $A$ be an $n\times n$-matrix over $\mathbb{L}$. Then, the following three
assertions are equivalent:

\begin{itemize}
\item \textit{Assertion }$\mathcal{U}$\textit{:} The set $\left\{  A^{0}%
,A^{1},A^{2},\ldots\right\}  $ is finite.

\item \textit{Assertion }$\mathcal{V}$\textit{:} The matrix $A$ is integral
over $\mathbb{K}$ (as an element of the $\mathbb{K}$-algebra $\mathbb{L}%
^{n\times n}$).

\item \textit{Assertion }$\mathcal{W}$\textit{:} There exists a positive
integer $m$ such that the polynomial $t^{2m}-t^{m}$ is a multiple of $\chi
_{A}$ in $\mathbb{L}\left[  t\right]  $.
\end{itemize}
\end{proposition}

\begin{proof}
[Proof of Proposition \ref{prop.finpowmat.char-crit}.]We shall prove the
implications $\mathcal{U}\Longrightarrow\mathcal{V}$ and $\mathcal{V}%
\Longrightarrow\mathcal{W}$ and $\mathcal{W}\Longrightarrow\mathcal{U}$:

\textit{Proof of the implication }$\mathcal{U}\Longrightarrow\mathcal{V}%
$\textit{:} Assume that Assertion $\mathcal{U}$ holds. We must prove that
Assertion $\mathcal{V}$ holds.

The set $\left\{  A^{0},A^{1},A^{2},\ldots\right\}  $ is closed under
multiplication. Thus, this set (equipped with multiplication) is a semigroup.
Furthermore, this set is finite (since Assertion $\mathcal{U}$ holds), and
thus is a finite semigroup. Hence, Theorem \ref{thm.finpowmat.finmon} (applied
to $M=\left\{  A^{0},A^{1},A^{2},\ldots\right\}  $ and $a=A$) shows that there
exists a positive integer $m$ such that $A^{m}=A^{2m}$ (since $A=A^{1}%
\in\left\{  A^{0},A^{1},A^{2},\ldots\right\}  $). Consider this $m$. Let
$g\in\mathbb{K}\left[  t\right]  $ be the polynomial $t^{2m}-t^{m}$. Then, $g$
is monic (since $m>0$) and satisfies $g\left(  A\right)  =A^{2m}-A^{m}=0$
(since $A^{m}=A^{2m}$). Hence, there exists a monic polynomial $f\in
\mathbb{K}\left[  t\right]  $ such that $f\left(  A\right)  =0$ (namely,
$f=g$). In other words, $A$ is integral over $\mathbb{K}$. In other words,
Assertion $\mathcal{V}$ holds. This proves the implication $\mathcal{U}%
\Longrightarrow\mathcal{V}$.

\textit{Proof of the implication }$\mathcal{V}\Longrightarrow\mathcal{W}%
$\textit{:} Assume that Assertion $\mathcal{V}$ holds. We must prove that
Assertion $\mathcal{W}$ holds.

We have assumed that Assertion $\mathcal{V}$ holds. In other words, $A$ is
integral over $\mathbb{K}$. Let $\mathbb{M}$ be the $\mathbb{K}$-subalgebra of
$\mathbb{L}$ generated by the coefficients of the characteristic polynomial
$\chi_{A}\in\mathbb{L}\left[  t\right]  $. Then, the coefficients of $\chi
_{A}$ belong to this $\mathbb{K}$-subalgebra $\mathbb{M}$; thus, $\chi_{A}%
\in\mathbb{M}\left[  t\right]  $. Furthermore, Corollary
\ref{cor.finpowmat.char-finger} shows that $\mathbb{M}$ is a finitely
generated $\mathbb{K}$-module. Thus, Lemma
\ref{lem.finpowmat.fin-gen-over-fin-ring} (applied to $M=\mathbb{M}$) shows
that $\mathbb{M}$ is finite (as a set).

The polynomial $\chi_{A}\in\mathbb{M}\left[  t\right]  $ is monic. Thus, the
$\mathbb{M}$-module $\mathbb{M}\left[  t\right]  /\left(  \chi_{A}\right)  $
is finitely generated (by Lemma \ref{lem.finpowmat.monic-fin}, applied to
$\mathbb{M}$ and $\chi_{A}$ instead of $\mathbb{K}$ and $f$). Thus, Lemma
\ref{lem.finpowmat.fin-gen-over-fin-ring} (applied to $\mathbb{M}$ and
$\mathbb{M}\left[  t\right]  /\left(  \chi_{A}\right)  $ instead of
$\mathbb{K}$ and $M$) shows that $\mathbb{M}\left[  t\right]  /\left(
\chi_{A}\right)  $ is finite (as a set). This ring $\mathbb{M}\left[
t\right]  /\left(  \chi_{A}\right)  $ becomes a semigroup when equipped with
its multiplication. This semigroup $\mathbb{M}\left[  t\right]  /\left(
\chi_{A}\right)  $ is finite (since we have just shown that $\mathbb{M}\left[
t\right]  /\left(  \chi_{A}\right)  $ is finite).

For each $u\in\mathbb{M}\left[  t\right]  $, we let $\overline{u}$ denote the
projection of $u$ onto $\mathbb{M}\left[  t\right]  /\left(  \chi_{A}\right)
$. Then, Theorem \ref{thm.finpowmat.finmon} (applied to $M=\mathbb{M}\left[
t\right]  /\left(  \chi_{A}\right)  $ and $a=\overline{t}$) yields that there
exists a positive integer $m$ such that $\overline{t}^{m}=\overline{t}^{2m}$.
Consider this $m$. Then, $\overline{t^{m}}=\overline{t}^{m}=\overline{t}%
^{2m}=\overline{t^{2m}}$; in other words, we have the congruence $t^{m}\equiv
t^{2m}\operatorname{mod}\chi_{A}$ in the ring $\mathbb{M}\left[  t\right]  $.
In other words, the polynomial $t^{2m}-t^{m}$ is a multiple of $\chi_{A}$ in
$\mathbb{M}\left[  t\right]  $. Hence, the polynomial $t^{2m}-t^{m}$ is a
multiple of $\chi_{A}$ in $\mathbb{L}\left[  t\right]  $ (since $\mathbb{M}%
\left[  t\right]  $ is a subring of $\mathbb{L}\left[  t\right]  $). Thus,
Assertion $\mathcal{W}$ holds. This proves the implication $\mathcal{V}%
\Longrightarrow\mathcal{W}$.

\textit{Proof of the implication }$\mathcal{W}\Longrightarrow\mathcal{U}%
$\textit{:} Assume that Assertion $\mathcal{W}$ holds. We must prove that
Assertion $\mathcal{U}$ holds.

We have assumed that Assertion $\mathcal{W}$ holds. In other words, there
exists a positive integer $m$ such that the polynomial $t^{2m}-t^{m}$ is a
multiple of $\chi_{A}$ in $\mathbb{L}\left[  t\right]  $. Consider this $m$.
Note that $2m$ and $m$ are positive integers satisfying $2m>m$. Consider the
ring $\mathbb{L}^{n\times n}$ as a semigroup (equipped with its multiplication).

Now, there exists a polynomial $g\in\mathbb{L}\left[  t\right]  $ such that
$t^{2m}-t^{m}=\chi_{A}\cdot g$ (since the polynomial $t^{2m}-t^{m}$ is a
multiple of $\chi_{A}$ in $\mathbb{L}\left[  t\right]  $). Consider this $g$.
Evaluating both sides of the polynomial identity $t^{2m}-t^{m}=\chi_{A}\cdot
g$ at $A$, we obtain%
\[
A^{2m}-A^{m}=\underbrace{\chi_{A}\left(  A\right)  }_{\substack{=0\\\text{(by
the Cayley--Hamilton}\\\text{theorem)}}}\cdot g\left(  A\right)  =0.
\]
In other words, $A^{2m}=A^{m}$. Hence, Proposition \ref{prop.finpowmat.sg-fin}
(applied to $M=\mathbb{L}^{n\times n}$, $a=A$, $p=2m$ and $q=m$) yields
$\left\{  A^{1},A^{2},A^{3},\ldots\right\}  =\left\{  A^{1},A^{2}%
,\ldots,A^{2m-1}\right\}  $. Thus, the set $\left\{  A^{1},A^{2},A^{3}%
,\ldots\right\}  $ is finite (since the set $\left\{  A^{1},A^{2}%
,\ldots,A^{2m-1}\right\}  $ is clearly finite). Hence, the set $\left\{
A^{0},A^{1},A^{2},\ldots\right\}  $ is also finite (since this set is
$\left\{  A^{1},A^{2},A^{3},\ldots\right\}  \cup\left\{  A^{0}\right\}  $). In
other words, Assertion $\mathcal{U}$ holds. This proves the implication
$\mathcal{W}\Longrightarrow\mathcal{U}$.

We have now proven all three implications $\mathcal{U}\Longrightarrow
\mathcal{V}$ and $\mathcal{V}\Longrightarrow\mathcal{W}$ and $\mathcal{W}%
\Longrightarrow\mathcal{U}$. Hence, $\mathcal{U}\Longleftrightarrow
\mathcal{V}\Longleftrightarrow\mathcal{W}$. This proves Proposition
\ref{prop.finpowmat.char-crit}.
\end{proof}

We can now easily prove Theorem \ref{thm.finpowmat.main}:

\begin{proof}
[Proof of Theorem \ref{thm.finpowmat.main}.]Proposition
\ref{prop.finpowmat.char-crit} (or, more precisely, the equivalence of the
Assertions $\mathcal{U}$ and $\mathcal{W}$ in this proposition) shows that the
set $\left\{  A^{0},A^{1},A^{2},\ldots\right\}  $ is finite if and only if
there exists a positive integer $m$ such that the polynomial $t^{2m}-t^{m}$ is
a multiple of $\chi_{A}$ in $\mathbb{L}\left[  t\right]  $. In other words, we
have the logical equivalence%
\begin{align*}
&  \ \left(  \text{the set }\left\{  A^{0},A^{1},A^{2},\ldots\right\}  \text{
is finite}\right) \\
&  \Longleftrightarrow\ \left(  \text{there exists a positive integer }m\text{
such that the}\right. \\
&  \ \ \ \ \ \ \ \ \ \ \left.  \text{polynomial }t^{2m}-t^{m}\text{ is a
multiple of }\chi_{A}\text{ in }\mathbb{L}\left[  t\right]  \right)  .
\end{align*}
The same argument (applied to $B$ instead of $A$) yields the logical
equivalence%
\begin{align*}
&  \ \left(  \text{the set }\left\{  B^{0},B^{1},B^{2},\ldots\right\}  \text{
is finite}\right) \\
&  \Longleftrightarrow\ \left(  \text{there exists a positive integer }m\text{
such that the}\right. \\
&  \ \ \ \ \ \ \ \ \ \ \left.  \text{polynomial }t^{2m}-t^{m}\text{ is a
multiple of }\chi_{B}\text{ in }\mathbb{L}\left[  t\right]  \right)  .
\end{align*}
But the right hand sides of these two equivalences are equivalent (since
$\chi_{A}=\chi_{B}$). Hence, their left hand sides are equivalent as well. In
other words, we have the equivalence%
\[
\left(  \text{the set }\left\{  A^{0},A^{1},A^{2},\ldots\right\}  \text{ is
finite}\right)  \ \Longleftrightarrow\ \left(  \text{the set }\left\{
B^{0},B^{1},B^{2},\ldots\right\}  \text{ is finite}\right)  .
\]
This proves Theorem \ref{thm.finpowmat.main}.
\end{proof}

\subsection{Digression: Traces of nilpotent matrices}

While this is unrelated to Theorem \ref{thm.finpowmat.main}, let us illustrate
the usefulness of Theorem \ref{thm.finpowmat.char-int} on a different application:

\begin{corollary}
\label{cor.finpowmat.tr-int}Let $\mathbb{K}$ be a commutative ring. Let
$n\in\mathbb{N}$. Let $\mathbb{L}$ be a commutative $\mathbb{K}$-algebra. Let
$A$ be an $n\times n$-matrix over $\mathbb{L}$. Assume that $A$ is integral
over $\mathbb{K}$ (as an element of the $\mathbb{K}$-algebra $\mathbb{L}%
^{n\times n}$). Then, the trace $\operatorname*{Tr}A\in\mathbb{L}$ is integral
over $\mathbb{K}$.
\end{corollary}

\begin{proof}
[Proof of Corollary \ref{cor.finpowmat.tr-int}.]The coefficient of $t^{n-1}$
in the characteristic polynomial $\chi_{A}\in\mathbb{L}\left[  t\right]  $ is
known to be $-\operatorname*{Tr}A$. But on the other hand, the coefficient of
$t^{n-1}$ in the characteristic polynomial $\chi_{A}\in\mathbb{L}\left[
t\right]  $ is integral over $\mathbb{K}$ (by Theorem
\ref{thm.finpowmat.char-int}). In other words, $-\operatorname*{Tr}A$ is
integral over $\mathbb{K}$ (since this coefficient is $-\operatorname*{Tr}A$).
Thus, $\operatorname*{Tr}A$ is integral over $\mathbb{K}$ (because it is easy
to see that if $u\in\mathbb{L}$ is an element such that $-u$ is integral over
$\mathbb{K}$, then $u$ is integral over $\mathbb{K}$). This proves Corollary
\ref{cor.finpowmat.tr-int}.
\end{proof}

\begin{corollary}
\label{cor.finpowmat.tr-nil}Let $\mathbb{K}$ be a commutative ring. Let
$n\in\mathbb{N}$. Let $A\in\mathbb{K}^{n\times n}$ be a nilpotent matrix.
Then, its trace $\operatorname*{Tr}A$ is nilpotent.
\end{corollary}

This is a generalization of the classical result that a nilpotent square
matrix over a field must have trace $0$.

There is actually a stronger version of Corollary \ref{cor.finpowmat.tr-nil},
which says that if $A^{m+1}=0$ for some $m\in\mathbb{N}$, then $\left(
\operatorname*{Tr}A\right)  ^{mn+1}=0$ (see \cite{Zeilbe93}, and \cite[Theorem
1.7 (i)]{Almkvi73} for an even more general result). We shall only prove
Corollary \ref{cor.finpowmat.tr-nil}. The proof relies on the following neat
fact, which reveals nilpotence to be an instance of integrality:

\begin{lemma}
\label{lem.finpowmat.nil-int}Let $\mathbb{K}$ be a commutative ring. Let
$\mathbb{L}$ be a $\mathbb{K}$-algebra. Let $a\in\mathbb{L}$. Consider the
polynomial ring $\mathbb{L}\left[  t\right]  $. Then, $a$ is nilpotent if and
only if $at\in\mathbb{L}\left[  t\right]  $ is integral over $\mathbb{K}$.
\end{lemma}

\begin{proof}
[Proof of Lemma \ref{lem.finpowmat.nil-int}.]$\Longrightarrow:$ Assume that
$a$ is nilpotent. Thus, there exists some $m\in\mathbb{N}$ such that $a^{m}%
=0$. Consider this $m$. Now, let $f$ be the polynomial $t^{m}\in
\mathbb{K}\left[  t\right]  $. Then, this polynomial $f$ is monic and
satisfies $f\left(  at\right)  =\left(  at\right)  ^{m}=\underbrace{a^{m}%
}_{=0}t^{m}=0$. Hence, $at\in\mathbb{L}\left[  t\right]  $ is integral over
$\mathbb{K}$ (by the definition of \textquotedblleft
integral\textquotedblright). This proves the \textquotedblleft$\Longrightarrow
$\textquotedblright\ direction of Lemma \ref{lem.finpowmat.nil-int}.

$\Longleftarrow:$ Assume that $at\in\mathbb{L}\left[  t\right]  $ is integral
over $\mathbb{K}$. Thus, there exists a monic polynomial $f\in\mathbb{K}%
\left[  t\right]  $ such that $f\left(  at\right)  =0$. Consider this $f$.

Write the polynomial $f$ in the form $f=k_{0}t^{0}+k_{1}t^{1}+\cdots
+k_{n}t^{n}$, where $n=\deg f$ and $k_{0},k_{1},\ldots,k_{n}\in\mathbb{K}$.
Then, $k_{n}=1$ (since $f$ is monic). Furthermore, from $f=k_{0}t^{0}%
+k_{1}t^{1}+\cdots+k_{n}t^{n}$, we obtain
\[
f\left(  at\right)  =k_{0}\left(  at\right)  ^{0}+k_{1}\left(  at\right)
^{1}+\cdots+k_{n}\left(  at\right)  ^{n}=k_{0}a^{0}t^{0}+k_{1}a^{1}%
t^{1}+\cdots+k_{n}a^{n}t^{n}.
\]
Comparing this with $f\left(  at\right)  =0$, we obtain%
\[
k_{0}a^{0}t^{0}+k_{1}a^{1}t^{1}+\cdots+k_{n}a^{n}t^{n}=0.
\]
This is an equality between two polynomials in $\mathbb{L}\left[  t\right]  $.
Comparing the coefficients of $t^{n}$ on both sides of this equality, we
conclude that $k_{n}a^{n}=0$. Since $k_{n}=1$, this rewrites as $a^{n}=0$.
Hence, $a$ is nilpotent. This proves the \textquotedblleft$\Longleftarrow
$\textquotedblright\ direction of Lemma \ref{lem.finpowmat.nil-int}.
\end{proof}

\begin{proof}
[Proof of Corollary \ref{cor.finpowmat.tr-nil}.]Consider the polynomial ring
$\mathbb{K}^{n\times n}\left[  t\right]  $. This is a $\mathbb{K}$-algebra
which may be noncommutative, but $t$ belongs to its center. The element
$A\in\mathbb{K}^{n\times n}$ is nilpotent. Hence, the \textquotedblleft%
$\Longrightarrow$\textquotedblright\ direction of Lemma
\ref{lem.finpowmat.nil-int} (applied to $\mathbb{K}^{n\times n}$ and $A$
instead of $\mathbb{L}$ and $a$) yields that $At\in\mathbb{K}^{n\times
n}\left[  t\right]  $ is integral over $\mathbb{K}$. Now, identify the
$\mathbb{K}$-algebra $\mathbb{K}^{n\times n}\left[  t\right]  $ with $\left(
\mathbb{K}\left[  t\right]  \right)  ^{n\times n}$ in the usual way (i.e., in
the same way as one identifies polynomial matrices with polynomials over
matrix rings in linear algebra). Thus, $At\in\mathbb{K}^{n\times n}\left[
t\right]  =\left(  \mathbb{K}\left[  t\right]  \right)  ^{n\times n}$. The
trace of this matrix $At$ is $\operatorname*{Tr}\left(  At\right)  =\left(
\operatorname*{Tr}A\right)  \cdot t$ (since the trace is linear).

This matrix $At$ is integral over $\mathbb{K}$ (as we know). Hence, Corollary
\ref{cor.finpowmat.tr-int} (applied to $\mathbb{K}\left[  t\right]  $ and $At$
instead of $\mathbb{L}$ and $A$) yields that the trace $\operatorname*{Tr}%
\left(  At\right)  \in\mathbb{K}\left[  t\right]  $ is integral over
$\mathbb{K}$. In other words, $\left(  \operatorname*{Tr}A\right)  \cdot
t\in\mathbb{K}\left[  t\right]  $ is integral over $\mathbb{K}$ (since
$\operatorname*{Tr}\left(  At\right)  =\left(  \operatorname*{Tr}A\right)
\cdot t$). Hence, the \textquotedblleft$\Longleftarrow$\textquotedblright%
\ direction of Lemma \ref{lem.finpowmat.nil-int} (applied to $\mathbb{L}%
=\mathbb{K}$ and $a=\operatorname*{Tr}A$) yields that $\operatorname*{Tr}A$ is
nilpotent. This proves Corollary \ref{cor.finpowmat.tr-nil}.
\end{proof}

\subsection{\label{sec.finpowmat.almk-coeff-pf2}Appendix:
Second proof of Proposition \ref{prop.finpowmat.almk-coeff}}

Let us also sketch a second proof of Proposition
\ref{prop.finpowmat.almk-coeff}, which avoids exterior powers over
$\mathbb{K}\left[  t\right]  $ but instead uses determinantal identities.

In this section, we shall use the following notations: Fix a commutative ring
$\mathbb{K}$ and an $n\in\mathbb{N}$. We let $\left[  n\right]  $ denote the
set $\left\{  1,2,\ldots,n\right\}  $.

Furthermore, if $A\in\mathbb{K}^{n\times n}$ is an $n\times n$-matrix, and if
$U$ and $V$ are two subsets of $\left[  n\right]  $, then $\operatorname*{sub}%
\nolimits_{U}^{V}A$ shall denote the $\left\vert U\right\vert \times\left\vert
V\right\vert $-matrix obtained from $A$ by removing all the rows whose
indices\footnote{The \textit{index} of a row in a matrix means the number
saying which row it is. In other words, the index of the $i$-th row in a
matrix means the number $i$. Similar terminology is used for columns.} don't
belong to $U$ and all the columns whose indices don't belong to $V$. (Formally
speaking, this $\operatorname*{sub}\nolimits_{U}^{V}A$ is defined by
\[
\operatorname*{sub}\nolimits_{U}^{V}A=\left(  a_{u_{i},v_{j}}\right)  _{1\leq
i\leq p,\ 1\leq j\leq q},
\]
where we have written the matrix $A$ in the form $A=\left(  a_{i,j}\right)
_{1\leq i\leq n,\ 1\leq j\leq n}$ and where we have written the two subsets
$U$ and $V$ as $U=\left\{  u_{1}<u_{2}<\cdots<u_{p}\right\}  $ and $V=\left\{
v_{1}<v_{2}<\cdots<v_{q}\right\}  $.)

For example,%
\[
\operatorname*{sub}\nolimits_{\left\{  2,4\right\}  }^{\left\{  1,4\right\}
}\left(
\begin{array}
[c]{cccc}%
a & b & c & d\\
a^{\prime} & b^{\prime} & c^{\prime} & d^{\prime}\\
a^{\prime\prime} & b^{\prime\prime} & c^{\prime\prime} & d^{\prime\prime}\\
a^{\prime\prime\prime} & b^{\prime\prime\prime} & c^{\prime\prime\prime} &
d^{\prime\prime\prime}%
\end{array}
\right)  =\left(
\begin{array}
[c]{cc}%
a^{\prime} & d^{\prime}\\
a^{\prime\prime\prime} & d^{\prime\prime\prime}%
\end{array}
\right)  .
\]

We shall now prove (or cite proofs of) a sequence of basic properties of
submatrices and their determinants.

\begin{lemma}
\label{lem.finpowmat.minors-1}Let $A\in\mathbb{K}^{n\times n}$ be a matrix.
For each $x\in\mathbb{K}$, we have%
\[
\det\left(  A+xI_{n}\right)  =\sum_{P\subseteq\left[  n\right]  }\det\left(
\operatorname*{sub}\nolimits_{P}^{P}A\right)  x^{n-\left\vert P\right\vert }.
\]

\end{lemma}

Lemma \ref{lem.finpowmat.minors-1} is precisely the first equality sign of
\cite[Corollary 6.164]{detnotes} (up to notation\footnote{Specifically, our
notations differ from those in \cite[Corollary 6.164]{detnotes} in two ways:
Firstly, we use the shorthand $\left[  n\right]  $ for $\left\{
1,2,\ldots,n\right\}  $; secondly, what we call $\operatorname*{sub}%
\nolimits_{U}^{V}A$ is called $\operatorname*{sub}\nolimits_{w\left(
U\right)  }^{w\left(  V\right)  }A$ in the notation of \cite[Definition 6.78
and Definition 6.153]{detnotes}.}). Thus, we don't need to prove it here.

\begin{corollary}
\label{cor.finpowmat.minors-2}Let $A\in\mathbb{K}^{n\times n}$ be a matrix.
Fix $k\in\mathbb{N}$. Let $a_{k}\in\mathbb{K}$ be the coefficient of $t^{k}$
in the characteristic polynomial $\chi_{A}\in\mathbb{K}\left[  t\right]  $.
Then,
\[
a_{k}=\left(  -1\right)  ^{n-k}\sum_{\substack{P\subseteq\left[  n\right]
;\\\left\vert P\right\vert =n-k}}\det\left(  \operatorname*{sub}%
\nolimits_{P}^{P}A\right)  .
\]

\end{corollary}

\begin{proof}
[Proof of Corollary \ref{cor.finpowmat.minors-2}.]Let us consider
$A\in\mathbb{K}^{n\times n}$ as an $n\times n$-matrix over the polynomial ring
$\mathbb{K}\left[  t\right]  $. Then,
\begin{align*}
\chi_{A}  &  =\det\underbrace{\left(  tI_{n}-A\right)  }_{=-\left(  A+\left(
-t\right)  I_{n}\right)  }=\det\left(  -\left(  A+\left(  -t\right)
I_{n}\right)  \right)  =\left(  -1\right)  ^{n}\underbrace{\det\left(
A+\left(  -t\right)  I_{n}\right)  }_{\substack{=\sum_{P\subseteq\left[
n\right]  }\det\left(  \operatorname*{sub}\nolimits_{P}^{P}A\right)  \left(
-t\right)  ^{n-\left\vert P\right\vert }\\\text{(by Lemma
\ref{lem.finpowmat.minors-1},}\\\text{applied to }\mathbb{K}\left[  t\right]
\text{ and }-t\\\text{instead of }\mathbb{K}\text{ and }x\text{)}}}\\
&  =\left(  -1\right)  ^{n}\sum_{P\subseteq\left[  n\right]  }\det\left(
\operatorname*{sub}\nolimits_{P}^{P}A\right)  \underbrace{\left(  -t\right)
^{n-\left\vert P\right\vert }}_{=\left(  -1\right)  ^{n-\left\vert
P\right\vert }t^{n-\left\vert P\right\vert }}\\
&  =\left(  -1\right)  ^{n}\sum_{P\subseteq\left[  n\right]  }\det\left(
\operatorname*{sub}\nolimits_{P}^{P}A\right)  \left(  -1\right)
^{n-\left\vert P\right\vert }t^{n-\left\vert P\right\vert }\\
&  =\sum_{P\subseteq\left[  n\right]  }\det\left(  \operatorname*{sub}%
\nolimits_{P}^{P}A\right)  \underbrace{\left(  -1\right)  ^{n}\left(
-1\right)  ^{n-\left\vert P\right\vert }}_{=\left(  -1\right)  ^{\left\vert
P\right\vert }}t^{n-\left\vert P\right\vert }=\sum_{P\subseteq\left[
n\right]  }\det\left(  \operatorname*{sub}\nolimits_{P}^{P}A\right)  \left(
-1\right)  ^{\left\vert P\right\vert }t^{n-\left\vert P\right\vert }.
\end{align*}
Hence,%
\begin{align*}
&  \left(  \text{the coefficient of }t^{k}\text{ in }\chi_{A}\right) \\
&  =\left(  \text{the coefficient of }t^{k}\text{ in }\sum_{P\subseteq\left[
n\right]  }\det\left(  \operatorname*{sub}\nolimits_{P}^{P}A\right)  \left(
-1\right)  ^{\left\vert P\right\vert }t^{n-\left\vert P\right\vert }\right) \\
&  =\sum_{\substack{P\subseteq\left[  n\right]  ;\\n-\left\vert P\right\vert
=k}}\det\left(  \operatorname*{sub}\nolimits_{P}^{P}A\right)  \left(
-1\right)  ^{\left\vert P\right\vert }=\sum_{\substack{P\subseteq\left[
n\right]  ;\\\left\vert P\right\vert =n-k}}\det\left(  \operatorname*{sub}%
\nolimits_{P}^{P}A\right)  \underbrace{\left(  -1\right)  ^{\left\vert
P\right\vert }}_{\substack{=\left(  -1\right)  ^{n-k}\\\text{(since
}\left\vert P\right\vert =n-k\text{)}}}\\
&  \ \ \ \ \ \ \ \ \ \ \left(
\begin{array}
[c]{c}%
\text{since the condition \textquotedblleft}n-\left\vert P\right\vert
=k\text{\textquotedblright\ on a subset }P\text{ of }\left[  n\right] \\
\text{is equivalent to the condition \textquotedblleft}\left\vert P\right\vert
=n-k\text{\textquotedblright}%
\end{array}
\right) \\
&  =\left(  -1\right)  ^{n-k}\sum_{\substack{P\subseteq\left[  n\right]
;\\\left\vert P\right\vert =n-k}}\det\left(  \operatorname*{sub}%
\nolimits_{P}^{P}A\right)  .
\end{align*}
Now, by the definition of $a_{k}$, we have%
\[
a_{k}=\left(  \text{the coefficient of }t^{k}\text{ in }\chi_{A}\right)
=\left(  -1\right)  ^{n-k}\sum_{\substack{P\subseteq\left[  n\right]
;\\\left\vert P\right\vert =n-k}}\det\left(  \operatorname*{sub}%
\nolimits_{P}^{P}A\right)  .
\]
Corollary \ref{cor.finpowmat.minors-2} is thus proven.
\end{proof}

We introduce two more notations:

\begin{itemize}
\item If $S$ is any subset of $\left[  n\right]  $, then $\widetilde{S}$ shall
denote the complement $\left[  n\right]  \setminus S$ of $S$.

\item If $S$ is any subset of $\left[  n\right]  $, then $\sum S$ shall denote
the sum of the elements of $S$.
\end{itemize}

For example, if $n=5$, then $\widetilde{\left\{  1,3\right\}  }=\left\{
2,4,5\right\}  $ and $\sum\left\{  1,3\right\}  =1+3=4$.

\begin{lemma}
\label{lem.finpowmat.minors-3}Let $P$ and $Q$ be two subsets of $\left[
n\right]  $. Let $A=\left(  a_{i,j}\right)  _{1\leq i\leq n,\ 1\leq j\leq
n}\in\mathbb{K}^{n\times n}$ be an $n\times n$-matrix such that%
\[
\text{every }i\in P\text{ and }j\in Q\text{ satisfy }a_{i,j}=0.
\]
If $\left\vert P\right\vert +\left\vert Q\right\vert =n$, then
\[
\det A=\left(  -1\right)  ^{\sum P+\sum\widetilde{Q}}\det\left(
\operatorname*{sub}\nolimits_{P}^{\widetilde{Q}}A\right)  \det\left(
\operatorname*{sub}\nolimits_{\widetilde{P}}^{Q}A\right)  .
\]

\end{lemma}

Lemma \ref{lem.finpowmat.minors-3} is precisely \cite[Exercise 6.47
\textbf{(b)}]{detnotes} (up to notation\footnote{Specifically, our notations
differ from those in \cite[Exercise 6.47]{detnotes} in two ways: Firstly, we
use the shorthand $\left[  n\right]  $ for $\left\{  1,2,\ldots,n\right\}  $;
secondly, what we call $\operatorname*{sub}\nolimits_{U}^{V}A$ is called
$\operatorname*{sub}\nolimits_{w\left(  U\right)  }^{w\left(  V\right)  }A$ in
the notation of \cite[Definition 6.78 and Definition 6.153]{detnotes}.}).
Thus, we don't need to prove it here.

Our next proposition tells us what happens to the determinant of a matrix if
we replace some columns of the matrix by the respective columns of the
identity matrix $I_{n}$. To state this proposition, we need the following
notation: If $A$ is an $n\times m$-matrix, and if $j\in\left\{  1,2,\ldots
,m\right\}  $, then $\operatorname*{Col}\nolimits_{j}A$ shall denote the
$j$-th column of $A$. For example, $\operatorname*{Col}\nolimits_{2}\left(
\begin{array}
[c]{ccc}%
a & b & c\\
a^{\prime} & b^{\prime} & c^{\prime}%
\end{array}
\right)  =\left(
\begin{array}
[c]{c}%
b\\
b^{\prime}%
\end{array}
\right)  $.

\begin{proposition}
\label{prop.finpowmat.minors-IA}Let $A\in\mathbb{K}^{n\times n}$ be a matrix.
Let $P$ be a subset of $\left[  n\right]  $. Let $B\in\mathbb{K}^{n\times n}$
be the $n\times n$-matrix defined by setting%
\begin{equation}
\operatorname*{Col}\nolimits_{j}B=%
\begin{cases}
\operatorname*{Col}\nolimits_{j}A, & \text{if }j\in P;\\
\operatorname*{Col}\nolimits_{j}\left(  I_{n}\right)  , & \text{if }j\notin P
\end{cases}
\ \ \ \ \ \ \ \ \ \ \text{for all }j\in\left[  n\right]  .
\label{eq.prop.finpowmat.minors-IA.ColjB=}%
\end{equation}
(That is, the columns of $B$ whose indices lie in $P$ equal the corresponding
columns of $A$, while the other columns equal the corresponding columns of
$I_{n}$.)

Then,%
\[
\det B=\det\left(  \operatorname*{sub}\nolimits_{P}^{P}A\right)  .
\]

\end{proposition}

\Needspace{5cm}

\begin{example}
Let $n=4$ and $A=\left(
\begin{array}
[c]{cccc}%
a & b & c & d\\
a^{\prime} & b^{\prime} & c^{\prime} & d^{\prime}\\
a^{\prime\prime} & b^{\prime\prime} & c^{\prime\prime} & d^{\prime\prime}\\
a^{\prime\prime\prime} & b^{\prime\prime\prime} & c^{\prime\prime\prime} &
d^{\prime\prime\prime}%
\end{array}
\right)  $ and $P=\left\{  2,4\right\}  $. Then, the matrix $B$ in Proposition
\ref{prop.finpowmat.minors-IA} is given by%
\[
B=\left(
\begin{array}
[c]{cccc}%
1 & b & 0 & d\\
0 & b^{\prime} & 0 & d^{\prime}\\
0 & b^{\prime\prime} & 1 & d^{\prime\prime}\\
0 & b^{\prime\prime\prime} & 0 & d^{\prime\prime\prime}%
\end{array}
\right)  .
\]
Proposition \ref{prop.finpowmat.minors-IA} states that this matrix satisfies
\[
\det B=\det\left(  \operatorname*{sub}\nolimits_{P}^{P}A\right)  =\det\left(
\operatorname*{sub}\nolimits_{\left\{  2,4\right\}  }^{\left\{  2,4\right\}
}A\right)  =\det\left(
\begin{array}
[c]{cc}%
b^{\prime} & d^{\prime}\\
b^{\prime\prime\prime} & d^{\prime\prime\prime}%
\end{array}
\right)  .
\]

\end{example}

\begin{proof}
[Proof of Proposition \ref{prop.finpowmat.minors-IA}.]For each $j\in P$, we
have%
\begin{align}
\operatorname*{Col}\nolimits_{j}B  &  =%
\begin{cases}
\operatorname*{Col}\nolimits_{j}A, & \text{if }j\in P;\\
\operatorname*{Col}\nolimits_{j}\left(  I_{n}\right)  , & \text{if }j\notin P
\end{cases}
\ \ \ \ \ \ \ \ \ \ \left(  \text{by (\ref{eq.prop.finpowmat.minors-IA.ColjB=}%
)}\right) \nonumber\\
&  =\operatorname*{Col}\nolimits_{j}A\ \ \ \ \ \ \ \ \ \ \left(  \text{since
}j\in P\right)  . \label{pf.prop.finpowmat.minors-IA.ColjB=ColjA}%
\end{align}
In other words, the columns of $B$ with indices $j\in P$ equal the
corresponding columns of $A$. Hence, the submatrix $\operatorname*{sub}%
\nolimits_{P}^{P}B$ of $B$ equals the corresponding submatrix
$\operatorname*{sub}\nolimits_{P}^{P}A$ of $A$ (because these two submatrices
are contained entirely in the columns with indices $j\in P$). In other words,%
\[
\operatorname*{sub}\nolimits_{P}^{P}B=\operatorname*{sub}\nolimits_{P}^{P}A.
\]

Define a subset $Q$ of $\left[  n\right]  $ by $Q=\widetilde{P}$. Thus, $Q$ is
the complement of $P$ in the $n$-element set $\left[  n\right]  $; hence,
$\left\vert Q\right\vert =n-\left\vert P\right\vert $. In other words,
$\left\vert P\right\vert +\left\vert Q\right\vert =n$. Moreover, from
$Q=\widetilde{P}$, we obtain $P=\widetilde{Q}$, so that $\widetilde{Q}=P$. For
each $j\in Q$, we have%
\begin{align*}
\operatorname*{Col}\nolimits_{j}B  &  =%
\begin{cases}
\operatorname*{Col}\nolimits_{j}A, & \text{if }j\in P;\\
\operatorname*{Col}\nolimits_{j}\left(  I_{n}\right)  , & \text{if }j\notin P
\end{cases}
\ \ \ \ \ \ \ \ \ \ \left(  \text{by (\ref{eq.prop.finpowmat.minors-IA.ColjB=}%
)}\right) \\
&  =\operatorname*{Col}\nolimits_{j}\left(  I_{n}\right)
\ \ \ \ \ \ \ \ \ \ \left(  \text{since }j\notin P\text{ (because }j\in
Q=\widetilde{P}=\left[  n\right]  \setminus P\text{)}\right)  .
\end{align*}
In other words, the columns of $B$ with indices $j\in Q$ equal the
corresponding columns of $I_{n}$. Hence, the submatrix $\operatorname*{sub}%
\nolimits_{Q}^{Q}B$ of $B$ equals the corresponding submatrix
$\operatorname*{sub}\nolimits_{Q}^{Q}\left(  I_{n}\right)  $ of $I_{n}$
(because these two submatrices are contained entirely in the columns with
indices $j\in Q$). In other words,%
\[
\operatorname*{sub}\nolimits_{Q}^{Q}B=\operatorname*{sub}\nolimits_{Q}%
^{Q}\left(  I_{n}\right)  =I_{\left\vert Q\right\vert }%
\]
(since any principal submatrix of an identity matrix is itself an identity matrix).

Write the matrix $B\in\mathbb{K}^{n\times n}$ in the form $B=\left(
b_{i,j}\right)  _{1\leq i\leq n,\ 1\leq j\leq n}$. Thus,
\begin{equation}
b_{i,j}=\left(  \text{the }\left(  i,j\right)  \text{-th entry of }B\right)
\label{pf.prop.finpowmat.minors-IA.1}%
\end{equation}
for all $i,j\in\left[  n\right]  $. Furthermore,%
\[
\text{every }i\in P\text{ and }j\in Q\text{ satisfy }b_{i,j}=0
\]
\footnote{\textit{Proof.} Let $i\in P$ and $j\in Q$. Then, $j\in
Q=\widetilde{P}=\left[  n\right]  \setminus P$ (by the definition of
$\widetilde{P}$), so that $j\notin P$. Hence, $i\neq j$ (since otherwise, we
would have $i=j\notin P$, which would contradict $i\in P$). Now, the
definition of $B$ yields%
\begin{align*}
\operatorname*{Col}\nolimits_{j}B  &  =%
\begin{cases}
\operatorname*{Col}\nolimits_{j}A, & \text{if }j\in P;\\
\operatorname*{Col}\nolimits_{j}\left(  I_{n}\right)  , & \text{if }j\notin P
\end{cases}
=\operatorname*{Col}\nolimits_{j}\left(  I_{n}\right)
\ \ \ \ \ \ \ \ \ \ \left(  \text{since }j\notin P\right) \\
&  =\left(  \underbrace{0,0,\ldots,0}_{j-1\text{ zeroes}}%
,1,\underbrace{0,0,\ldots,0}_{n-j\text{ zeroes}}\right)  ^{T}%
\end{align*}
(by the definition of $I_{n}$). Hence,%
\[
\left(  \text{the }i\text{-th entry of the vector }\operatorname*{Col}%
\nolimits_{j}B\right)  =%
\begin{cases}
1, & \text{if }i=j;\\
0, & \text{if }i\neq j
\end{cases}
=0\ \ \ \ \ \ \ \ \ \ \left(  \text{since }i\neq j\right)  .
\]
But (\ref{pf.prop.finpowmat.minors-IA.1}) yields%
\begin{align*}
b_{i,j}  &  =\left(  \text{the }\left(  i,j\right)  \text{-th entry of
}B\right) \\
&  =\left(  \text{the }i\text{-th entry of the vector }\operatorname*{Col}%
\nolimits_{j}B\right) \\
&  \ \ \ \ \ \ \ \ \ \ \left(  \text{since }\operatorname*{Col}\nolimits_{j}%
B\text{ is the }j\text{-th column of }B\right) \\
&  =0.
\end{align*}
Qed.}. Hence, Lemma \ref{lem.finpowmat.minors-3} (applied to $B$ and $b_{i,j}$
instead of $A$ and $a_{i,j}$) shows that
\begin{align*}
\det B  &  =\left(  -1\right)  ^{\sum P+\sum\widetilde{Q}}\det\left(
\operatorname*{sub}\nolimits_{P}^{\widetilde{Q}}B\right)  \det\left(
\operatorname*{sub}\nolimits_{\widetilde{P}}^{Q}B\right) \\
&  =\underbrace{\left(  -1\right)  ^{\sum P+\sum P}}_{=\left(  -1\right)
^{2\sum P}=1}\det\underbrace{\left(  \operatorname*{sub}\nolimits_{P}%
^{P}B\right)  }_{=\operatorname*{sub}\nolimits_{P}^{P}A}\det
\underbrace{\left(  \operatorname*{sub}\nolimits_{Q}^{Q}B\right)
}_{=I_{\left\vert Q\right\vert }}\ \ \ \ \ \ \ \ \ \ \left(  \text{since
}\widetilde{Q}=P\text{ and }\widetilde{P}=Q\right) \\
&  =\det\left(  \operatorname*{sub}\nolimits_{P}^{P}A\right)  \underbrace{\det
\left(  I_{\left\vert Q\right\vert }\right)  }_{=1}=\det\left(
\operatorname*{sub}\nolimits_{P}^{P}A\right)  .
\end{align*}
This proves Proposition \ref{prop.finpowmat.minors-IA}.
\end{proof}

Proposition \ref{prop.finpowmat.minors-IA} has the following consequence for
exterior powers:

\begin{corollary}
\label{cor.finpowmat.ext-sub-1}Let $A\in\mathbb{K}^{n\times n}$ be a matrix.

Let $V$ be the free $\mathbb{K}$-module $\mathbb{K}^{n}$. Consider $A$ as an
endomorphism of the free $\mathbb{K}$-module $V=\mathbb{K}^{n}$. Consider the
$n$-th exterior power $\Lambda^{n}V$ of the $\mathbb{K}$-module $V$.

Let $\left(  e_{1},e_{2},\ldots,e_{n}\right)  $ be the standard basis of the
$\mathbb{K}$-module $\mathbb{K}^{n}$ (so that $e_{i}$ is the vector with a $1$
in its $i$-th entry and $0$ everywhere else).

Let $\left(  i_{1},i_{2},\ldots,i_{n}\right)  \in\left\{  0,1\right\}  ^{n}$.
Define a subset $P$ of $\left[  n\right]  $ by $P=\left\{  p\in\left[
n\right]  \ \mid\ i_{p}=1\right\}  $. Then, in $\Lambda^{n}V$, we have%
\[
\det\left(  \operatorname*{sub}\nolimits_{P}^{P}A\right)  \cdot e_{1}\wedge
e_{2}\wedge\cdots\wedge e_{n}=A^{i_{1}}e_{1}\wedge A^{i_{2}}e_{2}\wedge
\cdots\wedge A^{i_{n}}e_{n}.
\]

\end{corollary}

\begin{proof}
[Proof of Corollary \ref{cor.finpowmat.ext-sub-1}.]Define an $n\times
n$-matrix $B\in\mathbb{K}^{n\times n}$ as in Proposition
\ref{prop.finpowmat.minors-IA}. Consider $B$ as an endomorphism of the free
$\mathbb{K}$-module $V=\mathbb{K}^{n}$ as well.

We have $A^{i_{p}}e_{p}=Be_{p}$ for every $p\in\left[  n\right]
$\ \ \ \ \footnote{\textit{Proof.} Let $p\in\left[  n\right]  $. Recall that
$\left(  e_{1},e_{2},\ldots,e_{n}\right)  $ is the standard basis of the
$\mathbb{K}$-module $\mathbb{K}^{n}$. Thus, $e_{p}$ is the column vector with
a $1$ in its $p$-th entry and $0$'s everywhere else. Hence, for every $n\times
n$-matrix $C\in\mathbb{K}^{n\times n}$, we have%
\[
Ce_{p}=\left(  \text{the }p\text{-th column of }C\right)  =\operatorname*{Col}%
\nolimits_{p}C
\]
(by the definition of $\operatorname*{Col}\nolimits_{p}C$). Applying this to
$C=B$, we obtain $Be_{p}=\operatorname*{Col}\nolimits_{p}B$. The same argument
(using $A$ instead of $B$) shows that $Ae_{p}=\operatorname*{Col}%
\nolimits_{p}A$.
\par
We are in one of the following two cases:
\par
\textit{Case 1:} We have $p\in P$.
\par
\textit{Case 2:} We have $p\notin P$.
\par
Let us first consider Case 1. In this case, we have $p\in P$. In other words,
$i_{p}=1$ (by the definition of $P$). Thus, $A^{i_{p}}=A^{1}=A$. Hence,
$A^{i_{p}}e_{p}=Ae_{p}=\operatorname*{Col}\nolimits_{p}A$. Comparing this with%
\begin{align*}
Be_{p}  &  =\operatorname*{Col}\nolimits_{p}B=%
\begin{cases}
\operatorname*{Col}\nolimits_{p}A, & \text{if }p\in P;\\
\operatorname*{Col}\nolimits_{p}\left(  I_{n}\right)  , & \text{if }p\notin P
\end{cases}
\ \ \ \ \ \ \ \ \ \ \left(  \text{by (\ref{eq.prop.finpowmat.minors-IA.ColjB=}%
), applied to }j=p\right) \\
&  =\operatorname*{Col}\nolimits_{p}A\ \ \ \ \ \ \ \ \ \ \left(  \text{since
}p\in P\right)  ,
\end{align*}
we obtain $A^{i_{p}}e_{p}=Be_{p}$. Hence, $A^{i_{p}}e_{p}=Be_{p}$ is proven in
Case 1.
\par
Let us next consider Case 2. In this case, we have $p\notin P$. In other
words, $i_{p}\neq1$ (by the definition of $P$). But $i_{p}\in\left\{
0,1\right\}  $ (since $\left(  i_{1},i_{2},\ldots,i_{n}\right)  \in\left\{
0,1\right\}  ^{n}$). Hence, $i_{p}=0$ (since $i_{p}\neq1$). Thus, $A^{i_{p}%
}=A^{0}=I_{n}$. Hence, $A^{i_{p}}e_{p}=I_{n}e_{p}=e_{p}$. Comparing this with%
\begin{align*}
Be_{p}  &  =\operatorname*{Col}\nolimits_{p}B=%
\begin{cases}
\operatorname*{Col}\nolimits_{p}A, & \text{if }p\in P;\\
\operatorname*{Col}\nolimits_{p}\left(  I_{n}\right)  , & \text{if }p\notin P
\end{cases}
\ \ \ \ \ \ \ \ \ \ \left(  \text{by (\ref{eq.prop.finpowmat.minors-IA.ColjB=}%
), applied to }j=p\right) \\
&  =\operatorname*{Col}\nolimits_{p}\left(  I_{n}\right)
\ \ \ \ \ \ \ \ \ \ \left(  \text{since }p\notin P\right) \\
&  =e_{p}\ \ \ \ \ \ \ \ \ \ \left(  \text{since the columns of }I_{n}\text{
are }e_{1},e_{2},\ldots,e_{n}\right)  ,
\end{align*}
we obtain $A^{i_{p}}e_{p}=Be_{p}$. Hence, $A^{i_{p}}e_{p}=Be_{p}$ is proven in
Case 2.
\par
We have now proven $A^{i_{p}}e_{p}=Be_{p}$ in both Cases 1 and 2. Thus, the
proof of $A^{i_{p}}e_{p}=Be_{p}$ is complete.}. Combining these equalities, we
obtain%
\begin{equation}
A^{i_{1}}e_{1}\wedge A^{i_{2}}e_{2}\wedge\cdots\wedge A^{i_{n}}e_{n}%
=Be_{1}\wedge Be_{2}\wedge\cdots\wedge Be_{n}.
\label{pf.cor.finpowmat.ext-sub-1.1}%
\end{equation}

Recall that $B$ is an endomorphism of the free $\mathbb{K}$-module
$\mathbb{K}^{n}$. Hence, Lemma \ref{lem.finpowmat.ext-det} \textbf{(b)}
(applied to $\mathbb{L}=\mathbb{K}$, $u=B$ and $p=e_{1}\wedge e_{2}%
\wedge\cdots\wedge e_{n}$) yields%
\begin{align*}
\left(  \Lambda^{n}B\right)  \left(  e_{1}\wedge e_{2}\wedge\cdots\wedge
e_{n}\right)   &  =\underbrace{\det B}_{\substack{=\det\left(
\operatorname*{sub}\nolimits_{P}^{P}A\right)  \\\text{(by Proposition
\ref{prop.finpowmat.minors-IA})}}}\cdot e_{1}\wedge e_{2}\wedge\cdots\wedge
e_{n}\\
&  =\det\left(  \operatorname*{sub}\nolimits_{P}^{P}A\right)  \cdot
e_{1}\wedge e_{2}\wedge\cdots\wedge e_{n}.
\end{align*}
Hence,%
\begin{align*}
&  \det\left(  \operatorname*{sub}\nolimits_{P}^{P}A\right)  \cdot e_{1}\wedge
e_{2}\wedge\cdots\wedge e_{n}\\
&  =\left(  \Lambda^{n}B\right)  \left(  e_{1}\wedge e_{2}\wedge\cdots\wedge
e_{n}\right) \\
&  =Be_{1}\wedge Be_{2}\wedge\cdots\wedge Be_{n}\ \ \ \ \ \ \ \ \ \ \left(
\text{by the definition of the map }\Lambda^{n}B\right) \\
&  =A^{i_{1}}e_{1}\wedge A^{i_{2}}e_{2}\wedge\cdots\wedge A^{i_{n}}%
e_{n}\ \ \ \ \ \ \ \ \ \ \left(  \text{by (\ref{pf.cor.finpowmat.ext-sub-1.1}%
)}\right)  .
\end{align*}
This proves Corollary \ref{cor.finpowmat.ext-sub-1}.
\end{proof}

\begin{lemma}
\label{lem.finpowmat.ext-alter}Let $V$ be any $\mathbb{K}$-module. Let $A$ be
an endomorphism of the $\mathbb{K}$-module $V$. Let $w_{1},w_{2},\ldots
,w_{n}\in V$ be arbitrary vectors. Assume that two of these $n$ vectors
$w_{1},w_{2},\ldots,w_{n}$ are equal. Let $k\in\mathbb{Z}$. Then,
\[
\sum_{\substack{i_{1},i_{2},\ldots,i_{n}\in\left\{  0,1\right\}
;\\i_{1}+i_{2}+\cdots+i_{n}=n-k}}A^{i_{1}}w_{1}\wedge A^{i_{2}}w_{2}%
\wedge\cdots\wedge A^{i_{n}}w_{n}=0
\]
in the exterior power $\Lambda^{n}V$ of the $\mathbb{K}$-module $V$.
\end{lemma}

\begin{proof}
[Proof of Lemma \ref{lem.finpowmat.ext-alter}.]We have assumed that two of the
$n$ vectors $w_{1},w_{2},\ldots,w_{n}$ are equal. In other words, there exist
two elements $u$ and $v$ of $\left[  n\right]  $ such that $u<v$ and
$w_{u}=w_{v}$. Consider these $u$ and $v$.

Let $Z$ be the set $\left\{  \left(  i_{1},i_{2},\ldots,i_{n}\right)
\in\left\{  0,1\right\}  ^{n}\ \mid\ i_{1}+i_{2}+\cdots+i_{n}=n-k\right\}  $.
Then, the summation sign \textquotedblleft$\sum_{\substack{i_{1},i_{2}%
,\ldots,i_{n}\in\left\{  0,1\right\}  ;\\i_{1}+i_{2}+\cdots+i_{n}=n-k}%
}$\textquotedblright\ can be rewritten as \textquotedblleft$\sum_{\left(
i_{1},i_{2},\ldots,i_{n}\right)  \in Z}$\textquotedblright.

Now, let $\sigma$ be the permutation of the set $\left[  n\right]  $ that
swaps $u$ with $v$ while leaving all other elements unchanged. (This
permutation is known as a transposition.) We write $\sigma p$ for the image of
an element $p\in\left[  n\right]  $ under this permutation $\sigma$. The
definition of $\sigma$ yields $\sigma u=v$ and $\sigma v=u$. Hence, it is easy
to see that%
\begin{equation}
w_{\sigma p}=w_{p}\ \ \ \ \ \ \ \ \ \ \text{for all }p\in\left[  n\right]
\label{pf.lem.finpowmat.ext-alter.wsip}%
\end{equation}
\footnote{\textit{Proof of (\ref{pf.lem.finpowmat.ext-alter.wsip}):} Let
$p\in\left[  n\right]  $. We must prove that $w_{\sigma p}=w_{p}$. We have
either $p=u$ or $p=v$ or $p\notin\left\{  u,v\right\}  $. In each of these
three cases, let us check that $w_{\sigma p}=w_{p}$ holds:
\par
\begin{itemize}
\item If $p=u$, then
\begin{align*}
w_{\sigma p}  &  =w_{v}\ \ \ \ \ \ \ \ \ \ \left(  \text{since }%
\sigma\underbrace{p}_{=u}=\sigma u=v\right) \\
&  =w_{u}\ \ \ \ \ \ \ \ \ \ \left(  \text{since }w_{u}=w_{v}\right) \\
&  =w_{p}\ \ \ \ \ \ \ \ \ \ \left(  \text{since }u=p\right)  .
\end{align*}
Thus, we have proven $w_{\sigma p}=w_{p}$ if $p=u$.
\par
\item A similar computation proves $w_{\sigma p}=w_{p}$ if $p=v$.
\par
\item If $p\notin\left\{  u,v\right\}  $, then $\sigma p=p$ (by the definition
of $\sigma$) and thus $w_{\sigma p}=w_{p}$.
\end{itemize}
\par
Thus, $w_{\sigma p}=w_{p}$ always holds. This proves
(\ref{pf.lem.finpowmat.ext-alter.wsip}).}.

Now, each $\left(  i_{1},i_{2},\ldots,i_{n}\right)  \in Z$ satisfies either
$i_{u}<i_{v}$ or $i_{u}=i_{v}$ or $i_{u}>i_{v}$. Hence, the sum $\sum_{\left(
i_{1},i_{2},\ldots,i_{n}\right)  \in Z}A^{i_{1}}w_{1}\wedge A^{i_{2}}%
w_{2}\wedge\cdots\wedge A^{i_{n}}w_{n}$ can be split up as follows:%
\begin{equation}
\sum_{\left(  i_{1},i_{2},\ldots,i_{n}\right)  \in Z}A^{i_{1}}w_{1}\wedge
A^{i_{2}}w_{2}\wedge\cdots\wedge A^{i_{n}}w_{n}=\alpha+\beta+\gamma,
\label{pf.lem.finpowmat.ext-alter.=a+b+c}%
\end{equation}
where%
\begin{align}
\alpha &  =\sum_{\substack{\left(  i_{1},i_{2},\ldots,i_{n}\right)  \in
Z;\\i_{u}<i_{v}}}A^{i_{1}}w_{1}\wedge A^{i_{2}}w_{2}\wedge\cdots\wedge
A^{i_{n}}w_{n};\label{pf.lem.finpowmat.ext-alter.a=}\\
\beta &  =\sum_{\substack{\left(  i_{1},i_{2},\ldots,i_{n}\right)  \in
Z;\\i_{u}=i_{v}}}A^{i_{1}}w_{1}\wedge A^{i_{2}}w_{2}\wedge\cdots\wedge
A^{i_{n}}w_{n};\label{pf.lem.finpowmat.ext-alter.b=}\\
\gamma &  =\sum_{\substack{\left(  i_{1},i_{2},\ldots,i_{n}\right)  \in
Z;\\i_{u}>i_{v}}}A^{i_{1}}w_{1}\wedge A^{i_{2}}w_{2}\wedge\cdots\wedge
A^{i_{n}}w_{n}. \label{pf.lem.finpowmat.ext-alter.c=}%
\end{align}
Consider these $\alpha,\beta,\gamma$.

If $\left(  i_{1},i_{2},\ldots,i_{n}\right)  \in Z$ satisfies $i_{u}=i_{v}$,
then $\underbrace{A^{i_{u}}}_{\substack{=A^{i_{v}}\\\text{(since }i_{u}%
=i_{v}\text{)}}}\underbrace{w_{u}}_{=w_{v}}=A^{i_{v}}w_{v}$ and therefore%
\begin{equation}
A^{i_{1}}w_{1}\wedge A^{i_{2}}w_{2}\wedge\cdots\wedge A^{i_{n}}w_{n}=0
\label{pf.lem.finpowmat.ext-alter.b2}%
\end{equation}
(because the exterior product is alternating). Hence,
\[
\beta=\sum_{\substack{\left(  i_{1},i_{2},\ldots,i_{n}\right)  \in
Z;\\i_{u}=i_{v}}}\underbrace{A^{i_{1}}w_{1}\wedge A^{i_{2}}w_{2}\wedge
\cdots\wedge A^{i_{n}}w_{n}}_{\substack{=0\\\text{(by
(\ref{pf.lem.finpowmat.ext-alter.b2}))}}}=0.
\]

Furthermore, fix $\left(  i_{1},i_{2},\ldots,i_{n}\right)  \in Z$. Then, the
exterior product $A^{i_{\sigma1}}w_{\sigma1}\wedge A^{i_{\sigma2}}w_{\sigma
2}\wedge\cdots\wedge A^{i_{\sigma n}}w_{\sigma n}$ is obtained from $A^{i_{1}%
}w_{1}\wedge A^{i_{2}}w_{2}\wedge\cdots\wedge A^{i_{n}}w_{n}$ by swapping two
factors (because $\sigma$ is a transposition). Therefore,%
\[
A^{i_{\sigma1}}w_{\sigma1}\wedge A^{i_{\sigma2}}w_{\sigma2}\wedge\cdots\wedge
A^{i_{\sigma n}}w_{\sigma n}=-A^{i_{1}}w_{1}\wedge A^{i_{2}}w_{2}\wedge
\cdots\wedge A^{i_{n}}w_{n}%
\]
(since the exterior product is antisymmetric). Comparing this with
\[
A^{i_{\sigma1}}\underbrace{w_{\sigma1}}_{\substack{=w_{1}\\\text{(by
(\ref{pf.lem.finpowmat.ext-alter.wsip}))}}}\wedge A^{i_{\sigma2}%
}\underbrace{w_{\sigma2}}_{\substack{=w_{2}\\\text{(by
(\ref{pf.lem.finpowmat.ext-alter.wsip}))}}}\wedge\cdots\wedge A^{i_{\sigma n}%
}\underbrace{w_{\sigma n}}_{\substack{=w_{n}\\\text{(by
(\ref{pf.lem.finpowmat.ext-alter.wsip}))}}}=A^{i_{\sigma1}}w_{1}\wedge
A^{i_{\sigma2}}w_{2}\wedge\cdots\wedge A^{i_{\sigma n}}w_{n},
\]
we obtain%
\begin{equation}
A^{i_{\sigma1}}w_{1}\wedge A^{i_{\sigma2}}w_{2}\wedge\cdots\wedge A^{i_{\sigma
n}}w_{n}=-A^{i_{1}}w_{1}\wedge A^{i_{2}}w_{2}\wedge\cdots\wedge A^{i_{n}}%
w_{n}. \label{pf.lem.finpowmat.ext-alter.a2}%
\end{equation}

Now, forget that we fixed $\left(  i_{1},i_{2},\ldots,i_{n}\right)  $. We thus
have proven (\ref{pf.lem.finpowmat.ext-alter.a2}) for each $\left(
i_{1},i_{2},\ldots,i_{n}\right)  \in Z$.

But it is easy to see that the map%
\begin{align}
\left\{  \left(  i_{1},i_{2},\ldots,i_{n}\right)  \in Z\ \mid\ i_{u}%
<i_{v}\right\}   &  \rightarrow\left\{  \left(  i_{1},i_{2},\ldots
,i_{n}\right)  \in Z\ \mid\ i_{u}>i_{v}\right\}  ,\nonumber\\
\left(  i_{1},i_{2},\ldots,i_{n}\right)   &  \mapsto\left(  i_{\sigma
1},i_{\sigma2},\ldots,i_{\sigma n}\right)
\label{pf.lem.finpowmat.ext-alter.bij}%
\end{align}
is well-defined and is a bijection\footnote{Indeed, all this map does is
swapping the $u$-th and the $v$-th entry of the $n$-tuple it is being applied
to (because $\sigma$ swaps $u$ with $v$ while leaving all other numbers
unchanged). Hence, it preserves the sum of all entries of the $n$-tuple. Thus,
it sends an $n$-tuple in $Z$ to an $n$-tuple in $Z$. Furthermore, if we apply
this map to an $n$-tuple $\left(  i_{1},i_{2},\ldots,i_{n}\right)  $
satisfying $i_{u}<i_{v}$, then its image under this map will be an $n$-tuple
$\left(  i_{1},i_{2},\ldots,i_{n}\right)  $ satisfying $i_{u}>i_{v}$ (since it
swaps the $u$-th and the $v$-th entry of the $n$-tuple). This shows that this
map is well-defined. In order to prove that it is a bijection, we just need to
construct its inverse; this is easily done (it is given by the same recipe
(\ref{pf.lem.finpowmat.ext-alter.bij}) as our original map, but it goes in the
opposite direction).}. Hence, we can substitute $\left(  i_{\sigma1}%
,i_{\sigma2},\ldots,i_{\sigma n}\right)  $ for $\left(  i_{1},i_{2}%
,\ldots,i_{n}\right)  $ in the sum on the right hand side of
(\ref{pf.lem.finpowmat.ext-alter.c=}). We thus find%
\begin{align*}
&  \sum_{\substack{\left(  i_{1},i_{2},\ldots,i_{n}\right)  \in Z;\\i_{u}%
>i_{v}}}A^{i_{1}}w_{1}\wedge A^{i_{2}}w_{2}\wedge\cdots\wedge A^{i_{n}}w_{n}\\
&  =\sum_{\substack{\left(  i_{1},i_{2},\ldots,i_{n}\right)  \in
Z;\\i_{u}<i_{v}}}\underbrace{A^{i_{\sigma1}}w_{1}\wedge A^{i_{\sigma2}}%
w_{2}\wedge\cdots\wedge A^{i_{\sigma n}}w_{n}}_{\substack{=-A^{i_{1}}%
w_{1}\wedge A^{i_{2}}w_{2}\wedge\cdots\wedge A^{i_{n}}w_{n}\\\text{(by
(\ref{pf.lem.finpowmat.ext-alter.a2}))}}}\\
&  =-\underbrace{\sum_{\substack{\left(  i_{1},i_{2},\ldots,i_{n}\right)  \in
Z;\\i_{u}<i_{v}}}A^{i_{1}}w_{1}\wedge A^{i_{2}}w_{2}\wedge\cdots\wedge
A^{i_{n}}w_{n}}_{\substack{=\alpha\\\text{(by
(\ref{pf.lem.finpowmat.ext-alter.a=}))}}}=-\alpha.
\end{align*}
Hence, (\ref{pf.lem.finpowmat.ext-alter.c=}) becomes%
\[
\gamma=\sum_{\substack{\left(  i_{1},i_{2},\ldots,i_{n}\right)  \in
Z;\\i_{u}>i_{v}}}A^{i_{1}}w_{1}\wedge A^{i_{2}}w_{2}\wedge\cdots\wedge
A^{i_{n}}w_{n}=-\alpha.
\]
Now, (\ref{pf.lem.finpowmat.ext-alter.=a+b+c}) becomes%
\[
\sum_{\left(  i_{1},i_{2},\ldots,i_{n}\right)  \in Z}A^{i_{1}}w_{1}\wedge
A^{i_{2}}w_{2}\wedge\cdots\wedge A^{i_{n}}w_{n}=\alpha+\underbrace{\beta}%
_{=0}+\underbrace{\gamma}_{=-\alpha}=\alpha+\left(  -\alpha\right)  =0.
\]
In other words,%
\[
\sum_{\substack{i_{1},i_{2},\ldots,i_{n}\in\left\{  0,1\right\}
;\\i_{1}+i_{2}+\cdots+i_{n}=n-k}}A^{i_{1}}w_{1}\wedge A^{i_{2}}w_{2}%
\wedge\cdots\wedge A^{i_{n}}w_{n}=0
\]
(since the summation sign \textquotedblleft$\sum_{\substack{i_{1},i_{2}%
,\ldots,i_{n}\in\left\{  0,1\right\}  ;\\i_{1}+i_{2}+\cdots+i_{n}=n-k}%
}$\textquotedblright\ can be rewritten as \textquotedblleft$\sum_{\left(
i_{1},i_{2},\ldots,i_{n}\right)  \in Z}$\textquotedblright). This proves Lemma
\ref{lem.finpowmat.ext-alter}.
\end{proof}

We now finally can prove Proposition \ref{prop.finpowmat.almk-coeff} again:

\begin{proof}
[Second proof of Proposition \ref{prop.finpowmat.almk-coeff}.]The map%
\begin{align*}
V\times V\times\cdots\times V  &  \rightarrow\Lambda^{n}V,\\
\left(  w_{1},w_{2},\ldots,w_{n}\right)   &  \mapsto\sum_{\substack{i_{1}%
,i_{2},\ldots,i_{n}\in\left\{  0,1\right\}  ;\\i_{1}+i_{2}+\cdots+i_{n}%
=n-k}}A^{i_{1}}w_{1}\wedge A^{i_{2}}w_{2}\wedge\cdots\wedge A^{i_{n}}w_{n}%
\end{align*}
is $\mathbb{K}$-multilinear and alternating\footnote{Indeed, it is easy to
show that it is $\mathbb{K}$-multilinear. But then, Lemma
\ref{lem.finpowmat.ext-alter} shows that it is alternating.}. Hence, the
universal property of $\Lambda^{n}V$ (see, e.g., \cite[Theorem 3.3]%
{Conrad-extmod}) shows that there is a unique $\mathbb{K}$-linear map
$\Phi:\Lambda^{n}V\rightarrow\Lambda^{n}V$ that satisfies%
\begin{align*}
\Phi\left(  w_{1}\wedge w_{2}\wedge\cdots\wedge w_{n}\right)   &
=\sum_{\substack{i_{1},i_{2},\ldots,i_{n}\in\left\{  0,1\right\}
;\\i_{1}+i_{2}+\cdots+i_{n}=n-k}}A^{i_{1}}w_{1}\wedge A^{i_{2}}w_{2}%
\wedge\cdots\wedge A^{i_{n}}w_{n}\\
&  \ \ \ \ \ \ \ \ \ \ \text{for all }w_{1},w_{2},\ldots,w_{n}\in V.
\end{align*}
Consider this $\Phi$.

Let $\left(  e_{1},e_{2},\ldots,e_{n}\right)  $ be the standard basis of the
$\mathbb{K}$-module $\mathbb{K}^{n}$ (so that $e_{i}$ is the vector with a $1$
in its $i$-th entry and $0$ everywhere else). Thus, $\left(  e_{1}%
,e_{2},\ldots,e_{n}\right)  $ is a basis of the $\mathbb{K}$-module $V$ (since
$\mathbb{K}^{n}=V$). Hence, Lemma \ref{lem.finpowmat.ext-rankn} (applied to
$\mathbb{L}=\mathbb{K}$, $M=V$ and $b_{i}=e_{i}$) shows that the $1$-tuple
$\left(  e_{1}\wedge e_{2}\wedge\cdots\wedge e_{n}\right)  $ is a basis of the
$\mathbb{K}$-module $\Lambda^{n}V$.

Let
\[
G:\left\{  0,1\right\}  ^{n}\rightarrow\left\{  \text{subsets of }\left[
n\right]  \right\}
\]
be the map that sends each $n$-tuple $\left(  i_{1},i_{2},\ldots,i_{n}\right)
\in\left\{  0,1\right\}  ^{n}$ to the subset \newline$\left\{  p\in\left[
n\right]  \ \mid\ i_{p}=1\right\}  $ of $\left[  n\right]  $. This map $G$ is
a bijection (and is, in fact, the famous correspondence between bitstrings and
subsets of $\left[  n\right]  $). Furthermore, each $n$-tuple $\left(
i_{1},i_{2},\ldots,i_{n}\right)  \in\left\{  0,1\right\}  ^{n}$ satisfies
$G\left(  i_{1},i_{2},\ldots,i_{n}\right)  =\left\{  p\in\left[  n\right]
\ \mid\ i_{p}=1\right\}  $ (by the definition of $G$) and therefore%
\begin{equation}
\det\left(  \operatorname*{sub}\nolimits_{G\left(  i_{1},i_{2},\ldots
,i_{n}\right)  }^{G\left(  i_{1},i_{2},\ldots,i_{n}\right)  }A\right)  \cdot
e_{1}\wedge e_{2}\wedge\cdots\wedge e_{n}=A^{i_{1}}e_{1}\wedge A^{i_{2}}%
e_{2}\wedge\cdots\wedge A^{i_{n}}e_{n}
\label{pf.prop.finpowmat.almk-coeff.2nd.3}%
\end{equation}
(by Corollary \ref{cor.finpowmat.ext-sub-1}, applied to $P=G\left(
i_{1},i_{2},\ldots,i_{n}\right)  $). Moreover, each $n$-tuple \newline$\left(
i_{1},i_{2},\ldots,i_{n}\right)  \in\left\{  0,1\right\}  ^{n}$ satisfies%
\begin{equation}
\left\vert G\left(  i_{1},i_{2},\ldots,i_{n}\right)  \right\vert =i_{1}%
+i_{2}+\cdots+i_{n}, \label{pf.prop.finpowmat.almk-coeff.2nd.sizeG}%
\end{equation}
because
\begin{align*}
i_{1}+i_{2}+\cdots+i_{n}  &  =\sum_{p\in\left[  n\right]  }i_{p}%
=\sum_{\substack{p\in\left[  n\right]  ;\\i_{p}\neq1}}\underbrace{i_{p}%
}_{\substack{=0\\\text{(since }i_{p}\in\left\{  0,1\right\}  \\\text{(because
}\left(  i_{1},i_{2},\ldots,i_{n}\right)  \in\left\{  0,1\right\}
^{n}\text{)}\\\text{but }i_{p}\neq1\text{)}}}+\sum_{\substack{p\in\left[
n\right]  ;\\i_{p}=1}}\underbrace{i_{p}}_{=1}\\
&  =\underbrace{\sum_{\substack{p\in\left[  n\right]  ;\\i_{p}\neq1}}0}%
_{=0}+\sum_{\substack{p\in\left[  n\right]  ;\\i_{p}=1}}1=\sum_{\substack{p\in
\left[  n\right]  ;\\i_{p}=1}}1=\left\vert \left\{  p\in\left[  n\right]
\ \mid\ i_{p}=1\right\}  \right\vert \cdot1\\
&  =\left\vert \underbrace{\left\{  p\in\left[  n\right]  \ \mid
\ i_{p}=1\right\}  }_{\substack{=G\left(  i_{1},i_{2},\ldots,i_{n}\right)
\\\text{(by the definition of }G\text{)}}}\right\vert =\left\vert G\left(
i_{1},i_{2},\ldots,i_{n}\right)  \right\vert .
\end{align*}

Now,%
\begin{align}
&  \underbrace{a_{k}}_{\substack{=\left(  -1\right)  ^{n-k}\sum
_{\substack{P\subseteq\left[  n\right]  ;\\\left\vert P\right\vert =n-k}%
}\det\left(  \operatorname*{sub}\nolimits_{P}^{P}A\right)  \\\text{(by
Corollary \ref{cor.finpowmat.minors-2})}}}\cdot e_{1}\wedge e_{2}\wedge
\cdots\wedge e_{n}\nonumber\\
&  =\left(  -1\right)  ^{n-k}\sum_{\substack{P\subseteq\left[  n\right]
;\\\left\vert P\right\vert =n-k}}\det\left(  \operatorname*{sub}%
\nolimits_{P}^{P}A\right)  \cdot e_{1}\wedge e_{2}\wedge\cdots\wedge
e_{n}\nonumber\\
&  =\left(  -1\right)  ^{n-k}\underbrace{\sum_{\substack{\left(  i_{1}%
,i_{2},\ldots,i_{n}\right)  \in\left\{  0,1\right\}  ^{n};\\\left\vert
G\left(  i_{1},i_{2},\ldots,i_{n}\right)  \right\vert =n-k}}}_{\substack{=\sum
_{\substack{\left(  i_{1},i_{2},\ldots,i_{n}\right)  \in\left\{  0,1\right\}
^{n};\\i_{1}+i_{2}+\cdots+i_{n}=n-k}}\\\text{(by
(\ref{pf.prop.finpowmat.almk-coeff.2nd.sizeG}))}}}\underbrace{\det\left(
\operatorname*{sub}\nolimits_{G\left(  i_{1},i_{2},\ldots,i_{n}\right)
}^{G\left(  i_{1},i_{2},\ldots,i_{n}\right)  }A\right)  \cdot e_{1}\wedge
e_{2}\wedge\cdots\wedge e_{n}}_{\substack{=A^{i_{1}}e_{1}\wedge A^{i_{2}}%
e_{2}\wedge\cdots\wedge A^{i_{n}}e_{n}\\\text{(by
(\ref{pf.prop.finpowmat.almk-coeff.2nd.3}))}}}\nonumber\\
&  \ \ \ \ \ \ \ \ \ \ \left(
\begin{array}
[c]{c}%
\text{here, we have substituted }G\left(  i_{1},i_{2},\ldots,i_{n}\right)
\text{ for }P\text{ in the sum,}\\
\text{since the map }G:\left\{  0,1\right\}  ^{n}\rightarrow\left\{
\text{subsets of }\left[  n\right]  \right\}  \text{ is a bijection}%
\end{array}
\right) \nonumber\\
&  =\left(  -1\right)  ^{n-k}\underbrace{\sum_{\substack{\left(  i_{1}%
,i_{2},\ldots,i_{n}\right)  \in\left\{  0,1\right\}  ^{n};\\i_{1}+i_{2}%
+\cdots+i_{n}=n-k}}}_{=\sum_{\substack{i_{1},i_{2},\ldots,i_{n}\in\left\{
0,1\right\}  ;\\i_{1}+i_{2}+\cdots+i_{n}=n-k}}}A^{i_{1}}e_{1}\wedge A^{i_{2}%
}e_{2}\wedge\cdots\wedge A^{i_{n}}e_{n}\nonumber\\
&  =\left(  -1\right)  ^{n-k}\underbrace{\sum_{\substack{i_{1},i_{2}%
,\ldots,i_{n}\in\left\{  0,1\right\}  ;\\i_{1}+i_{2}+\cdots+i_{n}%
=n-k}}A^{i_{1}}e_{1}\wedge A^{i_{2}}e_{2}\wedge\cdots\wedge A^{i_{n}}e_{n}%
}_{\substack{=\Phi\left(  e_{1}\wedge e_{2}\wedge\cdots\wedge e_{n}\right)
\\\text{(by the definition of }\Phi\text{)}}}\nonumber\\
&  =\left(  -1\right)  ^{n-k}\Phi\left(  e_{1}\wedge e_{2}\wedge\cdots\wedge
e_{n}\right)  . \label{pf.prop.finpowmat.almk-coeff.2nd.9}%
\end{align}

Now, let $w_{1},w_{2},\ldots,w_{n}\in V$ be arbitrary. Then, there exists some
$\lambda\in\mathbb{K}$ such that $w_{1}\wedge w_{2}\wedge\cdots\wedge
w_{n}=\lambda\cdot e_{1}\wedge e_{2}\wedge\cdots\wedge e_{n}$ (since the
$1$-tuple $\left(  e_{1}\wedge e_{2}\wedge\cdots\wedge e_{n}\right)  $ is a
basis of the $\mathbb{K}$-module $\Lambda^{n}V$). Consider this $\lambda$.
Now,%
\begin{align*}
a_{k}\cdot\underbrace{w_{1}\wedge w_{2}\wedge\cdots\wedge w_{n}}%
_{=\lambda\cdot e_{1}\wedge e_{2}\wedge\cdots\wedge e_{n}}  &  =\lambda
\cdot\underbrace{a_{k}\cdot e_{1}\wedge e_{2}\wedge\cdots\wedge e_{n}%
}_{\substack{=\left(  -1\right)  ^{n-k}\Phi\left(  e_{1}\wedge e_{2}%
\wedge\cdots\wedge e_{n}\right)  \\\text{(by
(\ref{pf.prop.finpowmat.almk-coeff.2nd.9}))}}}\\
&  =\lambda\cdot\left(  -1\right)  ^{n-k}\Phi\left(  e_{1}\wedge e_{2}%
\wedge\cdots\wedge e_{n}\right) \\
&  =\left(  -1\right)  ^{n-k}\underbrace{\lambda\cdot\Phi\left(  e_{1}\wedge
e_{2}\wedge\cdots\wedge e_{n}\right)  }_{\substack{=\Phi\left(  \lambda\cdot
e_{1}\wedge e_{2}\wedge\cdots\wedge e_{n}\right)  \\\text{(since the map }%
\Phi\text{ is }\mathbb{K}\text{-linear)}}}\\
&  =\left(  -1\right)  ^{n-k}\Phi\left(  \underbrace{\lambda\cdot e_{1}\wedge
e_{2}\wedge\cdots\wedge e_{n}}_{=w_{1}\wedge w_{2}\wedge\cdots\wedge w_{n}%
}\right) \\
&  =\left(  -1\right)  ^{n-k}\underbrace{\Phi\left(  w_{1}\wedge w_{2}%
\wedge\cdots\wedge w_{n}\right)  }_{\substack{=\sum_{\substack{i_{1}%
,i_{2},\ldots,i_{n}\in\left\{  0,1\right\}  ;\\i_{1}+i_{2}+\cdots+i_{n}%
=n-k}}A^{i_{1}}w_{1}\wedge A^{i_{2}}w_{2}\wedge\cdots\wedge A^{i_{n}}%
w_{n}\\\text{(by the definition of }\Phi\text{)}}}\\
&  =\left(  -1\right)  ^{n-k}\sum_{\substack{i_{1},i_{2},\ldots,i_{n}%
\in\left\{  0,1\right\}  ;\\i_{1}+i_{2}+\cdots+i_{n}=n-k}}A^{i_{1}}w_{1}\wedge
A^{i_{2}}w_{2}\wedge\cdots\wedge A^{i_{n}}w_{n}.
\end{align*}
Thus, Proposition \ref{prop.finpowmat.almk-coeff} is proven again.
\end{proof}

\subsection{\label{sec.finpowmat.converse}Appendix:
Proof of Proposition \ref{prop.finpowmat.char-int-conv}}

We have yet to prove Proposition \ref{prop.finpowmat.char-int-conv}. This is
much easier than proving its converse, which we have already done. We begin by
stating a trivial consequence of Theorem \ref{thm.finpowmat.int-G10}:

\begin{corollary}
\label{cor.finpowmat.int-then-finger}Let $\mathbb{K}$ be a commutative ring.
Let $\mathbb{L}$ be a commutative $\mathbb{K}$-algebra. Let $n\in\mathbb{N}$.
Let $u\in\mathbb{L}$. Assume that there exists a monic polynomial
$f\in\mathbb{K}\left[  t\right]  $ of degree $n$ such that $f\left(  u\right)
=0$. Then, $\mathbb{K}\left[  u\right]  =\left\langle u^{0},u^{1}%
,\ldots,u^{n-1}\right\rangle _{\mathbb{K}}$.
\end{corollary}

\begin{proof}
[Proof of Corollary \ref{cor.finpowmat.int-then-finger}.]Assertion
$\mathcal{A}$ of Theorem \ref{thm.finpowmat.int-G10} holds (since there exists
a monic polynomial $f\in\mathbb{K}\left[  t\right]  $ of degree $n$ such that
$f\left(  u\right)  =0$). Hence, Assertion $\mathcal{D}$ of of Theorem
\ref{thm.finpowmat.int-G10} holds as well (since Theorem
\ref{thm.finpowmat.int-G10} yields that Assertions $\mathcal{A}$ and
$\mathcal{D}$ are equivalent). In other words, $\mathbb{K}\left[  u\right]
=\left\langle u^{0},u^{1},\ldots,u^{n-1}\right\rangle _{\mathbb{K}}$. This
proves Corollary \ref{cor.finpowmat.int-then-finger}.
\end{proof}

\begin{proof}
[Proof of Proposition \ref{prop.finpowmat.char-int-conv}.]Let $\mathbb{M}$ be
the $\mathbb{K}$-subalgebra of $\mathbb{L}$ generated by the coefficients of
the characteristic polynomial $\chi_{A}\in\mathbb{L}\left[  t\right]  $. Let
$u_{1},u_{2},\ldots,u_{m}$ be these coefficients. These coefficients
$u_{1},u_{2},\ldots,u_{m}$ are integral over $\mathbb{K}$ (since we assumed
that each coefficient of the characteristic polynomial $\chi_{A}\in
\mathbb{L}\left[  t\right]  $ is integral over $\mathbb{K}$), and generate
$\mathbb{M}$ as a $\mathbb{K}$-algebra (by the definition of $\mathbb{M}$);
thus, in particular, they are elements of $\mathbb{M}$. Hence, Theorem
\ref{thm.finpowmat.fin-alg} (applied to $\mathbb{M}$ instead of $\mathbb{L}$)
yields that the $\mathbb{K}$-module $\mathbb{M}$ is finitely generated. Thus,
the $\mathbb{K}$-module $\mathbb{M}^{n}$ is finitely generated as well.

Clearly, $\mathbb{L}$ is a commutative $\mathbb{M}$-algebra (since
$\mathbb{M}$ is a subring of $\mathbb{L}$). Also, the $\mathbb{M}$-algebra
$\mathbb{M}\left[  A\right]  $ is commutative (since it is generated by a
single element $A$ over the commutative ring $\mathbb{M}$).

On the other hand, the Cayley-Hamilton theorem yields $\chi_{A}\left(
A\right)  =0$. But the characteristic polynomial $\chi_{A}$ is a monic
polynomial of degree $n$; furthermore, all its coefficients belong to
$\mathbb{M}$ (since $\mathbb{M}$ was defined to be the $\mathbb{K}$-subalgebra
of $\mathbb{L}$ generated by these coefficients). Thus, the polynomial
$\chi_{A}$ belongs to $\mathbb{M}\left[  t\right]  $. Hence, there exists a
monic polynomial $f\in\mathbb{M}\left[  t\right]  $ of degree $n$ such that
$f\left(  A\right)  =0$ (namely, $f=\chi_{A}$). Hence, Corollary
\ref{cor.finpowmat.int-then-finger} (applied to $\mathbb{M}$, $\mathbb{M}%
\left[  A\right]  $ and $A$ instead of $\mathbb{K}$, $\mathbb{L}$ and $u$)
yields $\mathbb{M}\left[  A\right]  =\left\langle A^{0},A^{1},\ldots
,A^{n-1}\right\rangle _{\mathbb{M}}$. Thus, each element of $\mathbb{M}\left[
A\right]  $ is an $\mathbb{M}$-linear combination of the powers $A^{0}%
,A^{1},\ldots,A^{n-1}$. Therefore, the $\mathbb{K}$-module homomorphism%
\begin{align*}
\pi:\mathbb{M}^{n} &  \rightarrow\mathbb{M}\left[  A\right]  ,\\
\left(  m_{0},m_{1},\ldots,m_{n-1}\right)   &  \mapsto m_{0}A^{0}+m_{1}%
A^{1}+\cdots+m_{n-1}A^{n-1}%
\end{align*}
is surjective. Hence, Lemma \ref{lem.finpowmat.finger-sur} (applied to
$M=\mathbb{M}^{n}$, $N=\mathbb{M}\left[  A\right]  $ and $f=\pi$) shows that
the $\mathbb{K}$-module $\mathbb{M}\left[  A\right]  $ is finitely generated.
This $\mathbb{K}$-module $\mathbb{M}\left[  A\right]  $ clearly satisfies
$A\cdot\mathbb{M}\left[  A\right]  \subseteq\mathbb{M}\left[  A\right]  $
(since it is a $\mathbb{K}$-algebra and contains $A$). Moreover, every
$v\in\mathbb{M}\left[  A\right]  $ satisfying $v\cdot\mathbb{M}\left[
A\right]  =0$ satisfies $v=0$ (since it satisfies $v=v\cdot\underbrace{1}%
_{\in\mathbb{M}\left[  A\right]  }\in v\cdot\mathbb{M}\left[  A\right]  =0$).
Hence, Corollary \ref{cor.finpowmat.det-crit} (applied to $\mathbb{M}\left[
A\right]  $, $A$, $\mathbb{M}\left[  A\right]  $ and $\mathbb{M}\left[
A\right]  $ instead of $\mathbb{L}$, $u$, $C$ and $U$) yields that
$A\in\mathbb{M}\left[  A\right]  $ is integral over $\mathbb{K}$. In other
words, $A\in\mathbb{L}^{n\times n}$ is integral over $\mathbb{K}$. This proves
Proposition \ref{prop.finpowmat.char-int-conv}.
\end{proof}

\section{Dynamical Behaviour of Linear and Additive Cellular Automata}
\label{sec:LCA}

A \emph{discrete time dynamical system} (DTDS) is a pair $(\mathcal{X}, \mathcal{F})$, where~$\mathcal{X}$ is any set  equipped with a distance function~$d$ and~$\mathcal{F} : \mathcal{X} \to \mathcal{X}$ is a map that is
continuous on~$\mathcal{X}$ according to the topology induced by~$d$ (see~\cite{KatHas95, LinMar95} for a background on discrete time dynamical systems). The main goal of this section is to prove that several important properties of discrete time dynamical systems are decidable for additive cellular automata over a finite abelian group. First of all, we will prove that the following properties are decidable for a nontrivial subclass of additive cellular automata, namely, the linear cellular automata over $\K^n$ (see~\cite{LebMar95} for an introduction to them), where $\K$ is the finite ring  $\Z/m\Z$:
\begin{itemize}
\item sensitivity to the initial conditions and equicontinuity;
\item equicontinuity;
\item topological transitivity and other mixing properties;
\item ergodicity and other ergodic properties.
\end{itemize}
Then, we will extend the decidability results about the above mentioned properties and other ones (as injectivity and surjectivity) from linear cellular automata to additive cellular automata over a finite abelian group.
 
\subsection{Background on DTDS and Cellular Automata}
We begin by reviewing some general notions about discrete time dynamical systems and cellular automata. 

\medskip
Let $(\mathcal{X}, \mathcal{F})$ be a DTDS. 

We say that $(\mathcal{X}, \mathcal{F})$ is \emph{surjective}, resp., \emph{injective}, if  $\mathcal{F}$ is \emph{surjective}, resp., \emph{injective}.  
The DTDS $(\mathcal{X}, \mathcal{F})$ is \emph{sensitive to the initial conditions} (or simply \emph{sensitive}) if there exists $\varepsilon>0$ such that for any $x\in \mathcal{X}$ and any
$\delta>0$ there is an element $y\in \mathcal{X}$ such that
$0<d(y,x)<\delta$ and $d(\mathcal{F}^k(y),\mathcal{F}^k(x))>\varepsilon$
for some $k\in\n$. The system $(\mathcal{X}, \mathcal{F})$ is said to be \emph{equicontinuous} if $\forall\varepsilon>0$ there exists $\delta>0$ such that for all $x,y\in \mathcal{X}$,
$d(x,y)<\delta$ implies that $\forall k\in\n,\;d(\mathcal{F}^k(x),\mathcal{F}^k(y))<\varepsilon$.  
 As dynamical properties, sensitivity and equicontinuity represent the main features of unstable and stable dynamical systems, respectively. The former is the well-known basic component and essence of the chaotic behavior of discrete time dynamical systems, while the latter is a strong form of stability. 
 
The DTDS $(\mathcal{X}, \mathcal{F})$ is \emph{topologically transitive} (or, simply, \emph{transitive}) if  for all nonempty open subsets $U$ and $V$ of
$\mathcal{X}$ there exists a natural number $h$ such that $\mathcal{F}^h(U)\cap V\neq \emptyset$, while it is said to be \emph{topologically mixing} if for all nonempty open subsets $U$ and $V$ of
$\mathcal{X}$ there exists a natural number $h_0$ such that the previous intersection condition holds for every $h\geq h_0$. Clearly, topological mixing is a stronger condition than transitivity. Moreover, $(\mathcal{X}, \mathcal{F})$ is \emph{topologically weakly mixing} if the DTDS $(\mathcal{X}\times\mathcal{X}, \mathcal{F} \times \mathcal{F})$ is topologically transitive, while it is \emph{totally transitive} if  $(\mathcal{X}, \mathcal{F}^h)$ is topologically transitive for all $h\in\N$.

Let $(\mathcal{X},{\cal{M}},\mu)$ be a probability space and let $(\mathcal{X}, \mathcal{F})$ be a DTDS where $\mathcal{F}$ is a measurable map which preserves  $\mu$, \ie,
$\mu(E)=\mu(\mathcal{F}^{-1}(E))$ for every $E\in {\cal{M}}$. The DTDS $(\mathcal{X}, \mathcal{F})$, or,  the map $\mathcal{F}$, is \emph{ergodic} with respect to $\mu$ if for every $E\in {\cal{M}}$
\[
\left(E=\mathcal{F}^{-1}(E)\right) \Rightarrow \mu(E)(1-\mu(E))=0
\]
It is well known that $\mathcal{F}$ is ergodic iff for any pair of sets $A,B\in\mathcal{M}$ it holds that 
\[ 
\lim_{h \to\infty} \frac{1}{h}\sum_{i=0}^{h-1}\mu (\mathcal{F}^{-i}(A) \cap B)=\mu(A) \mu(B)
\]
The DTDS $(\mathcal{X}, \mathcal{F})$ is \emph{(ergodic) mixing}, if for any pair of sets $A,B\in\mathcal{M}$ it holds that 
%
\[ 
\lim_{h \to\infty} \mu (\mathcal{F}^{-h}(A) \cap B)=\mu(A) \mu(B)\enspace,
\]
while it is \emph{(ergodic) weak mixing}, 
 if for any pair of sets $A,B\in\mathcal{M}$ it holds that
\[ 
\lim_{h \to\infty} \frac{1}{h}\sum_{i=0}^{h-1}| \mu (\mathcal{F}^{-i}(A) \cap B)-\mu(A) \mu(B)|=0
\]

We now recall some general notions about cellular automata. 

\smallskip
Let $S$ be a finite set. A configuration over $S$ is a map from $\Z$ to $S$. 
We consider the following \textit{space of configurations}
$$
S^{\Z}=\left\{\ccc|\ \ccc\colon{\Z}\to S\right\}.
$$
Each element  $\ccc\in S^{\Z}$ can be visualized as an infinite
one-dimensional cell lattice in which each cell $i\in\Z$ contains the element $\ccc_i\in S$.

\smallskip

Let $r\in\N$ and $\delta\colon S^{2r+1}\to S^n$ be any map.  We say that $r$
is the radius of~$\delta$. 
\begin{definition}[Cellular Automaton]
A \emph{one-dimensional CA based on a radius $r$ local rule $\delta$} 
is a pair $(S^{\Z},F)$, where
$$
F\colon S^{\Z}\to S^{\Z},
$$
is the {\em global transition map} defined as follows:  
\begin{equation} \label{FDdef}
\forall \ccc\in S^{\Z},\,\forall i\in\Z, \quad F(\ccc)_i=\delta\left(\ccc_{i-r}, \ldots, \ccc_{i+r}\right).
\end{equation}
In other words, the content of cell $i$ in the configuration $F(\ccc)$ is a
function of the content of cells $i-r,\ldots,i+r$ in the
configuration~$\ccc$.  
\end{definition}
We stress that the local rule~$\delta$ 
completely determines the global rule~$F$ of a CA. 

\medskip
In order to study the dynamical properties of one-dimensional CA, we
introduce a distance over the space of the configurations. Namely, $S^{\Z}$ is 
equipped with the Tychonoff distance $d$ defined as follows
\[
\forall \ccc,\ccc'\in\az, \quad d(\ccc,\ccc')=\frac{1}{2^\ell}\;\quad \text{where}\;\ell=\min\{i\in\N\,:\,\ccc_i\ne
\ccc'_i \;\text{or}\;\ccc_{-i}\ne \ccc'_{-i}\}\enspace.
\]

It is easy to verify that metric topology induced by $d$ coincides with the product
topology induced by the discrete topology on $\az$.  With this topology,
$\az$ is a compact and totally disconnected space and the global transition map $F$ of any CA $(\az, F)$ turns out to be (uniformly) continuous. Therefore, any CA itself is also a discrete time dynamical system. Moreover, any map $F:\az\to\az$ is the global transition rule of a CA if and only if $F$ is (uniformly) continuous and $F\circ\sigma=\sigma\circ F$, where $\sigma:\az\to\az$ is the \emph{shift map} defined as $\forall \ccc\in\az$, $\forall i\in\Z$, $\sigma(\ccc)_i=\ccc_{i+1}$. From now, when no misunderstanding is possible, we identify a CA with its global rule. Moreover, whenever an ergodic property is considered for CA, $\mu$ is the well-known Haar measure over the collection $\mathcal{M}$ of measurable subsets of $S^{\Z}$, \ie, the one defined as the product measure induced by the uniform probability distribution over $S$.
\subsection{Additive and Linear Cellular Automata}
Let us introduce the background of additive CA. The alphabet $S$ will be a finite abelian group $G$, with group operation $+$, neutral element $0$, and inverse operation $-$. In this way, the configuration space $G^{\Z}$ turns out to be a finite abelian group, too, where the group operation of $G^{\Z}$ is the component-wise extension of $+$ to $G^{\Z}$. With an abuse of notation, we denote by the same symbols $+$, $0$, and $-$ the group operation, the neutral element, and the inverse operation, respectively, both of $G$ and $G^{\Z}$. Observe that $+$ and $-$ are continuous functions in the topology induced by the metric $d$. A configuration $\ccc\in G^{\Z}$ is said to be \emph{finite} if the number of positions $i\in\Z$ with $\ccc_i\neq 0$ is finite.

\begin{definition}[Additive Cellular Automata]
An \emph{additive CA} over a abelian finite group $G$ is a  CA 
$(G^{\Z},F)$  
where the global transition map $F:G^{\Z}\to G^{\Z}$  is an 
endomorphism of $G^{\Z}$.
\end{definition}
The \emph{sum of two additive CA}  $F_1$ and $F_2$ over $G$ is naturally defined as the map on $G^{\Z}$ denoted by $F_1+F_2$ and such that 
\[
\forall \ccc\in G^{\Z},\quad  (F_1+F_2)(\ccc)=F_1(\ccc)+F_2(\ccc)
\]
Clearly, $F_1+F_2$ is an additive CA over $G$.

\medskip
We now recall the notion of linear CA, an important subclass of additive CA. We stress that, whenever the term \emph{linear} is involved, the alphabet $S$ is $\mathbb{K}^n$, where $\mathbb{K} = \Z/m\Z$ for some positive integer $m$. Both $\K^n$ and $\Sp$ become $\K$-modules in the obvious (\ie, entrywise) way.


\smallskip
A local rule $\delta\colon (\mathbb{K}^n)^{2r+1}\to \mathbb{K}^n$ of radius $r$ is said to be linear
if it is defined by  $2r+1$ matrices  
$A_{-r},\ldots, A_0,\ldots, A_r\in \mathbb{K}^{n\times n}$
as follows: 
$$
\forall (x_{-r}, \ldots,x_0,\ldots,x_r)\in(\mathbb{K}^n)^{2r+1}, \quad \delta(x_{-r}, \ldots,x_0,\ldots,x_r) = 
\sum_{i=-r}^r A_i\cdot x_i
\enspace.
$$ 

\begin{definition}[Linear Cellular Automata (LCA)]
A \emph{linear CA (LCA)} over $\mathbb{K}^n$ is a CA based on a linear local rule.  
\end{definition}
We remark that for a
linear one-dimensional CA, equation~\eqref{FDdef} becomes
$$F(\ccc)_i= 
A_{i-r} \cdot \ccc_{i-r} + \ldots + A_{i+r} \cdot \ccc_{i+r}
$$

Let $\LP$ denote the set of Laurent polynomials with coefficients in $\K$. Before proceeding, let us recall the  \emph{formal power series} (fps)
\noindent
which have been successfully used to study the dynamical 
behaviour of LCA in the case $n=1$~\cite{ItoOsa83,ManMan99}. The idea of this formalism is that configurations and global rules are represented by
suitable polynomials and the application of the global rule turns into multiplications of
polynomials. 
In the more general case of LCA over $\K^n$,  a configuration $\ccc\in(\K^n)^\Z$ can be associated with the fps
\[
\P_{\ccc}(X)=\sum_{i\in\z}\ccc_i X^i = 
\begin{bmatrix}
c^1(X)\\
\vdots\\
c^n(X)
\end{bmatrix}
=
\begin{bmatrix}
\sum_{i\in\Z} c_i^1 X^i\\
\vdots\\
\sum_{i\in\Z} c_i^n X^i
\end{bmatrix}
\]
\noindent
Then, if 
$F$ is the global rule of a LCA defined by $A_{-r},\ldots, A_0,\ldots, A_r$, one finds 
\[
\P_{F(\ccc)}(X)=
A \cdot \P_{\ccc}(X)
\]
where 
$$
A=
\sum\limits_{i=-r}^{r} A_i X^{-i}\in\LP^{n\times n}
$$
is the \emph{finite fps}, or, \emph{the matrix}, \emph{associated with the LCA $F$}. In this way, for any integer $k>0$ the fps associated with $F^k$ is $A^t$, and then 
$
\P_{F^k(\ccc)}(X)=
A^k \cdot \P_{\ccc}(X)
\enspace.
$  


A matrix 
$A\in\LP^{n\times n}$ is in \emph{Frobenius normal form} if  
\begin{eqnarray}
\label{FNF}
A=
\begin{bmatrix}
0&1&0& \dots&0&0\\
0&0& 1&\ddots&0&0\\
0&0&0 &\ddots&0&0\\
\vdots&\vdots& \vdots&\ddots&\ddots&\vdots\\
\\
0        &   0 &0   & \dots &0&1\\
\\
\mm_0&\mm_1& \mm_2 &\dots &\mm_{n-2}&\mm_{n-1}
\end{bmatrix}
\end{eqnarray}
where each $\mm_i\in\LP$. Recall that the coefficients of $\chi_A$ are just the elements $\mm_i$ of the $n$-th row of $A$ (up to sign).


\begin{definition}[Frobenius LCA]
  \label{frobeniusLCA}
A LCA $(\Sp,F)$  is said to be a \emph{Frobenius LCA} if the fps $A\in\LP^{n\times n}$ associated with $F$ is in Frobenius normal form. 
\end{definition}

\subsection{Decidability Results about Linear CA}
We now deal with sensitivity and equicontinuity for LCA over $\K^n$. First of all, 
we remind that a dichotomy between sensitivity and equicontinuity holds for LCA. Moreover, these properties are characterized by the behavior of the powers of the matrix associated  with a LCA.
\begin{proposition}[\cite{Dennun19}]\label{prop:dich}
Let $\left(\left( \K^n \right)^{\Z}, F\right )$ be a LCA over $\K^n$ and let $A$ be the matrix associated with $F$. The following statements are equivalent:
\begin{enumerate}
\item $F$ is sensitive to the initial conditions;
\item $F$ is not equicontinuous;
\item $\left | \{A^1, A^2, A^3, \ldots \} \right | = \infty$.
\end{enumerate}
\end{proposition}
Note that the statement ``$\left | \{A^1, A^2, A^3, \ldots \} \right | = \infty$'' here is clearly equivalent to ``$\left | \{A^0, A^1, A^2, \ldots \} \right | = \infty$'', which is the kind of statement discussed in Theorem~\ref{thm.finpowmat.main}. 
An immediate  consequence of Proposition~\ref{prop:dich} is that any decidable characterization of sensitivity to the initial conditions in terms of the matrices defining LCA over $\K^n$ would also provide a characterization of equicontinuity. In the sequel, we are going to show that such a characterization actually exists. First of all, we remind that the a decidable characterization of sensitivity and equicontinuity was provided for the class of Frobenius LCA in~\cite{Dennun19}. In particular, the following result holds.
\begin{theorem}[Theorem 31 in~\cite{Dennun19}]\label{froblca}
Sensitivity and equicontinuity are decidable for Frobenius LCA over $\K^n$.
\end{theorem}
By means of the main result from Section~\ref{chp.finpowmat} we are now able to prove the following
\begin{theorem}
\label{declca}
Sensitivity and equicontinuity are decidable for LCA over $\K^n$.
\end{theorem}
\begin{proof}
Let $\left(\left( \K^n \right)^{\Z}, G\right )$ 
be any LCA over $\K^n$ and let $A$ be the matrix associated with $G$. Consider the  Frobenius LCA $\left(\left( \K^n \right)^{\Z}, F\right )$ such that $\chi_{A}=\chi_{B}$, where $B$ is the matrix (in Frobenius normal form) associated with $F$. By Theorem~\ref{thm.finpowmat.main} and Proposition~\ref{prop:dich} the former LCA is equicontinuous if and only if the latter is.  Theorem~\ref{froblca} concludes the proof. 
\end{proof}
For a sake of completeness, we recall that injectivity and surjectivity are decidable for LCA over $\K^n$. This result follows from a characterization of these properties in terms of the determinant of the matrix associated with a LCA and from the fact that injectivity and surjectivity are decidable for LCA over $\K$ (for the latter, see~\cite{ItoOsa83}).
\begin{theorem}[\cite{LebMar95,Kari00}]
\label{decinjsurjlca}
Injectivity and equicontinuity are decidable for LCA over $\K^n$. In particular, a LCA over $\K^n$ is injective (resp., surjective) if and only if the determinant of the matrix associated with it is the fps associated with an injective (resp., surjective) LCA over $\K$.
\end{theorem}
The decidability of topologically transitivity, ergodicity, and other mixing and ergodic properties for LCA over $\K^n$ has been recently proved in~\cite{Dennun19b}. In particular, authors showed the equivalence of all the mixing and ergodic properties for additive CA over a finite abelian group and the decidability  for LCA over $\K^n$ (see also~\cite{Dennun19c}). 
\begin{theorem}[\cite{Dennun19b}]\label{teo:equiv}
\label{transerg}
Let $F$ be any additive CA over a finite abelian group. The following statements are equivalent:
\begin{enumerate}
\item $F$ is topologically transitive;
\item $F$ is ergodic;
\item $F$ is surjective and for every $k\in\N$ it holds that $F^k-I$ is surjective;
\item $F$ is topologically mixing;
\item $F$ is weak topologically transitive;
\item $F$ is totally transitive;
\item $F$ is weakly ergodic mixing;
\item $F$ is ergodic mixing.
\end{enumerate}
Moreover, all the previously mentioned properties are decidable for LCA over $\K^n$.
\end{theorem}



\subsection{From Linear to Additive CA}
In this section we are going to prove that sensitivity, equicontinuity, injectivity, surjectivity, topological transitivity, and all the properties equivalent to the latter are decidable also for additive CA over a finite abelian group. For each of them we will reach the decidability result by exploiting the analogous one obtained for LCA 
and extending it to the wide class of additive CA over a finite abelian group. 

We recall that the local rule $\delta: G^{2r+1}\to G$ of an additive CA of radius $r$ over a finite abelian group $G$ can be written as
\begin{equation}
\label{decomp}
\forall (x_{-r}, \ldots, x_{r})\in G^{2r+1}, \qquad \delta(x_{-r},\dots,x_r)=\sum_{i=-r}^{r} \delta_i(x_i)
\end{equation}
where the functions $\delta_i$ are endomorphisms of $G$.

The fundamental theorem of finite abelian groups states 
that every finite abelian group $G$ is isomorphic to
$\bigoplus_{i=1}^{h} \zetaki$ where
the numbers $k_1,\dots, k_h$ are powers of (not necessarily distinct) 
primes and $\oplus$ is the direct sum operation. 
Hence, the global rule  $F$ of an additive CA over $G$ splits into the 
direct sum of a suitable number $h'$ of additive CA 
over subgroups $G_1,\dots,G_{h'}$ with $h' \leq h$ and such that 
$\gcd(|G_i|,|G_j|)=1$ for each pair of distinct $i,j\in\{1, \ldots, h'\}$.
Each of them can be studied separately and then the analysis of the 
dynamical behavior of $F$ can be  carried out by combining together 
the results obtained for each component. 

In order to make things clearer, consider the following example.
If $F$ is an additive CA over  $G \cong \Z/4\Z \times \Z/8\Z \times \Z/3\Z \times \Z/3\Z \times \Z/{25}\Z$ then $F$  splits 
into the direct sum of 3 additive CA $F_1$, $F_2$, and $F_3$ over $\Z/4\Z \times \Z/8\Z$, $\Z/3\Z \times \Z/3\Z$ and $ \Z/{25}\Z$, respectively.
Therefore, $F$ will be sensitive to initial conditions (resp., topologically transitive)
iff at least one (resp., each) $F_i$ is sensitive to initial conditions.

The above considerations lead us to three distinct scenarios:\\[-3mm]
\begin{description}
\item[1) $G\cong\zetapk$.]
Then, $G$ is cyclic and we can define each $\delta_i$ simply 
assigning the value of  $\delta_i$ applied to the unique generator of $G$.  
Moreover, every pair $\delta_i, \delta_j$ commutes, i.e., $\delta_i \circ \delta_j = \delta_j \circ \delta_i$, and this 
makes it possible a detailed analysis of the global behavior of $F$.
Indeed, additive cellular automata over $\zetapk$ are nothing but LCA over $\zetapk$ and 
almost all dynamical properties, including sensitivity to the initial conditions, equicontinutity, injectivity, surjectivity, topological transitivity and so on are well understood and characterized (see~\cite{ManMan99}). \\[-3mm]
\item[2) $G\cong (\zetapk)^n$.] In this case, $G$ is not cyclic anymore and has $n$ generators.
We can define each $\delta_i$ assigning the value of $\delta_i$ for each generator of $G$. 
This gives rise to the class of linear CA over $(\zetapk)^n$.
Now, $\delta_i$ and $\delta_j$ do not commute in general and this makes the analysis of the dynamical behavior much harder.  Nevertheless, in Section~\ref{sec:LCA} we have proved that sensitivity and equicontinuity are decidable by exploiting the main result of Section~\ref{chp.finpowmat}. As pointed out in~\cite{Dennun19}, we also recall that linear CA over $(\zetapk)^n$ allow the investigation of some classes of non-uniform CA over $\zetapk$ (
(for these latter see
~\cite{Cattan09,Dennun12,Dennun13}
). 
\\[-3mm]
\item[3) $G\cong \bigoplus_{i=1}^{n} \zetapki$.] In this case ($\Z/4\Z \times \Z/8\Z$ in the example), $G$ is again 
not cyclic and 
turns out to be a subsystem 
of a suitable LCA. 
Then, the analysis of the dynamical behavior of $F$ is even more complex than in \textbf{2)}. 
We do not even know easy checkable characterizations of basic properties like surjectivity or injectivity so far. We will provide them in the sequel as we stated at the beginning of this section.
\end{description}
Therefore,  WLOG in the sequel we can assume that $G=\zetapku\times\ldots\times \zetapkn$ with $k_1\geq k_2\geq \ldots\geq k_n$ in order to reach our goal.
\medskip

For any $i\in\{1, \ldots, n\}$ let us denote by $\eee^{(i)}\in G^{\Z}$ the bi-infinite configuration such that $\eee_0^{(i)}=e_i$ and $\eee_j^{(i)}=0$ for every integer $j\neq 0$.
\begin{definition}
Let $(G^{\Z},F)$ be an additive CA over $G$. We say that $\eee^{(i)}\in G^{\Z}$ \emph{spreads under $F$} if for every $\ell\in\N$ there exists $k\in\N$ such that $F^k(\eee^{(i)})_j\neq 0$ for some integer $j$ with $|j|>\ell$. 
\end{definition}
\begin{remark}
Whenever we consider $\P_{\eee^{(i)}}(X)\in\LPG$, 
we will say that $\P_{\eee^{(i)}}(X)$ spreads under $F$ if for every $\ell\in\N$ there exists $k\in\N$ such that $\P_{F^k(\eee^{(i)})}(X)$ has at least one component with a non null monomial of degree which is greater than $\ell$ in absolute value. Clearly, $\P_{\eee^{(i)}}(X)$ spreads under $F$ if and only if $\eee^{(i)}$ spreads under $F$.
\end{remark}

Let $\hat{G}=\zetapku\times\ldots\times \zetapku$. Define the map $\emb: G\to \hat{G}$ as follows
$$
\forall i=1,\dots,n, \quad \forall h\in G, \quad \emb(h)^i=h^i\, p^{k_1-k_i}\enspace.
$$
\begin{definition}
We define the function $\Emb: G^{\Z}\to \hat{G}^{\Z}$ as the component-wise extension of $\emb$, \ie, 
$$
\forall \ccc\in G^{\Z}, \quad \forall j\in\Z, \quad \Emb(\ccc)_j=\emb(\ccc_j)\enspace.
$$
\end{definition}
It is easy to check that $\Emb$ is continuous and injective, but not surjective. Since every configuration $\ccc\in G^{\Z}$ (or $\hat{G}^{\Z}$) is associated with the fps $\P_{\ccc}(X)\in\LPG$ (or $\LPGt$), with an abuse of notation we will sometimes consider $\Emb$ as map from $\LPG$ to $\LPGt$ with the obvious meaning.

\medskip
For any additive CA over $G$, we are now going to define a LCA over $(\zetapku)^n$ associated to it. With a further abuse of notation, in the sequel we will write $p^{-m}$ with $m\in\N$ even if this quantity might not exist in $\zetapk$. However, we will use it only when  it multiplies $p^{m'}$ for some integer $m'>m$. In such a way $p^{m'-m}$ is well defined in $\zetapk$ and we will note it as product $p^{-m} \cdot p^{m'}$.  

\begin{definition}
Let $(G^{\Z},F)$ be any additive CA and let $\delta: G^{2r+1}\to G$ be its local rule defined, according to~\eqref{decomp}, by $2r+1$ endomorphisms $\delta_{-r}, \ldots, \delta_{r}$ of $G$ . For each $z\in\{-r, \ldots, r\}$, we define the matrix $A_z = (a^{(z)}_{i,j})_{1\leq i\leq n,\, 1\leq j\leq n}\in (\zetapku)^{n\times n}$ as 
\[
\forall i,j\in\{1, \ldots, n\}, \qquad a^{(z)}_{i,j}=p^{k_j-k_i} \cdot \delta_z(e_j)^i
\]
The \emph{LCA associated with the additive CA $(G^{\Z},F)$} is $(\hat{G}^\Z, L)$, where $L$ is defined by $A_{-r}, \ldots, A_r$ or, equivalently, by $A=\sum_{z=-r}^r A_z X^{-z}\in\LPGt^{n\times n}$.
\end{definition}
\begin{remark}
Since every $\delta_z$ is an endomorphism of $G$, by construction $A$ turns out to be well defined.
\end{remark}
\begin{remark}
\label{comm}
The following diagram commutes
\[
\begin{CD}
   G^{\Z}@>F>>&G^{\Z}\\
   @V{\Emb}VV&@VV{\Emb}V\\
   \hat{G}^{\Z}@>>L>&\hat{G}^{\Z}
\end{CD}\enspace, 
\]
\ie, $L\circ \Emb=\Emb\circ F$. This is the reason why we also say that $(\hat{G}^\Z, L)$ is the LCA associated with $(G^{\Z},F)$ \emph{via the embedding $\Emb$}. 
\end{remark}
\subsubsection{Sensitivity and Equicontinuity for Additive Cellular Automata}
Let us start with the decidability of sensitivity and equicontinuity.
\begin{lemma}\label{spredda}
Let $(G^{\Z},F)$ be any additive CA. If for some $i\in\{1, \ldots, n\}$ the configuration $\eee^{(i)}\in G^{\Z}$ spreads under $F$ then $(G^{\Z},F)$ is sensitive to the initial conditions.
\end{lemma}
\begin{proof}
We prove that $F$ is sensitive with constant $\varepsilon=1$. Let $\eee^{(i)}\in G^{\Z}$ be the configuration spreading under $F$. Choose arbitrarily an integer $\ell\in\N$ and a configuration $\ccc\in G^{\Z}$. Let $t\in\N$ and $j\notin\{-\ell,\ldots, \ell\}$  be the integers such that $F^t(\eee^{(i)})_j\neq 0$. Consider the configuration $\ccc'=\ccc+\sigma^j(\eee^{(i)})$. Clearly, it holds that $d(\ccc, \ccc')<2^{-\ell}$ and $F^t(\ccc')=F^t(\ccc)+ F^t(\sigma^j(\eee^{(i)}))=F^t(\ccc)+ \sigma^j(F^t(\eee^{(i)}))$. So, we get $d(F^t(\ccc'), F^t(\ccc))=1$ and this concludes the proof. 
\end{proof}
In order to prove the decidability of sensitivity, we need to deal with the degree of a polynomial. We stress that the notions we are going to introduce apply to finite Laurent polynomials.
\begin{definition}
  \label{degree}
Given any finite polynomial $\pp(X)\in\LPKuno$, the \emph{positive} (resp., \emph{negative}) \emph{degree of $\pp(X)$}, denoted by  $deg^+[\pp(X)]$ (resp., $deg^-[\pp(X)]$) is the maximum (resp., minimum) degree among those of the monomials having both positive (resp., negative) degree and coefficient which is not multiple of $p$. If there is no monomial satisfying both the required conditions, then $deg^+[\pp(X)]=0$ (resp., $deg^-[\pp(X)]$=0).
\end{definition}

\begin{lemma}
\label{explosion}
Let $(\hat{G}^\Z, L)$ be a LCA and let $A\in\LPKuno^{n\times n}$ be the matrix associated to it. If $(\hat{G}^\Z, L)$ is sensitive then for every integer $m\geq 1$ there exists an integer $k\geq 1$ such that at least one entry of $A^k$ has either positive or negative degree with absolute value which is greater than $m$. 
\end{lemma}
\begin{proof}
We can write $A= B+p\cdot C$ for some $A,B\in\LPKuno^{n\times n}$, where the monomials of all entries of $A$ have coefficient which is not multiple of $p$. Assume that there exists a bound $b\geq 1$ such that for every $k \geq 1$ all entries of $A^k$ have degree less than $b$ in absolute value.  Therefore, it holds that $\left | \{A^k, k\geq 1 \} \right | < \infty$ and so $(\hat{G}^\Z, L)$ is not sensitive.
\end{proof}
We are now able to prove the following important result.
\begin{theorem}
\label{teo:sens}
Let $(G^{\Z},F)$ be any additive CA over $G$ and let $(\hat{G}^\Z, L)$ be the LCA associated to it via the embedding $\Emb$. Then, the CA $(G^{\Z},F)$ is sensitive to the initial conditions if and only if $(\hat{G}^\Z, L)$ is.
\end{theorem}
\begin{proof}
$\Longrightarrow:$ Assume that $(\hat{G}^\Z, L)$ is not sensitive. Then, by Proposition~\ref{prop:dich}, there exist two integers $k\in\N$ and $m>0$ such that $L^{k+m}=L^k$. Therefore, we get $\Emb\circ F^{k+m}= L^{k+m}\circ  \Emb=L^k\circ \Emb= \Emb \circ F^k$. Since $\Emb$ is injective, it holds that $F^{k+m}=F^k$ and so $(G^{\Z},F)$ is not sensitive.\\
$\Longleftarrow:$ Assume that $(\hat{G}^\Z, L)$ is sensitive and for any natural $k$ let $A^{k}= (a^{(k)}_{i,j})_{1\leq i\leq n,\, 1\leq j\leq n}$  be the $k$-th power of $A\in \LPKuno^{n\times n}$, where $A$ is the matrix associated to $(\hat{G}^\Z, L)$. We are going to show that at least one configuration among $\eee^{(1)}, \ldots, \eee^{(n)}$ spreads under $F$.  
Choose arbitrarily $\ell\in\N$. By Lemma~\ref{explosion}, there exist an integer $m\geq 1$ and one entry $(i,j)$ such that either $deg^-[a^{(m)}_{i,j}]<-\ell$ or $deg^+[a^{(m)}_{i,j}]>\ell$.  WLOG suppose that $deg^+[a^{(m)}_{i,j}]>\ell$. The $i$--th component of $\P_{F^m(\eee^{(j)})}(X)$ is the well defined polynomial $p^{k_i-k_1}\cdot p^{k_1-k_j} \cdot a^{(m)}_{i,j}$. Since $deg^+[a^{(m)}_{i,j}]>\ell$, we can state that $\eee^{(j)}$ spreads under $F$.  By Lemma~\ref{spredda}, it follows that $(G^{\Z},F)$ is sensitive.
%
\end{proof}
As immediate consequence of Theorem~\ref{teo:sens} we can state that the dichotomy between sensitivity and equicontinuity also holds for additive CA.
\begin{corollary}
Any additive CA over a finite abelian group is sensitive to the initial conditions if and only if it is not equicontinuous.
\end{corollary}
The following decidability result follows from Theorem~\ref{teo:sens} and the decidability of sensitivity for LCA. 
\begin{corollary}
Equicontinuity and sensitivity to the initial conditions are  decidable for additive CA over a finite abelian group.
\end{corollary}
\begin{proof}
Use Theorem~\ref{declca} and ~\ref{teo:sens}.
\end{proof}
\subsubsection{Surjectivity and Injectivity for Additive Cellular Automata}
We now study injectivity and surjectivity for additive CA. 
\begin{lemma}
\label{immagine}
Let $(\hat{G}^\Z, L)$ be any LCA over $\hat{G}$. For any configuration $\bbb\in\hat{G}^\Z$ with $\bbb\neq 0$ and $L(\bbb)=0$ there exists a configuration $\bbb'\in\Psi({G}^\Z)$ such that $\bbb'\neq 0$ and $L(\bbb')=0$. In particular, if $\bbb$ is finite then $\bbb'$ is too.
\end{lemma}
\begin{proof}
Let $\bbb\in\hat{G}^\Z$ any configuration with $\bbb\neq 0$ and $L(\bbb)=0$. 
Set $\bbb^{(1)}=p \cdot \bbb$. If $\bbb^{(1)}=0$ then for every $i\in\Z$ each component of $\bbb_i$ has $p^{k_1-1}$ as factor. So, $\bbb\in\Psi({G}^\Z)$ and $\bbb'=\bbb$ is just one possible configuration the thesis requires to exhibit. Otherwise, by repeating the same argument, set  $\bbb^{(2)}=p \cdot \bbb^{(1)}$. If $\bbb^{(2)}=0$ then, for every $i\in\Z$, each component  of $\bbb^{(1)}_i$ has $p^{k_1-1}$ as factor and so $\bbb^{(1)}\in\Psi({G}^\Z)$. Since $L(\bbb^{(1)})=0$, a configuration we are looking for is $\bbb'=\bbb^{(1)}$. After $k-1$ iterations,  \ie, once we get $\bbb^{(k_1-1)}=p \cdot \bbb^{(k-2)}$ (with $\bbb^{(k-2)}\neq 0$), if $\bbb^{(k_1-1)}=0$ holds we conclude that $\bbb'=\bbb^{(k-2)}$ by using  the same argument of the previous steps.  Otherwise, by definition, for every $i\in\Z$ each component of $\bbb^{(k_1-1)}_i$ itself certainly contains $p^{k_1-1}$ as factor.  Therefore,  $\bbb^{(k_1-1)}\in\Psi({G}^\Z)$. Moreover, $L(\bbb^{(k_1-1)})=0$. Hence, we can set $\bbb'=\bbb^{(k_1-1)}$ and this concludes the proof.
\end{proof}
The following lemma will be useful for studying both surjectivity and other properties.
\begin{lemma}
\label{lem:surj}
Let $(G^{\Z},F)$ and $(\hat{G}^\Z, L)$ be any additive CA over $G$ and any LCA over $\hat{G}$, respectively, such that $L\circ\Emb=\Emb\circ F$. Then, the CA $(G^{\Z},F)$ is surjective if and only if $(\hat{G}^\Z, L)$ is.
\end{lemma}
\begin{proof}
$\Longleftarrow:$ Assume that $F$ is not surjective. Then, $F$ is not injective on the finite configurations, \ie, there exist two distinct and finite $\ccc',\ccc''\in G^{\Z}$ with $F(\ccc')=F(\ccc'')$. Therefore, the element $\ccc=\ccc'-\ccc''\in G^{\Z}$ is a finite configuration such that $\ccc\neq 0$ and $F(\ccc)=0$. So, we get both $\Emb(\ccc)\neq 0$ and $L(\Emb(\ccc))=\Emb(F(\ccc))=0$. Since $\Emb(\ccc)\neq 0$, it follows that $L$ is not surjective.
\\
$\Longrightarrow:$ 
Assume that $L$ is not surjective. Then, it is not injective on the finite configurations. Thus, there exist a finite configuration $\bbb\neq 0$ with $L(\bbb)=0$. By Lemma~\ref{immagine}, there exists a finite configuration $\bbb'\in \Psi(G^{\Z})$ such that $\bbb'\neq 0$ and $L(\bbb')=0$. Let $\ccc\in G^{\Z}$ be the finite configuration such that $\Emb(\ccc)=\bbb'$. Clearly, it holds that  $\ccc\neq 0$. We get $\Emb(F(\ccc))=L(\Emb(\ccc))=0$. Since $\Emb$ is injective, it follows that $F(\ccc)=0$. Therefore, we conclude that $F$ is not surjective.
\end{proof}
Next two theorems state that surjectivity and injectivity behave as sensitivity when looking at an additive CA over $G$ and the associated LCA via the embedding $\Emb$.
\begin{theorem}
\label{teo:surj}
Let $(G^{\Z},F)$ be any additive CA over $G$ and let $(\hat{G}^\Z, L)$ be the LCA associated with it via the embedding $\Emb$. Then, the CA $(G^{\Z},F)$ is surjective if and only if $(\hat{G}^\Z, L)$ is.
\end{theorem}
\begin{proof}
Use Lemma~\ref{lem:surj}.
\end{proof}
\begin{theorem}
\label{teo:inj}
Let $(G^{\Z},F)$ be any additive CA and let $(\hat{G}^\Z, L)$ be the LCA associated with it via the embedding $\Emb$. Then, the CA $(G^{\Z},F)$ is injective if and only if $(\hat{G}^\Z, L)$ is.
\end{theorem}
\begin{proof}
$\Longleftarrow:$ Assume that $F$ is not injective. Then, there exist two distinct configurations $\ccc,\ccc'\in G^{\Z}$ with $F(\ccc)=F(\ccc')$. We get $L(\Emb(\ccc))=\Emb(F(\ccc))=\Emb(F(\ccc'))= L(\Emb(\ccc'))$ and, since $\Emb$ is injective,  it follows that $L$ is not injective. \\
$\Longrightarrow:$ Assume that $L$ is not injective. Then, there exists a configuration $\bbb\in\hat{G}^\Z$ such that  $\bbb\neq 0$ and $L(\bbb)=0$. By Lemma~\ref{immagine}, there exists a finite configuration $\bbb'\in \Psi(G^{\Z})$ such that $\bbb'\neq 0$ and $L(\bbb')=0$. Let $\ccc\in G^{\Z}$ be the finite configuration such that $\Emb(\ccc)=\bbb'$. Clearly, it holds that  $\ccc\neq 0$. We get $\Emb(F(\ccc))=L(\Emb(\ccc))=0$. Since $\Emb$ is injective, it follows that $F(\ccc)=0$. Since $F(0)=0$, we conclude that $F$ is not injective.
\end{proof}
The decidability of injectivity and surjectivity for LCA combined with the previous two theorems lead to the corresponding decidability result for additive CA.
\begin{corollary}
Surjectivity and injectivity are decidable for additive CA over a finite abelian group.
\end{corollary}
\begin{proof}
It immediately follows from Theorem~\ref{decinjsurjlca},~\ref{teo:inj}, and~\ref{teo:inj}.
\end{proof}
\subsubsection{Topological transitivity and ergodicity}
We start by proving that the embedding $\Emb$ also preserves topological transitivity between an additive CA over $G$ and the associated LCA.
\begin{theorem}
\label{teo:trans}
Let $(G^{\Z},F)$ be any additive CA over $G$ and let $(\hat{G}^\Z, L)$ be the LCA associated with it via the embedding $\Emb$. Then, the CA $(G^{\Z},F)$ is topologically transitive if and only if $(\hat{G}^\Z, L)$ is.
\end{theorem}
\begin{proof}
Since  $\Emb\circ F=L\circ\Emb$, for every $k\in\N$ it holds that $\Emb\circ (F^k-I)=\Emb\circ F^k - \Emb= L^k\circ\Emb - \Emb= (L^k-I)\circ \Emb$. By Lemma~\ref{lem:surj}
, $F^k-I$ is surjective iff  $L^k-I$ is. Theorem~\ref{teo:surj} and~\ref{transerg} conclude the proof.
\end{proof}
As a final result, we get the decidability of many mixing and ergodic properties for additive CA over any finite abelian group, including topological transitivity and ergodicity. 
\begin{corollary}
All the following properties are decidable for additive CA over any finite abelian group:
\begin{enumerate}
\item topological transitivity;
\item ergodicity;
\item  topological mixing;
\item weak topological transitivity;
\item total transitivity;
\item weak ergodic mixing;
\item ergodic mixing.
\end{enumerate}
\end{corollary}
\begin{proof}
It is an immediate consequence of Theorem~\ref{teo:equiv} and~\ref{teo:trans}.
\end{proof}

\end{document}